\documentclass[11pt]{amsart}
\usepackage{}
\usepackage{pifont}
\usepackage{mathtools}
\usepackage{esint}
\usepackage{cite}

\usepackage{amsmath}
\usepackage{amssymb}
\usepackage{latexsym}
\usepackage{amscd}
\usepackage{mathrsfs}
\usepackage[all]{xy}

\usepackage{hyperref}
\usepackage{graphicx} 
\usepackage{amsthm}
\usepackage{xypic}
\usepackage{bm}
\usepackage{color}

\usepackage{color}

\newdimen\AAdi%
\newbox\AAbo%
\def\AAk#1#2{\s_etbox\AAbo=\hbox{#2}\AAdi=\wd\AAbo\kern#1\AAdi{}}%
\def\AAr#1#2#3{\s_etbox\AAbo=\hbox{#2}\AAdi=\ht\AAbo\raise#1\AAdi\hbox{#3}}%
\font\tenmsb=msbm10 at 12pt \font\sevenmsb=msbm7 at 8pt
\font\fivemsb=msbm5 at 6pt
\newfam\msbfam
\textfont\msbfam=\tenmsb \scriptfont\msbfam=\sevenmsb
\scriptscriptfont\msbfam=\fivemsb

\newtheorem{theorem}{Theorem}

\newtheorem{remark}[theorem]{Remark}

\newtheorem{corollary}[theorem]{Corollary}

\newtheorem{lemma}[theorem]{Lemma}
\newtheorem{proposition}[theorem]{Proposition}

\numberwithin{equation}{section} \numberwithin{theorem}{section}

\renewcommand{\topmargin}{0cm}
\renewcommand{\oddsidemargin}{5mm}
\renewcommand{\evensidemargin}{5mm}
\renewcommand{\textwidth}{150mm}
\renewcommand{\textheight}{230mm}

\def\R{\mathbb R}

\def\Z{\mathbb Z}

\def\Z{\mathbb Z}

\def\n{\mathbf n}

\def\na{\nabla}
\def\bn{\overline\nabla}

\def\f#1#2{\frac{#1}{#2}}

\def\a{\alpha}
\def\be{\beta}

\def\r{\Re_{I\!V}}

\def\p#1{\partial #1}

\def\de{\delta}
\def\De{\Delta}
\def\e{\eta}
\def\ep{\epsilon}

\def\G{\Gamma}
\def\g{\gamma}

\def\la{\lambda}
\def\La{\Lambda}
\def\lan{\langle}
\def\ran{\rangle}

\def\Om{\Omega}
\def\th{\theta}
\def\Th{\Theta}
\def\si{\sigma}
\def\Si{\Sigma}

\def\r{\rho}

\begin{document}

\title
[Capacity for minimal graphs and the half-space property]
{Capacity for minimal graphs over manifolds and the half-space property}

\author{Qi Ding}
\address{Shanghai Center for Mathematical Sciences, Fudan University, Shanghai 200438, China}
\email{dingqi@fudan.edu.cn}

\begin{abstract}
In this paper, we define natural capacities using a relative volume of graphs over manifolds, which can be characterized by solutions of bounded variation to Dirichlet problems of minimal hypersurface equation.
Using the capacities, we introduce a notion '$M$-parabolicity' for ends of complete manifolds, where a parabolic end  
must be $M$-parabolic, but not vice versa in general. We study the boundary behavior of solutions associated with capacities in the measure sense, and the existence of minimal graphs over $M$-parabolic or $M$-nonparabolic manifolds outside compact sets.

For a $M$-parabolic manifold $P$, we prove a half-space theorem for complete proper minimal hypersurfaces in $P\times\R$. 
As a corollary, we immediately have a slice theorem for smooth mean concave domains in $P\times\R^+$, where the $M$-parabolic condition is sharp by our example. On the other hand, we prove that any $M$-parabolic end is indeed parabolic provided its Ricci curvature is uniformly bounded from below. Compared to harmonic functions, we get the asymptotic estimates with sharp orders for minimal graphic functions on nonparabolic manifolds of nonnegative Ricci curvature outside compact sets.

\end{abstract}

\maketitle
\tableofcontents

\section{Introduction}

Let $(\Si,\si)$ be an $n$-dimensional complete Riemannian manifold with the Levi-Civita connection $D$. Let $\Si\times\R$ denote the product manifold, which has the standard product metric $\si+dt^2$.
For an open set $\Om\subset\Si$, a closed set $K\subset\Si$ with $K\subset\Om$ and compact $\p K$, and a constant $t\ge0$, we define a capacity via a relative volume of graphs by
\begin{equation}\aligned\label{captKOm}
\mathrm{cap}_t(K,\Om)=\inf_{\phi\in\mathcal{L}_t(K,\Om)}\int_{\Om}\left(\sqrt{1+|D\phi|^2}-1\right).
\endaligned\end{equation}
Here, $\mathcal{L}_t(K,\Om)$ denotes the set containing every locally Lipschitz function $\phi$ on $\Si$ with compact $\overline{\mathrm{spt}\phi\setminus K}$ in $\Om$ such that $0\le\phi\le t$ and $\phi\big|_{K}=t$, where $\mathrm{spt}\phi$ denotes the support of $\phi$ in $\Si$. For simplicity, we denote
$\mathrm{cap}_t(K)=\mathrm{cap}_t(K,\Si)$.

The classical capacity (corresponding to harmonic functions) is given by
\begin{equation}\aligned\label{cap111KOm}
\mathrm{cap}(K,\Om)=\inf_{\phi\in\mathcal{L}_1(K,\Om)}\int_{\Om}|D\phi|^2.
\endaligned\end{equation}
Denote $\mathrm{cap}(K)=\mathrm{cap}(K,\Si)$.
Let an open set $E$ be an end of $\Si$. If $\mathrm{cap}(\Si\setminus E)=0$, then $E$ is said to be \emph{parabolic}. If $\mathrm{cap}(K)=0$ for some compact $K\subset\Si$ with the $n$-dimensional Hausdorff measure $\mathcal{H}^n(K)>0$, then $\Si$ is said to be \emph{parabolic}. 
Otherwise, $\Si$ is said to be \emph{nonparabolic}.

If we further assume that $\Om$ is bounded, and $\Om\setminus K$ has Lipschitz boundary, then
from $\S 14$ of Giusti \cite{Gi}, the capacity \eqref{captKOm} is attained by a BV solution $u$ of Dirichlet problems of the following minimal hypersurface equation  
\begin{equation}\label{u}
\mathrm{div}_\Si\left(\f{Du}{\sqrt{1+|Du|^2}}\right)=0
\end{equation}
on $\Om\setminus K$ with boundary data $u|_{\p K}=t$ and $u|_{\p\Om}=0$ (see \eqref{capttru} for more details),
where div$_{\Si}$ denotes the divergence on $\Si$. 
We call that $u$ is a (BV) solution on $\Si$ associated with $\mathrm{cap}_t(K,\Om)$.
Moreover,  the boundary of the subgraph of $u$ in $\overline{\Om\setminus K}\times\R$ is an area-minimizing hypersurface in $\overline{\Om\setminus K}\times\R$ with the boundary $(\p K\times\{t\})\cup(\p\Om\times\{0\})$ in rough speaking (see p. 75 in Lin' thesis \cite{Lin0}).
Since the boundary data here is constant (simpler than $\S14$ in \cite{Gi}), we are able to further study them in $\S4$ in the framework of geometric measure theory with the Lipschitz condition of $\p(\Om\setminus K)$ replaced by $\p\,\overline{\Om\setminus K}=\p(\Om\setminus K)$.

If there are a constant $t>0$ and a compact set $F\subset\Si$ with $\mathcal{H}^n(F)>0$ so that $\mathrm{cap}_t(F)=0$, then we call that $\Si$ is \emph{M-parabolic} as the capacity \eqref{captKOm} connects to minimal hypersurfaces. Otherwise, we call that $\Si$ is \emph{M-nonparabolic}. Moreover, let an open set $E$ be an end of $\Si$. The end $E$ is said to be \emph{M-parabolic} provided $\mathrm{cap}_t(\Si\setminus E)=0$ for some $t>0$.
Clearly, a closed manifold without boundary must be $M$-parabolic.
By the definitions, a complete parabolic manifold is $M$-parabolic automatically. However, the inverse is not true in general. For example, there is a rotationally symmetric complete $M$-parabolic manifold, but it is not parabolic (see Proposition \ref{M-par-not-par}).

We have the following equivalence for $M$-nonparabolic manifolds.
\begin{theorem}\label{main}
Let $\Si$ be a complete Riemannian manifold. The following properties are
equivalent.
  \begin{enumerate}
  \item[(1)] There are a constant $t>0$ and a compact set $K$ in $\Si$ so that $\mathrm{cap}_t(K)>0$.
  \item[(2)] For any $t>0$ and any compact set $K$ in $\Si$ with $\mathcal{H}^n(K)>0$, there holds $\mathrm{cap}_t(K)>0$.
  \item[(3)] For any compact set $K$ in $\Si$ with $\mathcal{H}^n(K)>0$, there is a smooth positive non-constant bounded solution $u$ to \eqref{u} on $\Si\setminus K$.
  \item[(4)] There is a non-flat smooth bounded graph over $\Si$ in $\Si\times\R$ with nonnegative mean curvature.
\end{enumerate}
\end{theorem}
In $\S2$, we will investigate the basic properties of the capacity \eqref{captKOm}, including that (1) implies (2). The rest proof of  the equivalence in Theorem \ref{main} will be given in $\S7$.

Let us review the history briefly on the 'half-space property'.
In 1990, Hoffman-Meeks \cite{HM} proved a famous half-space theorem. It asserts that any complete proper immersed minimal surface in $\R^2\times\R^+$ is $\R^2\times\{c\}$ for some constant $c\in\R^+=(0,\infty)$. 
Let $\Si$ be a complete Riemannian manifold, and $M$ be a complete minimal hypersurface properly immersed in $\Si\times\R^+$. We say that $\Si$ has the \emph{half-space property} if $M$ must be a slice $\Si\times\{c\}$ for some constant $c\in\R^+$.
In 2013,
Rosenberg-Schulze-Spruck \cite{RSS} proved an interesting result that a recurrent manifold with bounded sectional curvature has the half-space property. Here, the recurrence is equivalent to the parabolicity (see Grigor'yan \cite{G}).
Recently, Colombo-Magliaro-Mari-Rigoli \cite{CMMR} proved that a complete parabolic manifold with Ricci curvature bounded below has the half-space property. 

The half-space property can be seen as a special case of Frankel-type theorems \cite{Fr} for minimal hypersurfaces in manifolds. This has a close relation to the maximum principle at infinity for minimal hypersurfaces that has been studied in \cite{LR,MR1,ER} and related references therein.
In particular, Espinar-Rosenberg \cite{ER} proved the Maximum Principle at Infinity and Tubular Neighborhood Theorem for parabolic properly embedded minimal hypersurfaces in complete manifolds under some natural assumptions.

Using the capacity in \eqref{captKOm}, we can show that a $M$-parabolic manifold has the half-space property as follows, which infers directly that all complete parabolic manifolds have the half-space property.
\begin{theorem}\label{HSp000}
For a $M$-parabolic manifold $P$, any complete minimal hypersurface properly immersed in $P\times\R^+$, must be a slice $P\times\{c\}$ for some constant $c\in\R^+$. 
\end{theorem}

In \cite{CM2}, Colding-Minicozzi proved that a complete embedded minimal surface with finite topology in $\R^3$ must be proper, which confirmed a Calabi's conjecture (see also \cite{CM}). Meeks-Rosenberg \cite{MR} generalized it to get a local version in Riemannian 3-manifold. 
Hence, Theorem \ref{HSp000} holds provided the properness condition is replaced by embeddedness and finite topology as well as that $P$ has dimension 2.

The proof of Theorem \ref{HSp000} will be given in $\S5$. 
In \cite{IPS}, Impera-Pigola-Setti studied slice theorems for hypersurfaces of nonnegative mean curvature in product manifolds like $N\times[0,\infty)$.
From the proof of Theorem \ref{HSp000}, we immediately have the following corollory.
\begin{corollary}\label{HSpMC000}
For a $M$-parabolic manifold $P$, any smooth mean concave domain in $P\times[0,\infty)$ must be $P\times(c_1,c_2)$ for constants $0\le c_1< c_2\le\infty$.
\end{corollary}
Here, a smooth mean concave domain $U$ means that $U$ is a connected open set, and $\p U$ is smooth with nonnegative mean curvature w.r.t. the normal vector pointing out $U$. Our $M$-parabolic condition in Corollary \ref{HSpMC000} is sharp from (4) of Theorem \ref{main}.

Let $u$ be a BV solution on $\Si$ associated with $\mathrm{cap}_t(K,\Om)$,
and $M$ be the boundary of the subgraph of $u$ in $\overline{\Om\setminus K}\times\R$. The boundary gradient of $u$ may be infinity even for the smooth $\p\Om\cup\p K$ unlike harmonic functions w.r.t. the capacity \eqref{cap111KOm}.
In $\S6$, we will study the co-normal vector to $\p M$ in the measure sense using Riesz representation theorem without the smooth condition on $\p\Om\cup\p K$, which is useful in estimations of the capacities and solutions associated with them.

In $\S7$, we will study the existence of minimal graphs or graphs of nonnegative mean curvature over manifolds. From Theorem 6.1 in \cite{DJX} and Corollary \ref{RicMparpar} below, entire minimal graphs may be not exist on manifolds no matter whether they are $M$-nonparabolic or not. 
On the one hand, for a constant $t>0$ and a $M$-nonparabolic manifold $\Si$ with a ball $B_\de(p)\subset\Si$, we construct a smooth entire graph $M\subset\Si\times[0,t]$ over $\Si$ so that $M$ has nonnegative mean curvature and $M\setminus(B_\de(p)\times\R)$ is minimal (see Theorem \ref{NONMeanCurv}). If $\Si$ has two $M$-nonparabolic ends at least, then for any $t>0$ we have the existence of entire minimal graphs over $\Si$ in $\Si\times[0,t]$.
On the other hand, let $P$ be a $M$-parabolic manifold with \emph{nondegenerate boundary at infinity}, which means that there are a constant $\ep>0$ and a closed set $K_\ep\subset\Si$ so that $\mathcal{H}^{n-1}(\p U)\ge\ep$ for any open set $U\supset K_\ep$ with non-empty $\Si\setminus U$. We prove the existence of minimal graphs over $P$ outside compact sets (see Theorem \ref{EntiresolM-para}).
Here, the condition 'nondegenerate boundary at infinity' is necessary (see Proposition \ref{non-const BV solutions} in Appendix II).

Let $M$ be an $n$-dimensional complete non-compact manifold of nonnegative Ricci curvature. Varopoulos \cite{V} proved that $M$ is nonparabolic if and only if
$$\int_1^\infty\f{r dr}{\mathcal{H}^n(B_r(p))}<\infty.$$
Recall that a nonparabolic manifold admits positive Green functions on it.
Li-Yau \cite{LY} proved that the minimal positive Green function $G(x,p)$ on $M$ satisfies
\begin{equation}\aligned\label{Green}
c_n^{-1}\int_{d(x,p)}^\infty\f{r dr}{\mathcal{H}^n(B_r(p))}\le G(x,p)\le c_n\int_{d(x,p)}^\infty\f{r dr}{\mathcal{H}^n(B_r(p))},
\endaligned
\end{equation}
where $c_n\ge1$ is a constant depending only on $n$. If $n\ge3$ and $M$ further has Euclidean volume growth, then Colding-Minicozzi \cite{CM2} can give the sharp asymptotic estimate of Green functions at infinity (see also Li-Tam-Wang \cite{LTW}).

A $M$-parabolic manifold may be not parabolic in general. However, we can show that a $M$-parabolic manifold is indeed parabolic provided its Ricci curvature is uniformly bounded below (see Corollary \ref{RicMparpar}). Actually, we can prove this in a more general setting. More precisely, this property holds on manifolds satisfying the following
\eqref{VDg***}\eqref{NPBrxg***} for every $r\in(0,1]$ and every point $x$ in the considered manifold  (see Theorem \ref{SiVDNPpara}). The proof relies the convexity of volume functional of graphs, and a suitable average of functions on the manifold. 

Let $\Si$ be an $n$-dimensional complete non-compact nonparabolic manifold. We further suppose that $\Si$ satisfies volume doubling property
\begin{equation}\aligned\label{VDg***}
\mathcal{H}^n(B_{2r}(x))\le C_D\mathcal{H}^n(B_{r}(x))
\endaligned
\end{equation}
and (1,1)-Poincar\'e inequality
\begin{equation}\aligned\label{NPBrxg***}
\int_{B_r(x)}|f-\bar{f}_{x,r}|\le C_N r\int_{B_r(x)}|Df|
\endaligned
\end{equation}
for each $x\in\Si$ and $r>0$, whenever $f$ is a function in the Sobolev space $W^{1,1}(\Si)$ with $\bar{f}_{x,r}=\fint_{B_r(x)}f$,
Here, $C_D,C_N$ are positive constants.

Holopainen \cite{H} proved the estimate \eqref{Green} for Green functions on $\Si$ up to the constant $c_n$ repalced by a constant depending only on $C_D$ and $C_N$. Moreover, he generalized the estimate \eqref{Green} for Green functions to $q$-Green functions on $\Si$ for each $q>1$ (see \cite{H} for more results on general manifolds). 
Combining his result (the $q=2$ case) and Harnack's estimate in \cite{D}, we can deduce the following estimate(see Theorem \ref{BdUpLowercapu*}).
\begin{theorem}
Given a compact set $K\subset\Si$ and $t>0$, let $u$ be a BV function on $\Si$ associated with $\mathrm{cap}_t(K)$.
There is a constant $\Th\ge1$ depending only on $C_D,C_N$ such that for any $R\ge\max\{t,2\mathrm{diam}(K)\}$, $p\in K$ and $x\in \p B_R(p)$ there holds
\begin{equation}\aligned
\f{\mathrm{cap}_t(K)}{\Th t}\int_R^\infty\f{r dr}{\mathcal{H}^n(B_r(p))}\le u(x)\le\f{\Th\mathrm{cap}_t(K)}{t}\int_R^\infty\f{r dr}{\mathcal{H}^n(B_r(p))}.
\endaligned\end{equation}
\end{theorem}

$\mathbf{Notional\ convention}.$ For a point $x$ in a manifold $\Si$ defined in the paper, we let $d(x,\cdot)$ denote the distance function on $\Si$ from $x$, and $\r_V=d(\cdot,V)=\inf_{x\in V}d(\cdot,x)$ for any subset $V$ of $\Si$.
Let $B_r(x)$ denote the geodesic open ball in $\Si$ with radius $r>0$ and centered at $x$. For a subset $K\subset \Si$, 
let $B_r(K)=\{y\in\Si|\, d(y,K)<r\}$ denote $r$-tubular neighborhood of $K$ in $\Si$. Let $\overline{B}_r(x)$ denote the closure of $B_r(x)$, and $\overline{B}_r(K)$ denote the closure of $B_r(K)$.
For any two subsets $K_1,K_2\subset\Si$, $d_H(K_1,K_2)$ denotes the Hausdorff distance defined by
$d_H(K_1,K_2)=\inf\{s>0|\ K_1\subset B_s(K_2),\, K_2\subset B_s(K_1)\}.$
For a sequence of subsets $\{K_i\}_{i\ge1}\subset\Si$, $K_i\to K$ in the Hausdorff sense means $\lim_{i\to\infty}d_H(K_i,K)=0$.

For each integer $k>0$, let $\omega_k$ denote the volume of $k$-dimensional unit Euclidean ball, and $\mathcal{H}^k$ denote the $k$-dimensional Hausdorff measure.
When we write an integration on a subset of a Riemannian manifold w.r.t. some volume element, we always omit the volume element if it is associated with the Hausdorff measure of the subset with the standard metric of the given manifold. For a sequence $\{F_i\}_{i\ge1}$, we always write $F_i$ for short in the absence of ambiguity.

\section{Capacities defined by the relative volume of graphs}

In this section, we will study the basic properties of the capacity \eqref{captKOm}.
Let $\Si$ be an $n$-dimensional complete Riemannian manifold with Levi-Civita connection $D$. 
Let $\Om$ be an open set on $\Si$, and $K$ be a closed set in $\Si$ with $K\subset\Om$ and compact $\p K$. For any $t\ge0$, as before we let $\mathcal{L}_t(K,\Om)$ denote the set containing every locally Lipschitz function $\phi$ on $\Si$ with compact $\overline{\mathrm{spt}\phi\setminus K}$ in $\Om$
such that $0\le\phi\le t$ and $\phi\big|_{K}=t$.
From
$$\sqrt{1+|D\phi|^2}\le1+\f12|D\phi|^2,$$
we have
\begin{equation}\aligned\label{capt*KOm}
\mathrm{cap}_t(K,\Om)\le\f12\inf_{\mathcal{L}_t(K,\Om)}\int_{\Om}|D\phi|^2=\f{t^2}2\inf_{\mathcal{L}_1(K,\Om)}\int_{\Om}|D\phi|^2=\f{t^2}2\mathrm{cap}(K,\Om).
\endaligned
\end{equation}
For $0<t\le T$, we have
\begin{equation*}\aligned
\mathrm{cap}_T(K,\Om)=&\inf_{\mathcal{L}_t(K,\Om)}\int_{\Om}\left(\sqrt{1+\f{T^2}{t^2}|D\phi|^2}-1\right)=\inf_{\mathcal{L}_t(K,\Om)}\int_{\Om}\f{\f{T^2}{t^2}|D\phi|^2}{1+\sqrt{1+\f{T^2}{t^2}|D\phi|^2}}\\
\le&\f{T^2}{t^2}\inf_{\mathcal{L}_t(K,\Om)}\int_{\Om}\f{|D\phi|^2}{1+\sqrt{1+|D\phi|^2}}=\f{T^2}{t^2}\inf_{\mathcal{L}_t(K,\Om)}\int_{\Om}\left(\sqrt{1+|D\phi|^2}-1\right).
\endaligned\end{equation*}
Since the function $\sqrt{s^2+c^2}-s$ is monotonic non-increasing on $s>0$, it follows that
\begin{equation*}\aligned
\mathrm{cap}_t(K,\Om)=\inf_{\mathcal{L}_T(K,\Om)}\int_{\Om}\left(\sqrt{1+\f{t^2}{T^2}|D\phi|^2}-1\right)\le\f{t}{T}\inf_{\mathcal{L}_T(K,\Om)}\int_{\Om}\left(\sqrt{1+|D\phi|^2}-1\right).
\endaligned\end{equation*}
The above two inequalities give 
\begin{equation}\aligned\label{captTKOm}
\f{T}{t}\mathrm{cap}_t(K,\Om)\le\mathrm{cap}_T(K,\Om)\le\f{T^2}{t^2}\mathrm{cap}_t(K,\Om)\qquad\quad \mathrm{for\ all}\ 0<t\le T.
\endaligned\end{equation}
In particular, $\mathrm{cap}_t(K,\Om)$ is a continuous monotonic non-decreasing function on $t\in[0,\infty)$.

By the definition of $\mathrm{cap}_t(\cdot,\cdot)$, we have the following monotonicity:  
\begin{equation}\aligned\label{captmonot}
\mathrm{cap}_t(K,\Om)\le\mathrm{cap}_t\left(K',\Om'\right)
\endaligned\end{equation}
for any closed set $K'\supset K$ with compact $\p K'$ and any open $\Om'$ with $K'\subset\Om'\subset\Om$. 
For each fixed $t>0$, let us prove the semi-continuity of the capacity $\mathrm{cap}_t(\cdot,\cdot)$ in the Hausdorff topology as follows.
\begin{proposition}\label{semi-cont}
For any closed set $K\subset\Si$ with compact $\p K$, and any open set $\Om\supset K$ in $\Si$, let $K_i\subset\Si$ be a sequence of closed sets with $\p K_i\to\p K$ in the Hausdorff sense and $\Om_i\subset\Si$ be a sequence of open sets with $K_i\subset\Om_i\to \Om$ in the Hausdorff sense. Then
$\limsup_{i\to\infty}\mathrm{cap}_t\left(K_i,\Om_i\right)\le\mathrm{cap}_t(K,\Om)$.
\end{proposition}
\begin{proof}
For any $\de>0$, there is a locally Lipschitz function $\phi\in\mathcal{L}_t(K,\Om)$
such that 
\begin{equation}\aligned\label{OmKcaptde}
\int_{\Om\setminus K}\left(\sqrt{1+|D\phi|^2}-1\right)\le\mathrm{cap}_t(K,\Om)+\de.
\endaligned\end{equation}
Let $R>0$ be a constant so that $\overline{\mathrm{spt}\phi\setminus K}\subset B_R(p)$.
Hence, there is a constant $\La\ge2$ such that $|D\phi|\le\La$ on $B_1(K\cup\p\Om)$. 
Then we get
\begin{eqnarray}\label{phieppOmK}
    \left\{\begin{array}{cc}
           |\phi|\le2\La\ep   \quad\quad   \ \ \  \ \  \ \ \mathrm{on}&\ B_R(p)\cap B_{2\ep}(\p\Om) \\ [3mm]
           |t-\phi|\le2\La\ep      \quad\ \ \ \quad    \mathrm{on}&\ B_{2\ep}(K) 
     \end{array}\right.
\end{eqnarray}
for any $0<\ep\le\f12$ with $B_{2\ep}(K)\subset\Om$.

We define a Lipschitz function $\e_\ep$ on $B_{2\ep}(\Om)$ by letting $\e_\ep=1$ on $\Om\setminus B_{2\ep}(\p\Om\cup K)$, $\e_\ep=\f1\ep \r_{\p\Om\cup K}-1$ on $B_{2\ep}(\p\Om\cup K)\setminus B_{\ep}(\p\Om\cup K)$, and $\e_\ep=0$ on $B_{\ep}(\p\Om\cup K)$. Then with \eqref{phieppOmK} we have
\begin{equation}\aligned\label{pOmto0}
\int_{B_{2\ep}(\p\Om)}&\left(\sqrt{1+|D(\phi\e_\ep)|^2}-1\right)\le\int_{B_{2\ep}(\p\Om)}\f{|D(\phi\e_\ep)|^2}2\le\int_{B_{2\ep}(\p\Om)}\left(\e_\ep^2|D\phi|^2+\phi^2|D\e_\ep|^2\right)\\
\le&\int_{B_R(p)\cap B_{2\ep}(\p\Om)\cap\Om}\left(\La^2+4\La^2\right)\le 5\La^2\mathcal{H}^n\left(B_R(p)\cap B_{2\ep}(\p\Om)\cap\Om\right)
\endaligned\end{equation}
and
\begin{equation}\aligned\label{pKto0}
\int_{B_{2\ep}(K)}&\left(\sqrt{1+|D(t-(t-\phi)\e_\ep)|^2}-1\right)\le\int_{B_{2\ep}(K)}\left(\e_\ep^2|D\phi|^2+(t-\phi)^2|D\e_\ep|^2\right)\\
\le&\int_{B_{2\ep}(K)\setminus K}\left(\La^2+4\La^2\right)\le 5\La^2\mathcal{H}^n\left(B_{2\ep}(K)\setminus K\right).
\endaligned\end{equation}
Since the Hausdorff measure $\mathcal{H}^n$ is Borel-regular, it follows that
\begin{equation}\aligned
\lim_{s\to0}\mathcal{H}^n(B_s(F)\setminus F)=0\qquad \mathrm{for\ any\ compact\ set}\ F\subset\Si,
\endaligned\end{equation}
which implies 
\begin{equation}\aligned\label{pOmKto0}
\lim_{\ep\to0}\left(\mathcal{H}^n\left(B_R(p)\cap B_{2\ep}(\p\Om)\cap\Om\right)+\mathcal{H}^n\left(B_{2\ep}(K)\setminus K\right)\right)=0.
\endaligned\end{equation}
Let $\phi_\ep$ be a locally Lipschitz function on $\Si$ so that $\phi_\ep=\phi\e_\ep$ on $B_{2\ep}(\p\Om)$, $\phi_\ep=t-(t-\phi)\e_\ep$ on $B_{2\ep}(K)$, and $\phi_\ep=\phi$ on others. Then $\phi\in\mathcal{L}_t(K,\Om)$ implies $\phi_\ep\in\mathcal{L}_t(K,\Om)$. 
With \eqref{OmKcaptde}\eqref{pOmto0}\eqref{pKto0}\eqref{pOmKto0}, we get
\begin{equation}\aligned
\lim_{\ep\to0}\int_{\Om\setminus K}\left(\sqrt{1+|D\phi_\ep|^2}-1\right)\le\int_{\Om\setminus K}\left(\sqrt{1+|D\phi|^2}-1\right)\le\mathrm{cap}_t(K,\Om)+\de.
\endaligned\end{equation}
If $K_i\subset\Si$ is a sequence of closed sets with $\p K_i\to\p K$ in the Hausdorff sense and $\Om_i\subset\Si$ is a sequence of open sets with $K_i\subset\Om_i\to \Om$ in the Hausdorff sense, then
\begin{equation}\aligned
\limsup_{i\to\infty}\mathrm{cap}_t(K_i,\Om_i)\le\mathrm{cap}_t(K,\Om).
\endaligned\end{equation}
This completes the proof.
\end{proof}
As a corollary, we immediately have the following continuity in the sense of $\ep$-tubular neighborhood as $\ep\to0$.
\begin{corollary}\label{captKKiOmOmi***}
Let $K,K_i,\Om_i,\Om$ are the subsets of $\Si$ defined in Proposition \ref{semi-cont}. If $\Om_i\subset\Om$ and $K\subset K_i$ for each integer $i$, then
$\lim_{i\to\infty}\mathrm{cap}_t\left(K_i,\Om_i\right)=\mathrm{cap}_t(K,\Om)$.
In particular,
$\lim_{\ep\to0}\mathrm{cap}_t\left(\overline{B}_\ep(K),\Om\setminus \overline{B}_\ep(\p\Om)\right)=\mathrm{cap}_t(K,\Om)$.
\end{corollary}

\begin{remark}\label{capt0Ki}
In Proposition \ref{semi-cont}, even if the limit $\lim_{i\to\infty}\mathrm{cap}_t(K_i,\Om_i)$ exists, it still may be not equal to $\mathrm{cap}_t(K,\Om)$. 
For example, let  
$I_0=\{(s,0)\in\R^2|\ |s|\le1/2\}$ and $B$ be the open disk (with radius 1) in $\R^2$ centered at the origin. 
For each integer $i\ge1$, let $I_i=\{(\f{j}{2i},0)\in\R^2|\, |j|\le i,\, j\in\Z\}$, then $I_i\to I_0$ in the Hausdorff sense.  Given $t>0$, let $\phi_{i}$ be a Lipschitz function defined by $\phi_{i}(x)=t-i^{2}t\inf_{y\in I_i}|x-y|$ on $B_{i^{-2}}(I_i)$ and $\phi_{i}=0$ on $B\setminus B_{i^{-2}}(I_i)$. Then
\begin{equation*}\aligned
\int_{B\setminus I_i}\left(\sqrt{1+|D\phi_{i}|^2}-1\right)=\int_{B_{i^{-2}}(I_i)}\left(\sqrt{1+i^4t^2}-1\right)\le(2i+1)\pi i^{-4}i^2t,
\endaligned\end{equation*}
which implies $\lim_{i\to\infty}\mathrm{cap}_t(I_i,B)=0<\mathrm{cap}_t(I_0,B)$.
\end{remark}

Let $\{\Om_i\}_{i\ge1}$ be a sequence of open sets  in $\Si$ satisfying $\Om_i\subset\Om_{i+1}$ and $d(p,\p\Om_i)\rightarrow\infty$ for some fixed point $p\in\Si$, then from \eqref{captmonot} we have
$$\mathrm{cap}_t(K,\Si)=\lim_{i\rightarrow\infty}\mathrm{cap}_t(K,\Om_i).$$
For simplicity, we denote
$\mathrm{cap}_t(K)=\mathrm{cap}_t(K,\Si)$ as before.

From Proposition \ref{semi-cont}, we immediately have the following conclusion.
\begin{corollary}\label{semi-contSi}
For any compact set $K\subset\Si$,  let $K_i$ be a sequence of compact sets in $\Si$ with $K_i\to K$ in the Hausdorff sense. Then
$\limsup_{i\to\infty}\mathrm{cap}_t(K_i)\le\mathrm{cap}_t(K)$.
\end{corollary}
Now let us prove by contradiction that (1) implies (2) in Theorem \ref{main}.
\begin{proposition}\label{ParaKE}
If there are a constant $t>0$ and a compact set $K$ in $\Si$ with $\mathcal{H}^{n}(K)>0$ so that $\mathrm{cap}_t(K)=0$, then $\mathrm{cap}_T(F)=0$ for any compact set $F\subset\Si$ and $T>0$.
\end{proposition}
\begin{proof}
Suppose $\mathrm{cap}_t(K)=0$ for a constant $t>0$ and a compact set $K$ in $\Si$ with $\mathcal{H}^{n}(K)>0$.
Then there is a sequence of functions $\phi_i\in\mathcal{L}_t(K,B_i(p))$ such that
\begin{equation}\aligned\label{Bipphii-1}
\lim_{i\rightarrow\infty}\int_{B_i(p)}\left(\sqrt{1+|D\phi_i|^2}-1\right)=0.
\endaligned\end{equation}
For a compact set $F$ with $K\subset F\subset\Si$, we fix the integer $j$ so that $F\subset B_{j/2}(p)$. For each integer $i\ge j$, with Cauchy inequality we have
\begin{equation}\aligned\label{Bipphii-1*}
\int_{B_i(p)}\left(\sqrt{1+|D\phi_i|^2}-1\right)&\ge\int_{\{|D\phi_i|\le1\}}\left(\sqrt{1+2(\sqrt{2}-1)|D\phi_i|^2+(\sqrt{2}-1)^2|D\phi_i|^4}-1\right)\\
&+\int_{\{|D\phi_i|\ge1\}}\left(\sqrt{1+2(\sqrt{2}-1)|D\phi_i|+(\sqrt{2}-1)^2|D\phi_i|^2}-1\right)\\
=(\sqrt{2}&-1)\int_{\{|D\phi_i|\le1\}}|D\phi_i|^2+(\sqrt{2}-1)\int_{\{|D\phi_i|\ge1\}}|D\phi_i|\\
\ge\f{\sqrt{2}-1}{\mathcal{H}^n(B_{2j}(p))}&\left(\int_{\{|D\phi_i|\le1\}\cap B_{2j}(p)}|D\phi_i|\right)^2+(\sqrt{2}-1)\int_{\{|D\phi_i|\ge1\}}|D\phi_i|.
\endaligned\end{equation}
Combining \eqref{Bipphii-1}\eqref{Bipphii-1*}, we get
\begin{equation}\aligned\label{BjphiDphiij000}
\lim_{i\rightarrow\infty}\int_{B_{2j}(p)}|D\phi_i|=0.
\endaligned\end{equation}
With Neumann-Poincar$\mathrm{\acute{e}}$ inequality on $B_{j}(p)$, we have
\begin{equation}\aligned
\lim_{i\rightarrow\infty}\fint_{B_{j}(p)}\left|\phi_i-\fint_{B_{j}(p)}\phi_i\right|=0.
\endaligned\end{equation}
From $\phi_i\equiv t$ on $K$, we conclude that $\phi_i\rightarrow t$ a.e. on $B_{j}(p)$. 
Let $\e_{i}$ be a Lipschitz function defined on $\Si$ so that $\e_{i}\equiv1$ on $\Si\setminus B_{j}(p)$, $\e_{i}\equiv0$ on $B_{j-1}(p)$, $\e_{i}=t+1-j$ on $\p B_s(p)$ for each $s\in[j-1,j)$. From \eqref{BjphiDphiij000} and $\phi_i\rightarrow t$ a.e. on $B_{j}(p)$, we have
\begin{equation}\aligned\label{BsiphiDphiij000}
\int_{B_{j}(p)\setminus B_{j-1}(p)}&\left(\sqrt{1+|D((t-\phi_i)\e_{i})|^2}-1\right)\le\int_{B_{j}(p)\setminus B_{j-1}(p)}|D((t-\phi_i)\e_{i})|\\
\le&\int_{B_{j}(p)\setminus B_{j-1}(p)}(t-\phi_i)+\int_{B_{j}(p)\setminus B_{j-1}(p)}|D\phi_i|\to0
\endaligned\end{equation}
as $i\to\infty$.
Let $\phi^*_i=t-(t-\phi_i)\e_{i}$, then $\phi^*_i\in\mathcal{L}_t(F,B_i(p))$.
Combining \eqref{Bipphii-1}\eqref{BsiphiDphiij000}, it follows that
\begin{equation}\aligned
\lim_{i\rightarrow\infty}\int_{B_i(p)}&\left(\sqrt{1+|D\phi^*_i|^2}-1\right)\\
=\lim_{i\rightarrow\infty}\int_{B_i(p)\setminus B_{j}(p)}&\left(\sqrt{1+|D\phi_i|^2}-1\right)+\lim_{i\rightarrow\infty}\int_{B_{j}(p)\setminus B_{j-1}(p)}\left(\sqrt{1+|D\phi^*_i|^2}-1\right)=0,
\endaligned
\end{equation}
which implies $\mathrm{cap}_t(F)=0$. Combining \eqref{captTKOm}, we complete the proof.
\end{proof}

From \eqref{capt*KOm}, a complete parabolic manifold $\Si$ is $M$-parabolic automatically. Now we provide an example, which is $M$-parabolic, but not parabolic.
\begin{proposition}\label{M-par-not-par}
For each integer $n\ge2$, there is an $n$-dimensional complete noncompact rotationally symmetric manifold $S$ such that $\mathrm{cap}(F)>0$ and $\mathrm{cap}_t(F)=0$ for each $t>0$ and every compact set $F\subset\Si$.
\end{proposition}
\begin{proof}
Let $n\ge2$, and $f$ be a smooth function on $[0,\infty)$ satisfying $f(0)=0$, $f'(0)=1$, $f>0$ on $(0,\infty)$, $\liminf_{r\rightarrow\infty}f(r)=0$, $\int_1^\infty f(r)dr=\infty$, $\int_1^\infty f^{-n}(r)dr<\infty$ and $\int_1^\infty f^{-n-1}(r)dr<\infty$.
For example, we may choose
$$f(r)=r(1+r)\left(\sin^2r+(1+r^2)^{-4(n+1)}\right)^{\f1{4n}}\qquad \mathrm{for\ all}\ r\ge0.$$
Then clearly $\lim_{k\rightarrow\infty}f(k\pi)=k\pi(1+k\pi)(1+k^2\pi^2)^{-\f{n+1}{n}}=0$, 
\begin{equation}\aligned
\int_\pi^\infty f^{-n}(r)dr<\int_\pi^\infty r^{-2n}|\sin r|^{-\f12}dr\le\sum_{k\ge1}(k\pi)^{-2n}\int_{k\pi}^{(k+1)\pi}|\sin r|^{-\f12}dr<\infty
\endaligned
\end{equation}
and similarly
\begin{equation}\aligned
\int_\pi^\infty f^{-n-1}(r)dr<\int_\pi^\infty r^{-2(n+1)}|\sin r|^{-\f{n+1}{2n}}dr<\int_\pi^\infty r^{-2(n+1)}|\sin r|^{-\f{3}{4}}dr<\infty.
\endaligned
\end{equation}

Let $S=(\R^n,\si_S)$ be an $n$-dimensional complete noncompact manifold with the metric
$$\si_S=dr^2+f^2(r)d\th^2$$
in the polar coordinate, where $d\th^2$ is a standard metric on the $(n-1)$-dimensional unit sphere $\mathbb{S}^{n-1}$.
Let $\De_S$ be the Laplacian on $S$, then
$$\De_S=\f{\p^2}{\p r^2}+\f{nf'}{f}\f{\p}{\p r}+\f1{f^2}\De_{\mathbb{S}^{n-1}}$$
in the polar coordinate, where $\De_{\mathbb{S}^{n-1}}$ is the Laplacian of $\mathbb{S}^{n-1}$.
For any $r>0$, let $B_{r}$ denote the geodesic ball in $\Si$ with radius $r$ and centered at the vertex (w.r.t. the polar coordinate).
Given a constant $r_0\ge1$, let $\phi(r,\th)=\phi(r)$ be a nonnegative harmonic function on $S\setminus B_{r_0}$ with $\phi=1$ on $B_{r_0}$ and $\lim_{r\to\infty}\phi(r)=0$. From
\begin{equation}\aligned
0=\De_S\phi=\f{\p^2\phi}{\p r^2}+\f{nf'}{f}\f{\p\phi}{\p r}\qquad\qquad \mathrm{on}\ (r_0,\infty),
\endaligned
\end{equation}
we get 
\begin{equation*}\aligned
\f{\p\phi}{\p r}=-cf^{-n},\qquad \phi(r)=1-c\int_{r_0}^rf^{-n}(s)ds
\endaligned
\end{equation*}
for the constant $c=\left(\int_{r_0}^\infty f^{-n}(r)dr\right)^{-1}>0$.
Moreover,
\begin{equation*}\aligned
\mathrm{cap}(\overline{B}_{r_0})=\int_{S\setminus B_{r_0}}|D\phi|^2=n\omega_n\int_{r_0}^\infty\left(\f{\p\phi}{\p r}\right)^2f^{n-1}(r)dr=c^2n\omega_n\int_{r_0}^\infty f^{-n-1}(r)dr<\infty.
\endaligned
\end{equation*}
Combining \eqref{captTKOm}, it remains to prove $\mathrm{cap}_1(K)=0$ for any compact set $K\subset\Si$. 

By the definition of $f$, there is a sequence $r_i\rightarrow\infty$ with $f(r_i)\rightarrow0$ as $i\rightarrow\infty$. Then there is a sequence $\de_i\rightarrow0$ with $0<\de_i<1$ such that $f(r)\le 2f(r_i)$ for all $r\in(r_i,r_i+\de_i]$.
Given a compact $F$ in $S$,
for each integer $i$ with $F\subset B_{r_i}$, we define a function $u_i$ on $B_{r_i+\de_i}$ by letting $u_i=1$ on $B_{r_i}$ and $u_i(r,\th)=1-\f{r-r_i}{\de_i}$ on $B_{r_i+\de_i}\setminus B_{r_i}$. Then we have
\begin{equation}\aligned
\mathrm{cap}_1(F,B_{r_i+\de_i})\le\int_{B_{r_i+\de_i}\setminus F}&\left(\sqrt{1+|D u_i|^2}-1\right)\\
=n\omega_n\int_{r_i}^{r_i+\de_i}(1+\de_i^{-2})^{1/2}f^{n-1}(r)d&r\le n\omega_n(1+\de_i^{-2})^{1/2}\int_{r_i}^{r_i+\de_i}(2f)^{n-1}(r_i)dr\\
=n\omega_n2^{n-1}f^{n-1}(r_i)(1+\de_i^2)^{1/2}\le n&\omega_n2^{n}f^{n-1}(r_i),
\endaligned
\end{equation}
which implies
\begin{equation}\aligned
\mathrm{cap}_1(F)=\lim_{i\rightarrow\infty}\mathrm{cap}_1(F,B_{r_i+\de_i})=0.
\endaligned
\end{equation}
This completes the proof.
\end{proof}

\section{Definitions from geometric measure theory}

Let $(\Si,\si)$ be an $n$-dimensional complete manifold with Levi-Civita connection $D$. The product manifold $\Si\times\R$ admits a standard product metric $\si+dt^2$. Let $\bn$ denote the Levi-Civita connection of $\Si\times\R$ w.r.t. $\si+dt^2$.
Let $E_{n+1}$ denote the constant unit vector in $\Si\times\R$ perpendicular to $\Si\times\{0\}$.

Let $\overline{\mathrm{div}}$ denote the divergence of $\Si\times\R$.
For an open set $U\subset\Si\times\R$ and a function $f\in L^1(U)$, we define the variation of $f$ on $U$ by (see \cite{AFP}\cite{Gi} for the Euclidean case)
\begin{equation}\aligned
\int_U|\overline{\na}f|=\sup\left\{\int_U f\, \overline{\mathrm{div}}(Y)\Big|\, Y\in \G^1_c(TU),\, |Y|\le1\ \mathrm{on}\ \Si\times\R\right\},
\endaligned
\end{equation}
where $\G^1_c(TU)$ denotes the space containing all $C^1$ tangent vector fields on $U$ with compact supports in $U$.
Let $BV(U)$ denote the space of all functions in $L^1(U)$ with bounded variation, and $BV_{loc}(U)$ denote the space of all functions in $L^1_{loc}(U)$ with bounded variation on compact subsets. For a Borel set $\mathscr{B}$ in $\Si\times\R$, let 
$$\mathrm{Per}(\mathscr{B},U)=\int_U|\overline{\na}\chi_{_\mathscr{B}}|$$ 
denote the perimeter of $\mathscr{B}$ in $U$, where $\chi_{_{\mathscr{B}}}$ is the characteristic function on $\mathscr{B}$. 
If $\mathscr{B}$ has locally finite perimeter, i.e., $\mathrm{Per}(\mathscr{B},U')<\infty$ for every bounded open $U'\subset\subset U$, then $\mathscr{B}$ is called \emph{a Caccioppoli set}.
For each open subset $\Om\subset\Si$ and functions in $L^1(\Om)$, we may define perimeters and the space of BV functions similar to the ones on $\Si\times\R$.

Suppose that $U$ is properly embedded in $\R^{n+m}$ for some integer $m\ge1$.
For a set $S$ in $U$ and an integer $1\le k\le n+1$, $S$ is said to be \emph{countably $k$-rectifiable}
if $S\subset S_0\cup\bigcup_{j=1}^\infty F_j(\R^k)$, where $\mathcal{H}^k(S_0)=0$, and $F_j:\, \R^k\rightarrow \R^{n+m}$ are Lipschitz mappings for all integers $j\ge1$.
Let $\mathcal{D}^n(U)$ denote the set including all smooth $n$-forms on $U$ with compact supports in $U$. Denote $\mathcal{D}_n(U)$ be the set of $n$-currents in $U$, which are continuous linear functionals on $\mathcal{D}^n(U)$.
For each $T\in \mathcal{D}_n(U)$, one defines the mass of $T$ by
\begin{equation*}\aligned
\mathbf{M}(T)=\sup_{|\omega|_U\le1,\omega\in\mathcal{D}^n(U)}T(\omega)
\endaligned
\end{equation*}
with $|\omega|_U=\sup_{x\in U}\lan\omega(x),\omega(x)\ran^{1/2}$.
Let $\p T$ be the boundary of $T$ defined by $\p T(\omega')=T(d\omega')$ for any $\omega'\in\mathcal{D}^{n-1}(U)$ (see \cite{S} for instance). For the Borel set $\mathscr{B}$ in $\Si\times\R$, let $T\llcorner\mathscr{B}$ denote a current in $\mathcal{D}_n(U)$ by
$(T\llcorner\mathscr{B})(\omega)=T(\omega\chi_{_\mathscr{B}})$ for each $\omega\in\mathcal{D}^n(U)$.
Denote $T\setminus\mathscr{B}=T\llcorner((\Si\times\R)\setminus\mathscr{B})$.
A current $T\in\mathcal{D}_n(U)$ is said to be an \emph{integer multiplicity current} if it can be expressed as
$$T(\omega)=\int_M\th\lan \omega,\xi_T\ran\qquad \mathrm{for\ each}\ \omega\in \mathcal{D}^n(U),$$
where $M$ is a countably $n$-rectifiable subset of $U$, $\th$ is a locally $\mathcal{H}^n$-integrable positive integer-valued function, and $\xi_T$ is an orientation on $M$, i.e.,  $\xi_T(x)$ is an $n$-vector representing the approximate tangent space $T_xM$ for $\mathcal{H}^n$-a.e. $x$.

For a countably $n$-rectifiable set $S\subset U$ with orientation $\xi$, 
there is an $n$-current $[|S|]$ associated with $M$, i.e.,
$$[|S|](\omega)=\int_M\lan \omega,\xi\ran,  \qquad  \forall\omega\in\mathcal{D}^n(U).$$
Let $\xi_\Si$ be an orientation on $\Si$ 
so that the dual form of $\xi_\Si$ is the standard volume element on $\Si$. By parallel transport, we can assume that $\xi_\Si$ is defined on $\Si\times\R$.
For an $n$-dimensional Borel set $V\subset\Si$ and $s\in\R$, let $[|V|]\times\{s\}$ denote the current associated with $V\times\{s\}$, i.e., 
\begin{equation}\aligned
([|V|]\times\{s\})(\omega)=\int_{V\times\{s\}}\lan\omega,\xi_\Si\ran, \qquad  \forall\omega\in\mathcal{D}^n(\Si\times\R).
\endaligned
\end{equation}
Denote $\p([|V|]\times\{s\})=\p[|V|]\times\{s\}$.
For an $(n+1)$-dimensional Borel set $\widetilde{V}\subset\Si\times\R$, let $[|\widetilde{V}|]$ denote the current associated with $\widetilde{V}$, i.e., 
\begin{equation}\aligned
\big[\big|\widetilde{V}\big|\big](\widetilde{\omega})=\int_{\widetilde{V}}\lan\widetilde{\omega},\xi_\Si\wedge E_{n+1}\ran, \qquad  \forall\widetilde{\omega}\in\mathcal{D}^{n+1}(\Si\times\R).
\endaligned\end{equation}

For a subset $A\subset U$,
an integer multiplicity current $T\in\mathcal{D}_n(U)$ is said to be \emph{minimizing} in $A$ if
$\mathbf{M}(T)\le \mathbf{M}(T')$ whenever $W\subset\subset U$, $\p T'=\p T$ in $U$, spt$(T'-T)$ is a compact subset of $A\cap W$ (see \cite{S} for instance). A countably $n$-rectifiable set $S$ (with a natural orientation $\xi$) is said to be an \emph{area-minimizing} hypersurface (or \emph{area-minimizing}) in $U$ if the associated current $[|S|]$ is minimizing in $U$.

\section{BV solutions associated with the capacities}

Let $(\Si,\si)$ be an $n$-dimensional complete manifold with Levi-Civita connection $D$.
\begin{lemma}\label{MiniSmooth}
Let $\Om$ be an open set in $\Si$, and $M$ be an area-minimizing hypersurface in $\Om\times(-R,R)$ for some constant $0<R<\infty$. If there is a bounded measurable function $u$ on $\Om$ such that
$$M=\p\{(x,t)\in\Om\times\R|\, t<u(x)\}\cap(\Om\times\R),$$
then $u$ is smooth on $\Om$ with locally bounded gradient.
\end{lemma}
\begin{proof}
In the Euclidean case, De Giorgi has proved $u\in C^\infty(\Om)$ in \cite{Dg} (see also Theorem 14.13 in \cite{Gi}). In our situation of the manifold $\Si$, the proof is essentially same, where the interior gradient estimate is derived with the help of the following formula (see \cite{Sp} or (2.4) in \cite{DJX} for instance)
\begin{equation}\aligned\label{DeMEn1nu}
\De_M\lan E_{n+1},\mathbf{n}_M\ran=-\left(|A_M|^2+\overline{Ric}(\mathbf{n}_M,\mathbf{n}_M)\right)\lan E_{n+1},\mathbf{n}_M\ran\qquad 
\endaligned
\end{equation}
on the regular part $M_*$ of $M$. Here, $\De_M$ denotes the Laplacian on $M_*$, $A_M$ denotes the second fundamental form of $M_*$, $\overline{Ric}$ denotes the Ricci curvature of $\Si\times\R$, and $\mathbf{n}_M$ denotes the unit normal vector field to $M_*$.
\end{proof}

Let $\Om$ be a bounded open set in $\Si$ and $K$ be a compact set in $\Om$
such that $\p\Om$ and $\p K$ are Lipschitz-continuous.
Let $\varphi_t$ be a function on $\p\Om\cup\p K$ with $\varphi_t=t$ on $\p K$ and $\varphi=0$ on $\p\Om$.
For a function $f$ of bounded variation on $\Om\setminus K$, let $\mathrm{tr}f$ denote the trace of $f$ on $\p\Om\cup\p K$ (see $\S$2 in \cite{Gi} for details) defined by
\begin{equation}\aligned\label{DefTr}
\mathrm{tr}f(x)=\lim_{\r\to0}\fint_{B_\r(x)\cap\Om\setminus K}f\qquad\mathrm{for}\ \mathcal{H}^{n-1}\mathrm{-almost\ all}\ x\in\p\Om\cup\p K.
\endaligned
\end{equation}
From Proposition 14.3 and its proof in \cite{Gi} (the case of manifolds is the same as the Euclidean case essentially),
\begin{equation*}\aligned
\mathrm{cap}_t(K,\Om)=&\inf\left\{\int_{\Om\setminus K}\left(\sqrt{1+|Df|^2}-1\right)\Big|\, f\in BV(\Om\setminus K),\, \mathrm{tr}f=\varphi_t\ \mathrm{on}\ \p\Om\cup\p K\right\}\\
=&\inf\left\{\int_{\Om\setminus K}\left(\sqrt{1+|Df|^2}-1\right)+\int_{\p\Om\cup\p K}|\mathrm{tr}f-\varphi_t|\Big|\, f\in BV(\Om\setminus K)\right\}.
\endaligned
\end{equation*}
From Theorem 14.5 in \cite{Gi}, the functional
\begin{equation}\aligned
\int_{\Om\setminus K}\left(\sqrt{1+|Df|^2}-1\right)+\int_{\p\Om\cup\p K}|\mathrm{tr}f-\varphi_t|
\endaligned
\end{equation}
attains its minimum in $BV(\Om\setminus K)$. Moreover, with Theorem 14.13 in \cite{Gi} there exists a function $u\in BV(\Om\setminus K)$ so that $u$ is a smooth solution to \eqref{u} on $\Om\setminus K$ and
\begin{equation}\aligned\label{capttru}
\mathrm{cap}_t(K,\Om)=\int_{\Om\setminus K}\left(\sqrt{1+|Du|^2}-1\right)+\int_{\p\Om\cup\p K}|\mathrm{tr}u-\varphi_t|.
\endaligned
\end{equation}

Noting that our boundary data $\varphi_t$ are constants, which is simpler than the one in $\S14$ in \cite{Gi}. So we can consider a more general boundary condition.
Using geometric measure theory, we have an analogous representation as \eqref{capttru} in the following with Lipschitz condition of $\p(\Om\setminus K)$ replaced by $\p\,\overline{\Om\setminus K}=\p(\Om\setminus K)$. 
The condition $\p\,\overline{\Om\setminus K}=\p(\Om\setminus K)$ is equivalent to that 
$$\mathcal{H}^n(B_r(x)\cap\Om\setminus K)<\mathcal{H}^n(B_r(x))$$
for any $x\in\p(\Om\setminus K)$ and any $r>0$. See Proposition 3.1 in \cite{Gi} for more discussions on Borel sets satisfying the above inequality.
\begin{theorem}\label{wExist}
Let $\Om$ be a bounded open set in $\Si$, and $K$ be a compact set in $\Om$ with $\p\,\overline{\Om\setminus K}=\p(\Om\setminus K)$. 
For any $t>0$, there is a function $u$ of bounded variation on $\Si$ with $0\le u\le t$  such that $u$ is a smooth solution to \eqref{u} on $\Om\setminus K$, $u=t$ on $K\setminus\p K$, $u=0$ on $\Si\setminus\overline{\Om}$,  the boundary of the subgraph $U=\{(x,s)\in\Si\times\R|\, s< u(x)\}$ is countably $n$-rectifiable set in $\Si\times\R$,
and
\begin{equation}\aligned\label{captKOmu}
\mathrm{cap}_t(K,\Om)=&\mathcal{H}^n\left((\overline{\Om}\times\R)\cap\p U\right)-\mathcal{H}^n(\Om)=\mathcal{H}^n(M)-\mathcal{H}^n(\Om\setminus K)\\
=&\int_{\Om\setminus K}\left(\sqrt{1+|Du|^2}-1\right)+\int_{\p K}(t-u)+\int_{\p\Om}u.
\endaligned
\end{equation}
Here, $M=\p U\cap\left(\overline{\Om\setminus K}\times\R\right)$. Moreover, $\p_+\Om=\{x\in\p\Om|\, u(x)>0\}$, $\p_-K=\{x\in\p K|\, u(x)<t\}$ are both countably $(n-1)$-rectifiable, $u$ is lower semi-continuous on a small neighborhood of $\p K$, and upper semi-continuous on a small neighborhood of $\p\Om$.
\end{theorem}
\begin{remark}
The function $u$ here is said to be a (BV) solution on $\Si$ associated with $\mathrm{cap}_t(K,\Om)$ as in the introduction.
The condition $\p\,\overline{\Om\setminus K}=\p(\Om\setminus K)$ is necessary in view of Remark \ref{capt0Ki}.
\end{remark}
\begin{proof}
The proof is in the same spirit of the proof of \eqref{capttru} in $\S14$ of \cite{Gi}.
For any fixed $t>0$, there is a sequence of Lipschitz functions $u_i$ on $\Si$ with supports in $\overline{\Om}$ such that $0\le u_i\le t$, $u_i\big|_{K}=t$ and
\begin{equation}\aligned\label{captKOmui}
\mathrm{cap}_t(K,\Om)=\lim_{i\rightarrow\infty}\int_{\Om}\left(\sqrt{1+|Du_i|^2}-1\right).
\endaligned\end{equation}
Let $U_i$ denote the subgraph of $u_i$ defined by
$$U_i=\{(x,\tau)\in\Si\times\R|\, \tau< u_i(x)\}.$$
Let $V$ be a bounded open set in $\Si$ containing $\overline{\Om}$. Then
\begin{equation}\aligned\label{UDchiVi}
\int_{V\times\R}\left|\overline{\na}\chi_{_{U_i}}\right|=\int_{V}\sqrt{1+|Du_i|^2},
\endaligned\end{equation}
which means that $\chi_{_{V_i}}$ has uniform bounded variation on $V$ from \eqref{captKOmui}.
From compactness theorem for functions of bounded variation (see Theorem 2.6 in \cite{S} for instance), there is a function $\phi$ in $BV(V\times\R)$ such that
\begin{equation}\aligned\label{VphichiUi}
\lim_{i\rightarrow\infty}\int_{V\times\R}|\phi-\chi_{_{U_i}}|=0,
\endaligned\end{equation}
and
\begin{equation}\aligned
\int_{V\times\R}\left|\overline{\na}\phi\right|\le\liminf_{i\rightarrow\infty}\int_{V\times\R}\left|\overline{\na}\chi_{_{U_i}}\right|
\endaligned\end{equation}
by semicontinuity up to the subsequence. Combining \eqref{captKOmui}\eqref{UDchiVi} we have
\begin{equation}\aligned\label{UDphichiVi}
\int_{V\times\R}\left|\overline{\na}\phi\right|\le&\liminf_{i\rightarrow\infty}\int_{V}\sqrt{1+|Du_i|^2}=\mathcal{H}^n(V\setminus\Om)+\liminf_{i\rightarrow\infty}\int_{\Om}\sqrt{1+|Du_i|^2}\\
=&\mathcal{H}^n(V\setminus\Om)+\mathcal{H}^n(\Om)+\mathrm{cap}_t(K,\Om)=\mathcal{H}^n(V)+\mathrm{cap}_t(K,\Om).
\endaligned\end{equation}
Since $\chi_{_{U_i}}\rightarrow\phi$ a.e. in $V\times\R$, we may assume that $\phi$ is a characteristic function on an open set $U^*\subset\Si\times\R$. Set 
$$U=U^*\cup(\p U^*\cap(\p\Om\times\R)).$$ 
Moreover, there is a function $u:\,\Si\to[0,t]$ with $u\big|_{K\setminus\p K}=t$ and $u=0$ outside $\overline{\Om}$ such that
\begin{equation}\aligned\label{defUu}
U=\{(x,\tau)\in\Si\times\R|\, \tau< u(x)\}.
\endaligned\end{equation}
Let $x_i$ be a sequence of points in $\Om\setminus K$ and $x_i\to x_0\in\p K\cup\p\Om$. Since any limit of $(x_i,u(x_i))$ belongs to $\p U$ and $(x_0,u(x_0))\in\p U$, we get $u(x_i)\to u(x_0)$. It follows that $\limsup_{x\to x_0}u(x)\ge u(x_0)$ for any $x_0\in\p K$, and $\liminf_{x\to y_0}u(x)\le u(y_0)$ for any $y_0\in\p \Om$. In other words,
$u$ is lower semi-continuous on a small neighborhood of $\p K$, and upper semi-continuous on a small neighborhood of $\p\Om$.

From \eqref{VphichiUi}\eqref{UDphichiVi}, we can assume $\phi=\chi_{_U}$, and
\begin{equation}\aligned\label{PUVtimesR}
\mathrm{Per}(U,V\times\R)=\int_{V\times\R}\left|\overline{\na}\phi\right|\le\mathcal{H}^n(V)+\mathrm{cap}_t(K,\Om).
\endaligned\end{equation}
Hence, $U$ has locally finite perimeter in $\Si\times\R$, namely, $U$ is a Caccioppoli set in $\Si\times\R$.
Then $u$ has bounded variation on any compact set of $\Si$. From the countably $n$-rectifiable $\p U$, both
\begin{equation*}\aligned
\p U\cap(\p\Om\times\R)=&\{(x,\tau)|\, x\in\p\Om,\ 0\le\tau\le u(x)\},\\
\mathrm{and}\ \ \p U\cap(\p K\times\R)=&\{(x,\tau)|\, x\in\p K,\ u(x)\le\tau\le t\}
\endaligned\end{equation*}
are countably $n$-rectifiable. 
In particular, we have (see Proposition 3.1 in \cite{Gi})
\begin{equation}\aligned\label{01pU}
0<\mathcal{H}^n(U\cap B_\r(x))<\mathcal{H}^n(B_\r(x))\qquad \mathrm{for\ any}\ x\in\p U,\ \r>0.
\endaligned\end{equation}
Suppose $\mathcal{H}^n(\p U\cap(\p\Om\times\R))>0$. By the definition of $U$, $\{(x,\tau_2)|\, x\in\p\Om\}\cap\p U\subset\{(x,\tau_1)|\, x\in\p\Om\}\cap\p U$ for any $0<\tau_1<\tau_2$.
By the slicing theorem \cite{S}, $\{(x,\tau)|\, x\in\p\Om\}\cap\p U$ is countably $(n-1)$-rectifiable for almost all $\tau>0$.
This infers that $\{x\in\p\Om|\, u(x)>0\}$ is countably $(n-1)$-rectifiable. Similarly, $\{x\in\p K|\, u(x)<t\}$ is countably $(n-1)$-rectifiable. 
Set $\p\Om_s=\{x\in\p\Om|\, u(x)>s\}$, $\p K_s=\{x\in\p K|\, u(x)<s\}$,
$\p_+\Om=\lim_{s\to0+}\p\Om_s=\{x\in\p\Om|\, u(x)>0\}$ and $\p_-K=\lim_{s\to t-}\p K_s=\{x\in\p K|\, u(x)<t\}$. Since
$$\lim_{t\to0}\mathcal{H}^n\left(\{(x,\tau)|\, x\in\p\Om,\,0<\tau<t\}\cap\p U\right)=0,$$
from Levi's theorem (for monotonic functions) we have
\begin{equation}\aligned\label{pOm+-}
&\mathcal{H}^n\left(\p U\cap(\p\Om\times\R)\right)+\mathcal{H}^n\left(\p U\cap(\p K\times\R)\right)\\
=&\lim_{s\to0+}\mathcal{H}^n\left(\p U\cap(\p\Om\times[s,\infty))\right)+\lim_{s\to t-}\mathcal{H}^n\left(\p U\cap(\p K\times(-\infty,s])\right)\\
=&\lim_{s\to0+}\int_{\p\Om_s}(u-s)+\lim_{s\to t-}\int_{\p K_s}(s-u)\\
=&\lim_{s\to0+}\int_{\p_+\Om}(u-s)\chi_{_{\p\Om_s}}+\lim_{s\to t-}\int_{\p_- K}(s-u)\chi_{_{\p K_s}}\\
=&\int_{\p_+\Om}u+\int_{\p_- K}(t-u).
\endaligned\end{equation}

By the choice of $u_i$, for any open set $\Om_*\subset\subset \Om\setminus K$ we have
\begin{equation}\aligned\label{PerUOm*MIN}
\mathrm{Per}(U,\Om_*\times\R)=\int_{\Om_*\times\R}\left|\overline{\na}\phi\right|\le\int_{\Om_*}\sqrt{1+|Dw|^2},
\endaligned\end{equation}
whenever $w$ is a Lipschitz function on $\Om$ with $w=u$ on $\Om\setminus\Om_*$. With Theorem 14.8 in \cite{Gi} and \eqref{01pU}\eqref{PerUOm*MIN}, we conclude that $\p U$ is an area-minimizing hypersurface in $\Om_*\times\R$.
From Lemma \ref{MiniSmooth},  $u$ is a smooth solution to \eqref{u} on $\Om_*$. Thus, $u$ is a smooth solution to \eqref{u} on $\Om\setminus K$. 
Denote $M=\p U\cap\left(\overline{\Om\setminus K}\times\R\right)$.
Combining with \eqref{PUVtimesR}\eqref{pOm+-},  we have
\begin{equation}\aligned
\mathrm{cap}_t&(K,\Om)\ge\mathcal{H}^n\left((\overline{\Om}\times\R)\cap\p U\right)-\mathcal{H}^n(\Om)=\mathcal{H}^n(M)-\mathcal{H}^n(\Om\setminus K)\\
=&\int_{\Om\setminus K}\left(\sqrt{1+|Du|^2}-1\right)+\mathcal{H}^n(\p U\cap(\p K\times\R))+\mathcal{H}^n(\p U\cap(\p\Om\times\R))\\
=&\int_{\Om\setminus K}\left(\sqrt{1+|Du|^2}-1\right)+\int_{\p_- K}(t-u)+\int_{\p_+\Om}u.
\endaligned
\end{equation}
Clearly, $u$ can be approached by Lipschitz functions on $\Si$ taking values in $[0,t]$ with supports in $\overline{\Om}$ and equal to $t$ on $K$.
This completes the proof by combining with \eqref{captKOmui}.
\end{proof}

\begin{remark} For the case of mean convex $\Om\setminus K$, the Dirichlet problem for minimal hypersurface equation on $\Om\setminus K$ is solvable for classic solutions from \cite{JS}\cite{Sp}.
In general, it may not have a classic solution $u$ with the prescribed boundary data $u=t$ on $\p K$ and $u=0$ on $\p\Om$ (see Proposition \ref{non-bded BV solutions} in Appendix II for instance). 
\end{remark}

From the strictly convex volume functional of graphs, solutions associated with $\mathrm{cap}_t(K,\Om)$ are unique in the following sense (refer Lemma 12.5 and Proposition 14.2 both in \cite{Gi}). 
\begin{proposition}\label{UNIu}
Let $K,\Om$ be defined as in Theorem \ref{wExist}.
The solution on $\Si$ associated with $\mathrm{cap}_t(K,\Om)$ is unique up to a constant. Moreover, if there are two different solutions $u_1,u_2$ associated with $\mathrm{cap}_t(K,\Om)$, then $\mathcal{H}^{n-1}(\p K)=\mathcal{H}^{n-1}(\p\Om)<\infty$.
\end{proposition}
\begin{proof}
Suppose that there are two solutions $u_1,u_2$ on $\Si$ associated with $\mathrm{cap}_t(K,\Om)$.
We consider a function $u_*=\f12(u_1+u_2)$. 
By the strictly convexity of $\sqrt{1+s^2}$, it follows that
\begin{equation}\aligned
\sqrt{1+|Du_*|^2}\le\f12\sum_{i=1}^2\sqrt{1+|Du_i|^2}\qquad \mathrm{on}\ \Om\setminus K,
\endaligned
\end{equation}
where the 'equality' attains only if $|D(u_1-u_2)|=0$. From \eqref{captKOmu}, we have
\begin{equation}\aligned
\int_{\Om\setminus K}&\left(\sqrt{1+|Du_*|^2}-1\right)+\int_{\p K}(t-u_*)+\int_{\p\Om}u_*\\
\le\f12\sum_{i=1}^2&\int_{\Om\setminus K}\left(\sqrt{1+|Du_i|^2}-1\right)+\int_{\p K}(t-u_i)+\int_{\p\Om}u_i=\mathrm{cap}_t(K,\Om).
\endaligned
\end{equation}
Since
\begin{equation}\aligned
\mathrm{cap}_t(K,\Om)\le\int_{\Om\setminus K}\left(\sqrt{1+|Du_*|^2}-1\right)+\int_{\p K}(t-u_*)+\int_{\p\Om}u_*,
\endaligned
\end{equation}
we get $|D(u_1-u_2)|\equiv0$ on $\Om\setminus K$, and then there is a constant $c$ so that $u_1=u_2+c$. If $c\neq0$, then from \eqref{captKOmu} we conclude that $\mathcal{H}^{n-1}(\p K)=\mathcal{H}^{n-1}(\p\Om)<\infty$.
\end{proof}

F.H. Lin proved the $C^{1,\a}$-regularity for minimizing currents in a closed set of $C^{1,\a}$-boundary in Euclidean space in his thesis \cite{Lin0}. In particular, the following boundary regularity holds (the case of manifolds is essentially the same as the Euclidean case). 
\begin{lemma}\label{C1a-reg}
Let $t>0$ be a constant, $\Om$ be a bounded open set in $\Si$ with $\p\Om\in C^{1,\a}$, and $K$ be a compact set in $\Om$ with $\p K\in C^{1,\a}$ for some $\a\in(0,1)$. If an integer multiplicity current $T$ is minimizing in $\overline{\Om\setminus K}\times\R$ with boundary $\p [|\Om|]\times\{0\}-\p [|K|]\times\{t\}$. 
Then spt$T$ is $C^{1,\a}$ in a small neighborhood of spt$T\cap((\p\Om\cup\p K)\times\R)$.
\end{lemma}

In \cite{Lin0}(p. 75), F.H. Lin also proved the uniqueness (on integral currents) for Dirichlet problem on bounded Lipschitz domains of Euclidean space. By the peculiarities of boundary data, the uniqueness (up to constants) holds without Lipschitz boundary condition as follows.  
\begin{theorem}\label{UniPlateau}
Let $\Om$ be a bounded open set in $\Si$, and $K$ be a compact set in $\Om$. Suppose that $\p\,\overline{\Om\setminus K}=\p(\Om\setminus K)$. If an integer multiplicity current $T$ is minimizing in $\overline{\Om\setminus K}\times\R$ with $T=[|K|]\times\{t\}+[|\Si\setminus\Om|]\times\{0\}$ in $\left(\Si\setminus\overline{\Om\setminus K}\right)\times\R$ for some $t>0$. 
Then spt$T$ is the boundary of the subgraph of a solution $u$ associated with $\mathrm{cap}_t(K,\Om)$, i.e.,
$T=\p[|\{(x,s)\in\Si\times\R|\, s< u(x)\}|]$.
\end{theorem}
\begin{proof}
We provide a proof here using Stokes' formula for completeness.
From the maximum principle, spt$T\subset\Si\times[0,t]$. 
There are a sequence of open sets $\Om_i\subset\subset\Om$, and a sequence compact set $K_i\subset\Om_i$ with $K\subset K_i\setminus\p K_i$ such that $\p\Om_i,\p K_i\in C^\infty$, and $\Om_i\to\Om$, $K_i\to K$ both in the Hausdorff sense as $i\to\infty$. In particular,
\begin{equation}\aligned\label{OmOmiKiK}
\limsup_{i\to\infty}\mathcal{H}^n(\Om\setminus\Om_i)+\limsup_{i\to\infty}\mathcal{H}^n(K_i\setminus K)=0.
\endaligned\end{equation}
From Corollary \ref{captKKiOmOmi***}, we have
\begin{equation}\aligned\label{KiKOmicaplim}
\lim_{i\to\infty}\mathrm{cap}_t\left(K_i,\Om_i\right)=\mathrm{cap}_t(K,\Om).
\endaligned\end{equation}
Let $T_*=T\llcorner\left(\overline{\Om\setminus K}\times\R\right)$ and $T_i=T\llcorner\left((\Om_i\setminus K_i)\times\R\right)$ for each $i$. 
From Federer-Fleming compactness theorem, there is a minimizing current $\G_i$ in $\p(\Om_i\setminus K_i)\times[0,t]$ with boundary $T\llcorner\left(\p(\Om_i\setminus K_i)\times\R\right)-\p[|\Om_i|]\times\{0\}+\p[|K_i|]\times\{t\}$.
Since $\p T_i=T\llcorner\left(\p(\Om_i\setminus K_i)\times\R\right)$, it follows that
\begin{equation}\aligned
&\p(\G_i+T_*\setminus T_i)=\p\G_i+\p T_*-\p T_i\\
=&-\p[|\Om_i|]\times\{0\}+\p[|K_i|]\times\{t\}+\p[|\Om|]\times\{0\}-\p[|K|]\times\{t\}\\
=&\p[|\Om\setminus\Om_i|]\times\{0\}+\p[|K_i\setminus K|]\times\{t\}.
\endaligned\end{equation}
Noting spt$T\cap((\p K\cup\p\Om)\times\R)$ is countably $n$-rectifiable. With \eqref{OmOmiKiK} and the minimizing property of $\G_i$, we have
\begin{equation}\aligned\label{GiTTTi0}
\limsup_{i\to\infty}\left(\mathbf{M}(\G_i)-\mathbf{M}(T_*\setminus T_i)\right)\le\limsup_{i\to\infty}\mathcal{H}^n(\Om\setminus\Om_i)+\limsup_{i\to\infty}\mathcal{H}^n(K_i\setminus K)=0.
\endaligned\end{equation}

Let $u_i$ be a solution on $\Si$ associated with $\mathrm{cap}_t(K_i,\Om_i)$, and $M_i$ be the boundary of the subgraph of $u_i$ over $\overline{\Om_i\setminus K_i}$, i.e., $M_i=\p U_i\cap\left(\overline{\Om_i\setminus K_i}\times\R\right)$. Then
$\p M_i=\mathrm{spt}(\p(T_i+\G_i))$.
Let $$\mathbf{n}_{M_i}=\f{1}{\sqrt{1+|Du_i|^2}}\left(-Du_i+E_{n+1}\right)$$ 
denote the unit normal vector field to $M_i$.
From Lemma \ref{C1a-reg}, the vector field $\mathbf{n}_{M_i}$ can be extended to $M_i\cap\p\left(\overline{\Om_i\setminus K_i}\times\R\right)$ without confusion.
Let $X_i$ denote the vector field on $(\Om_i\setminus K_i)\times\R$ by translating $\mathbf{n}_{M_i}$ along $E_{n+1}$.  From \eqref{u}, we have
\begin{equation}\aligned\label{bdivX_i}
\overline{\mathrm{div}}(X_i)=0\qquad\qquad \mathrm{on}\ (\Om_i\setminus K_i)\times\R,
\endaligned\end{equation}
where $\overline{\mathrm{div}}$ denotes the divergence of $\Si\times\R$ as before.
Since spt$T$ is embedded (outside its singular set),
let $U_i$ denote the open set enclosed by spt$T_i$ and graph$_{u_i}$ (over $\overline{\Om_i\setminus K_i}$). 
Since $\p M_i=\mathrm{spt}(\p(T_i+\G_i))$,
from Stokes' formula for $X_i$ in $U_i$, with \eqref{bdivX_i} we get 
\begin{equation}\aligned\label{M_iTiGi}
\mathcal{H}^n(M_i)\le\int_{U_i}\overline{\mathrm{div}}(X_i)+\left|\int_{\mathrm{spt}(T_i+\G_i)}\lan\mathbf{n}_{M_i},X_i\ran\right|\le\mathbf{M}(T_i)+\mathbf{M}(\G_i).
\endaligned\end{equation}
From \eqref{GiTTTi0}, for any given $\ep>0$ there is a constant $i_\ep>0$ so that $\mathbf{M}(\G_i)\le\mathbf{M}(T_*\setminus T_i)+\ep$ for all $i\ge i_\ep$. Combining \eqref{M_iTiGi} and the definition of $M_i$, one gets
\begin{equation}\aligned\label{KiOkiTep}
\mathrm{cap}_t(K_i,\Om_i)+\mathcal{H}^n(\Om_i\setminus K_i)=&\mathcal{H}^n(M_i)\le\mathbf{M}(T_i)+\mathbf{M}(\G_i)\\
\le&\mathbf{M}(T_i)+\mathbf{M}(T_*\setminus T_i)+\ep=\mathbf{M}(T_*)+\ep.
\endaligned\end{equation}
Letting $i\to\infty$ in the above inequality, with \eqref{OmOmiKiK}\eqref{KiKOmicaplim} we get
\begin{equation}\aligned
\mathrm{cap}_t(K,\Om)+\mathcal{H}^n(\Om\setminus K)=\lim_{i\to\infty}\left(\mathrm{cap}_t(K_i,\Om_i)+\mathcal{H}^n(\Om_i\setminus K_i)\right)\le\mathbf{M}(T_*)+\ep.
\endaligned\end{equation}
Since $T$ is minimizing in $\overline{\Om\setminus K}\times\R$ and $\ep>0$ is arbitrary, we get 
$$\mathrm{cap}_t(K,\Om)+\mathcal{H}^n(\Om\setminus K)=\mathbf{M}(T_*),$$ 
which infers that
spt$T\cap((\Om\setminus K)\times\R)$ is a graph from \eqref{M_iTiGi}.
We complete the proof.
\end{proof}
\begin{remark}
Since $\mathrm{spt}T\cap(\p(\Om\setminus K)\times\R)$ in Theorem \ref{UniPlateau} may be not a graph over $\p(\Om\setminus K)$, we can not deduce directly that $\mathrm{spt}T$ is a graph by Schoen's theorem (see \cite{Sc} or Theorem 1.29 in \cite{CM}).
\end{remark}

By proposition \ref{UNIu} and the maximum principle, we have the following comparison.
\begin{proposition}\label{KK'OmOm'UNIQ}
Let $K,\Om,t,u$ be defined as in Theorem \ref{wExist}.
Let $\Om'\supset\Om$ be a bounded open set in $\Si$, and $K'\supset K$ be a compact set in $\Om$ with $\p\,\overline{\Om'\setminus K'}=\p(\Om'\setminus K')$.
For $0<t'\le t$, there is a BV solution $u'$ on $\Si$ associated with $\mathrm{cap}_{t'}(K',\Om')$ such that 
\begin{equation}\aligned
u\le u'+t-t'\qquad \mathrm{on}\ \Si.
\endaligned
\end{equation}
\end{proposition}
\begin{proof}
We only need to prove $u\le u'+t-t'$ on $\overline{\Om\setminus K'}$. Let $M'$ be the boundary of the subgraph of the solution $u'+t-t'$ in $\overline{\Om\setminus K'}\times\R$, which is an area-minimizing hypersurface in $\overline{\Om\setminus K'}\times\R$ by Theorem \ref{UniPlateau}.
Let $U_+$ and $U_-$ denote two components in $\overline{\Om\setminus K'}\times\R$ divided by $M'$, where $\overline{\Om\setminus K'}\times[t,\infty)\subset U_+$. 

Suppose $M\cap U_+\neq\emptyset$. Let $W_+$ be an open set enclosed by $M$ and $M'$ in $U_+$. Let $T$ be an $n$-current defined by $T=[|M'|]+\p[|W_+|]$. Then $T$ has multiplicity one with $\p T=\p[|M'|]$. From the minimizing $M$ and $M'$ in $\overline{\Om\setminus K'}\times\R$, we get $\mathbf{M}(T)=\mathcal{H}^n(M')$, and then conclude that $T$ is also minimizing in $\overline{\Om\setminus K'}\times\R$. Let $w$ be the graphic function of spt$T$, then $w$ satisfies the minimal hypersurface equation \eqref{u}. Let $z=(x,s)$ be a point in $\p(M\cap U_+)$ with $x\in\Om\setminus K'$ and $s\in\R$. 
Since both $|Du|$ and $|Du'|$ are finite at $x$, there are two non-empty open sets $B_+$ and $B_-$ in a ball $B\subset \Om\setminus K'$ with $x\in B$ so that $w=u$ on $B_+$ and $w=u'$ on $B_-$. This contradicts to the uniqueness of solutions to \eqref{u}. Hence, we get $M\cap U_+=\emptyset$. This completes the proof.
\end{proof}

\section{Half-space properties for minimal hypersurfaces in $\Si\times\R$}

\begin{lemma}\label{ConMR1}
Let $\Si$ be a $M$-parabolic Riemannian manifold.
Let $K$ be a compact set in $\Si$ and $\Om_i\supset K$ be a sequence of open sets with $\p\,\overline{\Om_i\setminus K}=\p(\Om_i\setminus K)$ and $d(p,\p\Om_i)\to\infty$ for some $p\in\Si$. If $u_i$ is a BV solution on $\Si$ associated with $\mathrm{cap}_t(K,\Om_i)$ for some $t>0$ and every integer $i\ge1$, then
$u_i\rightarrow t$ locally on $\Si$ in the $C^0$-sense.
\end{lemma}
\begin{proof}
From the definition of cap$_t(K)$ and the $M$-parabolicity of $\Si$, we have
\begin{equation}\aligned\label{Omiconv0}
\lim_{i\rightarrow\infty}\mathrm{cap}_t(K,\Om_i)=0.
\endaligned
\end{equation}
Let $U_i$ denote the subgraph of the BV solutions $u_i$ defined by
$$U_i=\{(x,t)\in\Si\times\R|\, t\le u_i(x)\},$$
and $\chi_{_{U_i}}$ be the characteristic function on $U_i$.
For any bounded open set $V\subset\Si$ with $K\subset V$, we have
\begin{equation}\aligned\label{UDVi*}
&\int_{V\times\R}\left|\overline{\na}\chi_{_{U_i}}\right|=\mathcal{H}^n(K)+\int_{\p K}(t-u_i)+\int_{V\setminus K}\sqrt{1+|Du_i|^2}\\
\le& \mathcal{H}^n(K)+\mathcal{H}^n(V\setminus K)+\mathrm{cap}_t(K,\Om_i)=\mathcal{H}^n(V)+\mathrm{cap}_t(K,\Om_i)
\endaligned\end{equation}
for each $i$ with $V\subset \Om_i$.
From compactness theorem for functions of bounded variation (see Theorem 2.6 in \cite{S} for instance), there is a function $\phi$ in $BV_{loc}(\Si\times\R)$ such that
\begin{equation}\aligned\label{phichiVi}
\lim_{i\rightarrow\infty}\int_{V\times\R}|\phi-\chi_{_{U_i}}|=0
\endaligned\end{equation}
 up to a choice of the subsequence, and
\begin{equation}\aligned\label{UDphi*}
\int_{V\times\R}\left|\overline{\na}\phi\right|\le\liminf_{i\rightarrow\infty}\int_{V\times\R}\left|\overline{\na}\chi_{_{U_i}}\right|\le\mathcal{H}^n(V),
\endaligned\end{equation}
where we have used \eqref{Omiconv0}\eqref{UDVi*} in \eqref{UDphi*}. From \eqref{phichiVi}, $\chi_{_{U_i}}\rightarrow\phi$ for $\mathcal{H}^{n+1}$-almost all the points. So we can assume that $\phi$ is a characteristic function $\chi_{_U}$ for some open set $U\subset\Si\times\R$. From \eqref{UDphi*}, it follows that
\begin{equation}\aligned\label{PerUVVV}
\mathrm{Per}(U,V\times\R)=\int_{V\times\R}\left|\overline{\na}\chi_{_U}\right|\le\mathcal{H}^n(V).
\endaligned\end{equation}
With \eqref{PerUVVV} and $\p U\cap((K\setminus\p K)\times\R)=\{(x,t)|\, x\in K\setminus\p K\}$, we conclude that 
\begin{equation}\aligned\label{Vuxs}
\p U=\Si\times\{t\}.
\endaligned
\end{equation}
Since $u_i$ satisfies \eqref{u} on $\Si\setminus K$, $u_i$ converges smoothly to the constant $t$ on any compact set in $\Si\setminus K$.
For a bounded open set $\Om\supset\supset K$, we assume 
$$t_*:=\liminf_{i\to\infty}\inf_\Om u_i<t.$$
Then there is a sequence of points $x_i\in\overline{\Om\setminus K}$ with $\lim_{i\rightarrow\infty}u_i(x_i)=t_*$, whose limit belongs to $\p K$.
We define a function $u_i^*$ on $\Si$ by 
\begin{equation}
u_i^*=
\left\{\begin{split}
u_i\ \ \qquad \mathrm{if}\ u_i>(t+t_*)/2\\
(t+t_*)/2\qquad \mathrm{if}\ u_i\le (t+t_*)/2\\
\end{split}\right..
\end{equation}
Hence, for each large $i$ there is an nonempty open set $\Om_i'\subset\Om$ so that $u_i^*=(t+t_*)/2>u_i$ on $\Om'_i$ and $u_i^*=u_i$ on $\Om\setminus\overline{\Om'_i}$.
Combining with \eqref{captKOmu}, we get
\begin{equation}\aligned
\int_{\Om_i\setminus K}\left(\sqrt{1+|Du_i^*|^2}-1\right)&+\int_{\p K}(t-u_i^*)+\int_{\p\Om_i}u_i^*\\
<\int_{\Om_i\setminus K}\left(\sqrt{1+|Du_i|^2}-1\right)&+\int_{\p K}(t-u_i)+\int_{\p\Om_i}u_i=\mathrm{cap}_t(K,\Om_i)
\endaligned
\end{equation}
for the large $i$.
It's a contradiction. This completes the proof.
\end{proof}
As an application, we prove a half-space theorem for minimal hypersurfaces in $\Si\times\R$.
\begin{theorem}\label{half-space}
Let $\Si$ be a $M$-parabolic Riemannian manifold. Then any minimal hypersurface properly immersed in $\Si\times\R^+$, must be a slice $\Si\times\{c\}$ for some constant $c\in\R$.
\end{theorem}
\begin{proof}
Let $S$ be a minimal hypersurface properly immersed in $\Si\times\R^+$. Let us prove Theorem \ref{half-space} by contradiction.
Assume $S$ is not a slice.
By translating $S$ vertically, we can assume that $S\subset\Si\times\R^+$ and there is a constant $\de>0$ such that $S\cap (\Si\times\{\de'\})\neq\emptyset$ for any $\de'\in(0,\de)$.
Hence, there are a point $p\in\Si$, a compact set $K\subset B_1(p)$ with smooth boundary, and a positive constant $\tau\in(0,\de)$ such that $S\cap (K\times[0,\tau])=\emptyset$.
For each integer $i\ge1$, let $\Om_i$ be a sequence of open sets in $B_{i+1}(p)$ with smooth boundary and $\overline{B}_i(p)\subset\Om_i$.
Let $u_i$ be a BV solution on $\Si$ associated with $\mathrm{cap}_\tau(K,\Om_i)$ for each $i$. Let $M_i$ denote the boundary of the subgraph of $u_i$ in $\overline{\Om_i\setminus K}\times\R$, i.e.,
$$M_i=\p\{(x,s)|\, x\in\Si,\ s<u_i(x)\}\cap\left(\overline{\Om_i\setminus K}\times\R\right).$$

From Lemma \ref{C1a-reg}, $M_i$ is smooth to the boundary. Let $H_{M_i}$ denote the mean curvature of $M_i\setminus\p M_i$ w.r.t. the unit normal vector $\mathbf{n}_i$ with $\lan\mathbf{n}_i,E_{n+1}\ran\ge0$. Given $i$, let $Y=-f\mathbf{n}_i$ be a vector field with a $C^1$-function $f\ge0$ on $M_i\setminus(B_i(p)\times\R)$ and supp$f\subset M_i\setminus((B_i(p)\times\R)\cup\p M_i)$. Let $\Phi_t$ denote a one-parameter family of mappings satisfying by $\f{d}{dt}\Phi_t=Y\circ\Phi_t$ and $\Phi_0=identity$. Then $\Phi_t(M_i)\subset\overline{\Om_i\setminus K}\times\R$ for all small $t>0$, and
\begin{equation}\aligned\label{phitMiHMi}
\f{d}{dt}\Big|_{t=0}\mathcal{H}^n(\Phi_t(M_i))=-\int_{M_i}\lan Y,H_{M_i}\mathbf{n}_i\ran=\int_{M_i} fH_{M_i}.
\endaligned\end{equation}
From Theorem \ref{UniPlateau}, $[|M_i|]$ is a minimizing current in $\overline{\Om_i\setminus K}\times\R$.
Hence, \eqref{phitMiHMi} implies $H_{M_i}\ge0$ on $M_i\setminus((B_i(p)\times\R)\cup\p M_i)$ by choosing a suitable $f$ satisfying the above condition.

We claim 
$$S\cap M_i=\emptyset\qquad \mathrm{for\ each}\ i\ge1.$$
Let us prove the claim by contradiction. Assume $S\cap M_i\neq\emptyset$ for some $i$. Let $M_{i,t}=\{(x,s+t)|\, (x,s)\in M_i\}$ for every $t\in\R$. Then there are a constant $t'\le0$ and a point $z\in S\cap M_{i,t'}$ so that $S\cap M_{i,t}=\emptyset$ for all $t<t'$. In other words, $M_{i,t'}$ is on the one side of $S$ in a small neighborhood of $z$.
Noting $M_{i,t'}\setminus(\p(\Om_i\setminus K)\times\R)$ is minimal.
From $S\cap (K\times[0,\tau])=\emptyset$ and the maximum principle of minimal hypersurfaces (see (7.23) in \cite{CM} for instance), we conclude $z=(x,t_x)$ for some $x\in\p\Om_i$ and $t_x>0$. But, this is also impossibe by the maximum principle and $H_{M_i}\ge0$ in a small neighborhood of $z$. Hence, we get the claim.

From Lemma \ref{ConMR1},  $M_i$ converges to $\Si\times\{\tau\}$ as $i\rightarrow\infty$. From the maximum principle again, $S\cap (\Si\times\{\tau\})=\emptyset$.
However, this contradicts to that $\tau\in(0,\de)$ and $S\cap (\Si\times\{\de'\})\neq\emptyset$ for any $\de'\in(0,\de)$. We complete the proof.
\end{proof}
With the help of the maximum principle (see \cite{W} for instance), Theorem \ref{half-space} still holds clearly provided the minimal hypersurface there allows singularities of Hausdorff dimension $<n-1$. 
By following the proof of Theorem \ref{half-space} step by step, we have a slice theorem as follows.
\begin{corollary}\label{half-space*}
Let $\Si$ be a $M$-parabolic Riemannian manifold. Any smooth mean concave domain in $\Si\times[0,\infty)$ must be $\Si\times(c_1,c_2)$ for constants $0\le c_1< c_2\le\infty$.
\end{corollary}
Theorem \ref{NONMeanCurv} implies that the $M$-nonparabolic condition is sharp in Corollary \ref{half-space*}.

\section{Boundary behavior of solutions associated with capacities}

Let $\Om$ be a bounded open set in $\Si$, and $K$ be a compact set in $\Om$. Suppose $\p\,\overline{\Om\setminus K}=\p(\Om\setminus K)$. 
Let $u$ be a BV solution associated with $\mathrm{cap}_t(K,\Om)$,
and $M$ be the boundary of the subgraph of $u$ in $\overline{\Om\setminus K}\times\R$, i.e.,
$$M=\p\{(x,\tau)\in\Si\times\R|\, \tau< u(x)\}\cap\left(\overline{\Om\setminus K}\times\R\right).$$
Let $\na$ denote the Levi-Civita connection on the regular part of $M$ and $\bn$ denote the Levi-Civita connection of $\Si\times\R$ w.r.t. its standard product metric as before.
Let $E_{n+1}$ denote the constant unit vector in $\Si\times\R$ perpendicular to $\Si\times\{0\}$. 
Let $\mathbf{h}$ denote the height function on $\Si\times\R$, i.e.,
\begin{equation}\aligned\label{height}
\mathbf{h}(x,t)=t\qquad \mathrm{for\ each}\ \ (x,t)\in\Si\times\R.
\endaligned\end{equation}
Then $E_{n+1}=\bn \mathbf{h}$.

In general, the gradient $Du$ in Theorem \ref{wExist} may not exist everywhere on the boundary. Instead, we have the following existence on the boundary in the measure sense. 
\begin{lemma}\label{ExistTh}
There is a Borel measure $\mu_{_{\p M}}$ on $\p M$ so that $\mu_{_{\p M}}\ge0$ on $\p K\times\{t\}$, $\mu_{_{\p M}}\le0$ on $\p\Om\times\{0\}$, and
\begin{equation}\aligned\label{Intphibyparts}
\int_{M}\lan\na\phi,\na \mathbf{h}\ran=\int \phi d\mu_{_{\p M}}\qquad\qquad \mathrm{for\ every}\ \phi\in C^1(\Si\times\R).
\endaligned\end{equation}
\end{lemma}
\begin{proof}
There are a sequence of open sets $\Om_i\subset\Om$ and a sequence of compact sets $K_i\supset K$ with $\p\Om_i,\p K_i\in C^\infty$ such that $\Om_i\to\Om$ and $K_i\to K$ both in the Hausdorff sense.
From Corollary \ref{captKKiOmOmi***}, we get
\begin{equation}\aligned\label{captKKi000**}
\mathrm{cap}_t(K,\Om)=\lim_{i\to\infty}\mathrm{cap}_t\left(K_i,\Om_i\right).
\endaligned\end{equation}
Let $u_i$ be a BV solution on $\Si$ associated with $\mathrm{cap}_t(K_i,\Om_i)$,
and $M_i$ be the boundary of the subgraph of $u_i$ in $\overline{\Om_i\setminus K_i}\times\R$.
From Theorem \ref{UniPlateau}, each $M_i$ is an area-minimizing hypersurface in $\overline{\Om_i\setminus K_i}\times\R$, and $M$ is area-minimizing in $\overline{\Om\setminus K}\times\R$. 
Let $\Om'$ be an open set in $\Si$ with $\p\Om\in C^\infty$ and $\Om\subset\subset\Om'$. Denote $M'$ be the subgraph of $u$ in $\overline{\Om'}\times\R$ and $M_i'$ be the subgraph of $u_i$ in $\overline{\Om_i}\times\R$ for each $i$. From Proposition \ref{UNIu}, we can assume $M_i'$ converges to $M'$ in the Hausdorff sense without loss of generality. By Federer-Fleming compactness theorem \cite{FF}, we conclude that the currents $[|M_i'|]$ converge as $i\to\infty$ to $[|M'|]$, which implies
\begin{equation}\aligned\label{HnMcapUMi}
\mathcal{H}^n\left(M\cap U\right)\le\lim_{i\to\infty}\mathcal{H}^n\left(M_i\cap U\right)
\endaligned\end{equation}
for any bounded open set $U\subset\subset \Si\times\R$.
From \eqref{OmOmiKiK} and \eqref{captKKi000**}, we have
\begin{equation}\aligned\label{HnMcapUMi*}
\lim_{i\to\infty}\mathcal{H}^n(M_i)=&\lim_{i\to\infty}\left(\mathrm{cap}_t\left(K_i,\Om_i\right)+\mathcal{H}^n(\Om_i\setminus K_i)\right)\\
=&\mathrm{cap}_t\left(K,\Om\right)+\mathcal{H}^n(\Om\setminus K)=\mathcal{H}^n(M).
\endaligned\end{equation}
Combining \eqref{HnMcapUMi}\eqref{HnMcapUMi*}, we obtain
\begin{equation}\aligned\label{MiconvMU}
\mathcal{H}^n\left(M\cap U\right)=\lim_{i\to\infty}\mathcal{H}^n\left(M_i\cap U\right)
\endaligned\end{equation}
for any bounded open set $U\subset\subset \Si\times\R$ with $\mathcal{H}^n\left(M\cap\p U\right)=0$.

For each $i$, let $\na^{M_i}$ denote the Levi-Civita connection of $M_i$ w.r.t. its induced metric.
We define a sequence of linear operators $L_i$ by
\begin{equation}\aligned
L_i\phi=\int_{M_i}\left\lan\na^{M_i}\phi,\na^{M_i} \mathbf{h}\right\ran\qquad \mathrm{for\ any}\ \phi\in \mathrm{Lip}(\Si\times\R).
\endaligned\end{equation}
Here, $\mathrm{Lip}(\Si\times\R)$ denotes the space containing all Lipschitz functions on $\Si\times\R$.
From Lemma \ref{C1a-reg}, $M_i$ is $C^{1,\a}$ in a small neighborhood of $\p M_i$. Let $\nu^i\in\G(TM_i)$ denote the co-normal vector to $\p M_i$ with $\lan \nu^i,E_{n+1}\ran\ge0$. 
Let $\De_{M_i}$ be the Laplacian on $M_i$. It's clear that $\De_{M_i}\mathbf{h}=0$. By Stokes' formula, we have
\begin{equation}\aligned\label{LiphiS}
L_i\phi=\int_{M_i}\mathrm{div}_{M_i}(\phi \na^{M_i} \mathbf{h})-\int_{M_i}\phi\De_{M_i}\mathbf{h}=\int_{\p M_i}\phi\lan \nu^i,\na^{M_i}\mathbf{h} \ran=\int_{\p M_i}\phi\lan \nu^i,E_{n+1}\ran.
\endaligned\end{equation}
Taking the limit, with the definition of $L_i\phi$ and \eqref{MiconvMU} we get
\begin{equation}\aligned\label{naphilimMi}
\int_{M}\lan \na\phi,\na\mathbf{h}\ran=\lim_{i\to\infty}\int_{\p M_i}\phi\lan \nu^i,E_{n+1}\ran.
\endaligned\end{equation}
Let $\e_i$ be a  Lipschitz function defined by letting $\e_i=1-\f{2\r_{\p K_i}}{d(\p K_i,\p\Om_i)}$ for $\r_{\p K_i}\le\f12d(\p K_i,\p\Om_i)$, and 
$\e_i=-1+\f{2\r_{\p \Om_i}}{d(\p K_i,\p\Om_i)}$ for $\r_{\p \Om_i}\le\f12d(\p K_i,\p\Om_i)$. Here, $d(\p K_i,\p\Om_i)=\inf_{x\in\p K_i,y\in\p\Om_i}d(x,y)$. Noting that $\lan \nu^i,E_{n+1}\ran\ge0$ on $\p K_i\times\{t\}$ and $\lan \nu^i,E_{n+1}\ran\le0$ on $\p\Om_i\times\{0\}$. 
Then combining with \eqref{HnMcapUMi*}\eqref{LiphiS}\eqref{naphilimMi}, we obtain
\begin{equation}\aligned\label{|naphiEn+1|}
\left|\int_{M}\lan \na\phi,\na\mathbf{h}\ran\right|=&\liminf_{i\to\infty}\left|\int_{\p M_i}\phi\lan \nu^i,E_{n+1}\ran\right|\le\sup_{\p M}|\phi|\liminf_{i\to\infty}\left|\int_{\p M_i}\e_i\lan \nu^i,E_{n+1}\ran\right|\\
=&\sup_{\p M}|\phi|\liminf_{i\to\infty}\left|\int_{M_i}\left\lan \na^{M_i}\e_i,\na^{M_i} \mathbf{h}\right\ran\right|\\
\le&\sup_{\p M}|\phi|\liminf_{i\to\infty}\f{2\mathcal{H}^n(M_i)}{d(\p K_i,\p\Om_i)}=\sup_{\p M}|\phi|\f{2\mathcal{H}^n(M)}{d(\p K,\p\Om)}.
\endaligned\end{equation}

Set $\G_+=\p K\times\{t\}$ and $\G_-=\p\Om\times\{0\}$ for convenience. Let $\mathcal{K}$ denote a space containing every function $f\in C^0(\Si\times\R)$ with $\pm f\ge0$ on $\G_\pm$. 
We define a functional $L$ on $\mathcal{K}$ by
\begin{equation}\aligned
Lf=\sup\left\{\int_{M}\lan\na\phi,\na \mathbf{h}\ran\bigg|\ \phi\in\mathrm{Lip}(\Si\times\R),\, 0\le\pm\phi\le\pm f\ \mathrm{on}\ \G_\pm\right\}
\endaligned\end{equation}
for each $f\in\mathcal{K}$.
From \eqref{|naphiEn+1|}, $L$ is bounded.
Let $f_1,f_2\in \mathcal{K}$ with $f_1\equiv f_2\equiv0$ on $\G_-$. By the definition, we get
\begin{equation}\aligned\label{Lf1f2**}
Lf_1+Lf_2\le L(f_1+f_2).
\endaligned\end{equation}
For any $\ep>0$, there is a function $\phi_\ep\in\mathrm{Lip}(\Si\times\R)$ with $0\le\phi_\ep\le f_1+f_2$ on $\G_+$ and $\phi_\ep=0$ on $\G_-$ such that 
$$L(f_1+f_2)\le\ep+\int_{M}\lan\na\phi_\ep,\na \mathbf{h}\ran.$$
By mollifiers, we could further assume $\phi_\ep+\ep'\le f_1+f_2$ on $\G_+$ without loss of generality for some positive $\ep'<\ep$. Let $\phi_{\ep,i}\in\mathrm{Lip}(\Si\times\R)$ with $f_i-\ep'/2\le\phi_{\ep,i}\le f_i$ on $\G_+$ and $\phi_{\ep,i}=0$ on $\G_-$ for $i=1,2$. Then $\phi_{\ep,1}+\phi_{\ep,2}\ge f_1+f_2-\ep'\ge\phi_\ep$ on $\G_+$.
From \eqref{LiphiS}, we have
\begin{equation}\aligned
\int_{M}\lan \na\phi_\ep,\na\mathbf{h}\ran=&\lim_{i\to\infty}\int_{\p M_i}\phi_\ep\lan \nu^i,E_{n+1}\ran\le\lim_{i\to\infty}\int_{\p M_i}(\phi_{\ep,1}+\phi_{\ep,2})\lan \nu^i,E_{n+1}\ran\\
=&\int_{M}\lan \na(\phi_{\ep,1}+\phi_{\ep,2}),E_{n+1}\ran\le Lf_1+Lf_2.
\endaligned\end{equation}
Combining \eqref{Lf1f2**}, we get
\begin{equation}\aligned\label{Lf1f2L++}
L(f_1+f_2)= Lf_1+Lf_2.
\endaligned\end{equation}
Clearly, \eqref{Lf1f2L++} holds for $f_1,f_2\in \mathcal{K}$.
It's not hard to see that $L(cf)=cLf \ge 0$ for any constant $c\ge0$ and $f\in \mathcal{K}$. 
Hence, from Riesz representation theorem for non-negative functionals (see Theorem 4.1 in \cite{S} for instance), there is a Borel measure $\mu_{_{\p M}}$ on $\p M$ so that $\mu_{_{\p M}}\ge0$ on $\p K\times\{t\}$, $\mu_{_{\p M}}\le0$ on $\p\Om\times\{0\}$, and
\begin{equation}\aligned
Lf=\int f d\mu_{_{\p M}}
\endaligned\end{equation}
for each $f\in \mathcal{K}$. 
In particular, \eqref{Intphibyparts} holds by decomposition of the $C^1$-function $\phi$ into the difference of two $C^1$-functions $\phi_1,\phi_2\in\mathcal{K}$.
\end{proof}
We call $\mu_{_{\p M}}$ \emph{the co-normal measure to $\p M$}.
It's clear that $\mu_{_{\p M}}$ is independent of the choice of the BV solutions associated with $\mathrm{cap}_t(K,\Om)$ even the solutions are not unique.
\begin{theorem}\label{RemfiniteHn-1pK}
We further assume that $\p(\Om\setminus K)$ has finite $\mathcal{H}^{n-1}$-measure. Then the co-normal vector to $\p M$  (pointing out $M$) exists a.e., denoted by $\nu$, satisfies $\mu_{_{\p M}}=\left\lan\nu,E_{n+1}\right\ran \mathcal{H}^{n-1}$. Namely,
\begin{equation}\aligned\label{defthu}
\int_{M}\lan\na\phi,\na \mathbf{h}\ran=\int_{\p M}\phi\left\lan\nu,E_{n+1}\right\ran d\mathcal{H}^{n-1}\qquad\qquad \mathrm{for\ every}\ \phi\in C^1(\Si\times\R).
\endaligned\end{equation}
\end{theorem}
\begin{proof}
Let $\e_r$ be a Lipschitz function defined by $\e_r=1-\f1{r}\r_{\p M}$ on $B_{r}(\p M)$ and $\e_r=0$ outside $B_{r}(\p M)$. 
From \eqref{Intphibyparts}, we have
\begin{equation}\aligned\label{phiThtKOmpM}
\int\phi d\mu_{_{\p M}}=&\lim_{r\to0}\int_{M}\lan\na(\phi\e_r),\na \mathbf{h}\ran\\
=&\lim_{r\to0}\int_{M}\lan\phi\na\e_r,\na \mathbf{h}\ran=-\lim_{r\to0}\f1r\int_{M\cap B_r(\p M)}\lan\phi\na\r_{\p M},\na \mathbf{h}\ran
\endaligned\end{equation}
for any $\phi\in C^1(\Si\times\R)$.
Let $p\in\p K\times\{t\}$ be a $C^1$-point of $\p K\times\R$ (the case that $p\in\p \Om\times\{0\}$ is a $C^1$-point of $\p\Om\times\R$ is similar).
Since $M$ is area-minimizing in $\overline{\Om\setminus K}\times\R$, for any given $r_i\to0$, $\f1{r_i}(M,p)$ has a subsequence converging to an area-minimizing hypersurface $M_\infty$ in a closed half-space of $\R^{n+1}$. It's easy to verify that $M_\infty$ is flat (i.e., the boundary of the half-space). 
Let $\nu(p)$ denote the co-normal vector to $\p M_\infty$ (pointing out $M_\infty$).
By approaching $\chi_{_{B_{r_i}(p)}}$ via a sequence of functions $\phi_{i,j}\in C^1(\Si\times\R)$, from \eqref{phiThtKOmpM} with $\phi$ instead by $\phi_{i,j}$ we can obtain
\begin{equation}\aligned
\lim_{i\to\infty}\fint_{B_{r_i}(p)\cap\p M}d\mu_{_{\p M}}=-\lim_{i\to\infty}\fint_{B_{r_i}(p)\cap\p M}\lan\na\r_{\p M},\na \mathbf{h}\ran=\left\lan\nu(p),E_{n+1}\right\ran.
\endaligned\end{equation}
By Lebesgue's differential theorem, we get
\begin{equation}\aligned
\mu_{_{\p M}}=\left\lan\nu,E_{n+1}\right\ran \mathcal{H}^{n-1}\qquad \mathrm{for\ almost\ every\ point\ in}\ \p M.
\endaligned\end{equation}
Hence, $\nu(p)$ does not depend on the choice of $r_i$ for $\mathcal{H}^{n-1}$-a.e. $p\in\p M$. Moreover, the approxiamted tangent plane $T_p(\p M)$ exists for $\mathcal{H}^{n-1}$-a.e. $p\in\p M$.
So $\nu$ is well-defined a.e. on $\p M$ and $\nu(p)\in T_pM$ is the co-normal vector to $\p M$ for $\mathcal{H}^{n-1}$-a.e. $p\in\p M$. This completes the proof.
\end{proof}

\begin{remark}
Suppose $\p K,\p\Om\in C^{1,\a}$ further. From Lemma \ref{C1a-reg}, $M$ is $C^{1,\a}$ in a small neighborhood of $M\cap((\p\Om\cup\p K)\times\R)$. The vector $\nu$ in Theorem \ref{RemfiniteHn-1pK} exists everywhere on $\p M$ and satisfies
\begin{equation}\aligned
\left\lan\nu(p),E_{n+1}\right\ran=\f{\left\lan Du,\mathbf{n}_{\p(\Om\setminus K)}\right\ran}{\sqrt{1+|Du|^2}}(p')\qquad \mathrm{for}\ \mathrm{every}\ p\in\p M,
\endaligned\end{equation}
where $p=(p',s)$ for some $p'\in\p(\Om\setminus K)$, $s\in\R$, and $\mathbf{n}_{\p(\Om\setminus K)}\in\G(T\Si)$ denotes the outward normal vector field to $\p(\Om\setminus K)$. 
\end{remark}

\begin{proposition}
Let the notations as in Lemma \ref{ExistTh}. We have
\begin{equation}\aligned\label{captKOmuDu}
\mathrm{cap}_t(K,\Om)=\int_{\Om\setminus K}\left(\f1{\sqrt{1+|Du|^2}}-1\right)+t\mu_{_{\p M}}(\p K\times\{t\}),
\endaligned
\end{equation}
and 
\begin{equation}\aligned\label{uDunKup}
\f1t\mathrm{cap}_t(K,\Om)\le\mu_{_{\p M}}(\p K\times\{t\})\le\f2t\mathrm{cap}_t(K,\Om).
\endaligned\end{equation}
\end{proposition}
\begin{proof}
If we choose $\phi=1$ in \eqref{Intphibyparts}, then
\begin{equation}\aligned\label{Stokes1}
\mu_{_{\p M}}(\p M)=0.
\endaligned\end{equation}
Moreover, we choose $\phi=\mathbf{h}$ in \eqref{Intphibyparts}, and get
\begin{equation}\aligned\label{Stokesu}
\int_{M}|\na \mathbf{h}|^2=\int_{M\cap\left((\Om\setminus K)\times\R\right)}|\na \mathbf{h}|^2+\int_{M\cap\left((\p K\cup\p\Om)\times\R\right)}|\na \mathbf{h}|^2=\int_{\p M}\mathbf{h}d\mu_{_{\p M}}.
\endaligned\end{equation}
Let $u$ be a BV solution on $\Si$ associated with $\mathrm{cap}_t(K,\Om)$, then
\begin{equation}\aligned\int_{\p K}(t-u)+\int_{\p\Om}u=\int_{M\cap\left((\p K\cup\p\Om)\times\R\right)}|\na \mathbf{h}|^2.
\endaligned
\end{equation}
Now we can rewrite \eqref{captKOmu} to get 
\begin{equation}\aligned
\mathrm{cap}_t(K,\Om)=\int_{\Om\setminus K}\left(\sqrt{1+|Du|^2}-1\right)+\int_{M\cap\left((\p K\cup\p\Om)\times\R\right)}|\na \mathbf{h}|^2.
\endaligned
\end{equation}
Noting $|\na \mathbf{h}|^2=|Du|^2(1+|Du|^2)^{-1}$.
With \eqref{Stokes1}\eqref{Stokesu}, we have
\begin{equation*}\aligned
\mathrm{cap}_t&(K,\Om)=\int_{\Om\setminus K}\left(\sqrt{1+|Du|^2}-1\right)+\int_{\p M}\mathbf{h}d\mu_{_{\p M}}-\int_{M\cap\left((\Om\setminus K)\times\R\right)}|\na \mathbf{h}|^2 \\
=&\int_{\Om\setminus K}\left(\sqrt{1+|Du|^2}-1-\f{|Du|^2}{\sqrt{1+|Du|^2}}\right)+\int_{\p M}\mathbf{h}d\mu_{_{\p M}}\\
=&\int_{\Om\setminus K}\left(\f1{\sqrt{1+|Du|^2}}-1\right)+t\mu_{_{\p M}}(\p K\times\{t\}),
\endaligned
\end{equation*}
which completes the proof of \eqref{captKOmuDu}.
In particular, the above equality implies
\begin{equation}\aligned\label{uDunKlow}
\mu_{_{\p M}}(\p K\times\{t\})\ge\f1t\mathrm{cap}_t(K,\Om).
\endaligned\end{equation}
On the other hand, the Cauchy inequality gives
\begin{equation}\aligned\label{Caychy1+Du2}
\f1{\sqrt{1+|Du|^2}}-1\ge1-\sqrt{1+|Du|^2}.
\endaligned
\end{equation}
Then from \eqref{captKOmuDu} we get
\begin{equation}\aligned
\mathrm{cap}_t(K,\Om)\ge&\int_{\Om\setminus K}\left(1-\sqrt{1+|Du|^2}\right)+t\mu_{_{\p M}}(\p K\times\{t\})\\
\ge&-\mathrm{cap}_t(K,\Om)+t\mu_{_{\p M}}(\p K\times\{t\}).
\endaligned
\end{equation}
With \eqref{uDunKlow}, we can get \eqref{uDunKup}. This completes the proof.
\end{proof}

\begin{proposition}\label{CompThOmOm'}
Let the notations be as in Lemma \ref{ExistTh}, and
$\Om'\supset\Om$ be a bounded open set in $\Si$ with $\p(\Si\setminus\Om')=\p\Om'$.
For each $0<t'\le t$, let $M'$ be the boundary of the subgraph of $u'+(t-t')$ in $\overline{\Om'\setminus K}\times\R$. Then the co-normal measure $\mu_{_{\p M'}}$ to $\p M'$ satisfies $\mu_{_{\p M}}(\cdot,t)\ge\mu_{_{\p M'}}(\cdot,t')$ a.e. on $\p K$.
\end{proposition}
\begin{proof}
There are two sequence of open sets $\Om_i\subset\Om_i'\subset\Om'$ with $\Om_i\subset\Om$, and a sequence of compact sets $K_i\supset K$ with $\p\Om_i,\p\Om_i',\p K_i\in C^\infty$ such that $\Om_i\to\Om$, $\Om_i'\to\Om'$, $K_i\to K$ all in the Hausdorff sense.
For each $i$, let $u_i$ be a BV solution associated with $\mathrm{cap}_{t}(K_i,\Om_i)$ and $u_i'$ be a BV solution associated with $\mathrm{cap}_{t'}(K_i,\Om_i')$. 
From Proposition \ref{KK'OmOm'UNIQ}, we may assume 
$u_i\le u_i'+t-t'$ on $\Si$ up to a choice of constants.

Let $M_i$ be the boundary of the subgraph of $u_i$ in $\overline{\Om\setminus K}\times\R$, and $M_i'$ be the boundary of the subgraph of $u_i'+(t-t')$ in $\overline{\Om\setminus K}\times\R$.
Let $\nu^i\in\G(TM_i)$ and $\nu^{'i}\in\G(TM_i')$ denote the co-normal vector to $\p M_i$ and $\p M_i'$ , respectively, so that $\lan \nu^i,E_{n+1}\ran\ge0$ and $\lan \nu^{'i},E_{n+1}\ran\ge0$. Then $u_i\le u_i'+t-t'$ on $\Si$ gives
\begin{equation}\aligned
\lan \nu^i,E_{n+1}\ran\ge\lan \nu^{'i},E_{n+1}\ran\qquad \mathrm{on}\ \p K\times\{t\}.
\endaligned
\end{equation}
We use the argument in Lemma \ref{ExistTh} and conclude that
\begin{equation}\aligned
\int_{M}\lan\na\phi,\na \mathbf{h}\ran\ge\int_{M'}\lan\na\phi,\na \mathbf{h}\ran
\endaligned\end{equation}
for any $\phi\in\mathrm{Lip}(\Si\times\R)$ with $\phi\ge0$ on $\p K\times\{t\}$ and the support of $\phi$ contained in a sufficiently small neighborhood of $\p K\times\{t\}$.
By the definition of the functional $L$ in Lemma \ref{ExistTh}, we complete the proof.
\end{proof}

\section{Existence of non-compact minimal graphs over manifolds}

Let $\Si$ be a complete noncompact Riemannian manifold and a compact set $K\subset\Si$ with $\p\,\overline{\Si\setminus K}=\p K$. 
Let $\Om_i$ be a sequence of bounded open sets $\Om_1\subset\cdots\subset\Om_i\subset\cdots\subset\Si$ satisfying $\p(\Si\setminus\Om_i)=\p \Om_i$ and $d(p,\p\Om_i)\rightarrow\infty$ for any fixed $p\in\Si$. Given $t>0$, let $u_i$ be a BV solution on $\Si$ associated with $\mathrm{cap}_t(K,\Om_i)$ for each $i$. Let 
$$M_i=\p\{(x,s)|\, x\in\Si,\ s<u_i(x)\}\cap(\overline{\Si\setminus K}\times\R).$$ 
From Proposition \ref{KK'OmOm'UNIQ}, there holds $u_i\le u_{i+1}$ on $\Si$
for each $i$. Hence, there is a BV function $u$ on $\Si$ so that $U\cap\p\{(x,s)|\, x\in\Si,\ s<u_i(x)\}$ converges as $i\to\infty$ to $U\cap\p\{(x,s)|\, x\in\Si,\ s<u(x)\}$ in the Hausdorff sense for any bounded open set $U\subset\Si\times\R$.
We say that $u$ is a \emph{(BV) solution on $\Si$ associated with} $\mathrm{cap}_t(K)$.

Set $M=\p\{(x,s)|\, x\in\Si,\ s<u(x)\}\cap(\overline{\Si\setminus K}\times\R)$.
We have a non-compact version of Lemma \ref{ExistTh} and \eqref{uDunKup}.
\begin{lemma}\label{ThtKomiThtK}
Let $\mu_{_{\p M_i}}$ be the co-normal measure to $\p M_i$ defined in Lemma \ref{ExistTh} for each $i$.
Then the limit $\mu_{_{\p M}}=\lim_{i\to\infty}\mu_{_{\p M_i}}$ exists, and it is a Radon measure on $\p M$ so that
\begin{equation}\aligned\label{Intphibyparts***}
\int_{M}\lan\na\phi,\na \mathbf{h}\ran =\int_{\p M}\phi d\mu_{_{\p M}}\qquad\qquad \mathrm{for\ every}\ \phi\in C^1_c(\Si\times\R).
\endaligned\end{equation}
Moreover, $||\mu_{_{\p M}}||:=\mu_{_{\p M}}(\p M)$ satisfies
\begin{equation}\aligned\label{capKulu}
\f1t\mathrm{cap}_t(K)\le||\mu_{_{\p M}}||\le\f2t\mathrm{cap}_t(K).
\endaligned\end{equation}
\end{lemma}
\begin{proof}
From Lemma \ref{ExistTh}, there holds
\begin{equation}\aligned\label{MiphibfhmupMi}
\int_{M_i}\left\lan\na^{M_i}\phi,\na^{M_i} \mathbf{h}\right\ran =\int_{\p K\times\{t\}}\phi\mu_{_{\p M_i}}
\endaligned\end{equation}
for every $\phi\in C^1_c(\Si\times\R)$ and large $i$,
where $\na^{M_i}$ is the Levi-Civita connection on $M_i$ w.r.t. its induced metric for each $i$. 
From Proposition \ref{CompThOmOm'}, we have $\mu_{_{\p M_{i+1}}}\le\mu_{_{\p M_i}}$ a.e. on $\p M=\p K\times\{t\}$ for each integer $i\ge1$. This monotonicity implies that $\mu_{_{\p M}}=\lim_{i\to\infty}\mu_{_{\p M_i}}$ is a Radon measure on $\p M$.
Hence we have
\begin{equation}\aligned\label{FatouOmiSi}
\int_{\p M}\phi d\mu_{_{\p M}}=\lim_{i\to\infty}\int_{\p M}\phi d\mu_{_{\p M_i}}.
\endaligned\end{equation}
From \eqref{MiconvMU}, we get
\begin{equation}\aligned\label{naphinaMi***}
\int_{M}\left\lan\na\phi,\na \mathbf{h}\right\ran=\lim_{i\to\infty}\int_{M_i}\left\lan\na^{M_i}\phi,\na^{M_i} \mathbf{h}\right\ran.
\endaligned\end{equation}
Combining \eqref{MiphibfhmupMi}-\eqref{naphinaMi***}, we can obtain \eqref{Intphibyparts***}. We replace $M,\Om$ by $M_i,\Om_i$ in \eqref{uDunKup}, and then take limits of $M_i$ and $\Om_i$ as $i\to\infty$, which yields \eqref{capKulu}. This completes the proof.
\end{proof}
\begin{theorem}\label{Uniu<t}
The function $u$ on $\Si$ associated with $\mathrm{cap}_t(K)$ is unique, and $\inf_{\Si\setminus K}u<t$ provided $\mathrm{cap}_{t}(K)>0$.
\end{theorem}
\begin{proof}
Let $u'$ be a BV solution on $\Si$ associated with $\mathrm{cap}_t(K)$. Let $\Om_i$ be a sequence of bounded open sets on $\Si$ so that $\Om_i\subset\Om_{i+1}$, $\p(\Si\setminus\Om_i)=\p \Om_i$ and $d(p,\p\Om_i)\to\infty$. Let $u_i$ be a BV solution on $\Si$ associated with $\mathrm{cap}_t(K,\Om_i)$ for each $i\ge1$.
By the maximum principle and Proposition \ref{KK'OmOm'UNIQ}, we have $u_i\le u'$ on every open $\Om\subset\subset\Si$ for each $i$. Taking the limit implies $u\le u'$ on $\Si$. Similarly, $u'\le u$ on $\Si$. Hence, we conclude that the BV solution on $\Si$ associated with $\mathrm{cap}_t(K)$ is unique.

Suppose $\mathrm{cap}_{t}(K)>0$. 
From \eqref{captKOmuDu}, we get
\begin{equation}\aligned
t\mu_{_{\p M_i}}(\p M)=\int_{\Om_i\setminus K}\left(1-\f1{\sqrt{1+|Du_i|^2}}\right)+\mathrm{cap}_t(K,\Om_i)\ge\mathrm{cap}_t(K,\Om_i).
\endaligned\end{equation}
From Lemma \ref{ThtKomiThtK}, we get
\begin{equation}\aligned
t||\mu_{_{\p M}}||\ge\mathrm{cap}_t(K)>0.
\endaligned\end{equation}
This implies $\inf_{\Si\setminus K}u<t$. We complete the proof.
\end{proof}
\begin{remark}
In general, we actually do not have $\inf_{\Si\setminus K}u=0$ in view of Proposition \ref{captrCOV>0} in the appendix II.
Moreover, we do not have $\lim_{r\to\infty}\mathrm{cap}_t(B_r(p))=\infty$ provided $\mathrm{cap}_t(B_1(p))>0$ (see also Proposition \ref{captrCOV>0}).
\end{remark}

\begin{theorem}\label{NONMeanCurv}
Let $\Si$ be a M-nonparabolic Riemannian manifold. For any $p\in\Si$ and $t,\de>0$, there is a smooth graph $M\subset\Si\times[0,t]$ over $\Si$ so that $M$ has nonnegative mean curvature and $M\setminus(B_\de(p)\times\R)$ is a minimal hypersurface.
\end{theorem}
\begin{proof}
Combining Theorem \ref{Uniu<t} and Lemma \ref{smalluG} in the Appendix I, we can choose $0<t,\de<<1$ so that $D^2\r_p^2\ge\f{\sqrt{5}}2(\de_{ij})_{n\times n}$ on $B_\de(p)$, and the BV solution $u$ on $\Si$ associated with $\mathrm{cap}_t(B_\de(p))$ satisfies $\sup_{\p B_\de(p)}|Du|<1$. By Hopf's lemma, there is a constant $\ep_0\in(0,\f1{4\de})$ so that $\inf_{\p B_\de(p)}|Du|\ge4\ep_0\de$.
Let $\tilde{u}$ be a locally Lipschitz function on $\Si$ defined by 
\begin{eqnarray}
   \tilde{u}= \left\{\begin{array}{cc}
           u   \quad\quad   \ \ \  \ \ \qquad\quad \ \ \mathrm{on}&\ \Si\setminus B_\de(p) \\ [3mm]
           t+\ep_0(\de^2-\r_p^2)      \quad\quad\ \ \   &  \mathrm{on} \quad B_\de(p)
     \end{array}\right..
\end{eqnarray}
Clearly, $|D\tilde{u}|=2\ep_0\r_p\le2\ep_0\de<\f12$ on $B_{\de}(p)$.
Let $\tilde{U}$ denote the subgraph of $\tilde{u}$ on $\Si$, i.e., $\tilde{U}=\{(x,s)\in\Si\times\R|\, s<\tilde{u}(x)\}$. Then $\tilde{U}$ is a (generalized) mean convex domain w.r.t. the unit vector pointing into $\tilde{U}$. From $D^2\tilde{u}=-\ep_0D^2\r_p^2\le-\f{\sqrt{5}\ep_0}2(\de_{ij})_{n\times n}$ on $B_\de(p)$ and $|D\tilde{u}|<\f12$ on $B_{\de}(p)$, $\p\tilde{U}\cap(B_\de(p)\times\R)$ has the mean curvature
\begin{equation}\aligned
\f{-1}{\sqrt{1+|D\tilde{u}|^2}}\left(\De_\Si \tilde{u}-D^2\tilde{u}(D\tilde{u},D\tilde{u})\right)\le\f{1}{\sqrt{1+|D\tilde{u}|^2}}\f{\sqrt{5}\ep_0}2(|D\tilde{u}|^2-n)<(1-n)\ep_0,
\endaligned\end{equation}
where we have used the diagonal argument for the matrix $D^2\tilde{u}$ in the above inequality.

Let $H\not\equiv0$ be a smooth nonnegative function on $\Si$ with compact support in $B_\de(p)$ and $0\le H\le\ep_0$ on $B_\de(p)$. We will also see $H$ as a function on $\Si\times\R$ by letting $H(x,\tau)=H(x)$ for any $(x,\tau)\in\Si\times\R$.
For each open bounded set $\Om\subset\Si$,
let $\mathrm{Lip}_0(\tilde{u},\Om)$ denote the space containing all Lipschitz functions $\phi\le\tilde{u}$ on $\Om$ with $\phi=0$ on $\p\Om$. We consider a functional $\mathbf{F}_{\Om,H}$ on $\mathrm{Lip}_0(\tilde{u},\Om)$ by letting
\begin{equation}\aligned
\mathbf{F}_{\Om,H}(\phi)=\int_\Om\sqrt{1+|D\phi|^2}-\int_\Om\phi H\qquad\qquad\mathrm{for\ each}\ \phi\in\mathrm{Lip}_0(\tilde{u},\Om).
\endaligned\end{equation}
There is a sequence of functions $\phi_i\in\mathrm{Lip}_0(\tilde{u},\Om)$ so that
\begin{equation}\aligned
\inf_{\phi\in\mathrm{Lip}_0(\tilde{u},\Om)}\mathbf{F}_{\Om,H}(\phi)=\lim_{i\to\infty}\mathbf{F}_{\Om,H}(\phi_i).
\endaligned\end{equation}
Since $\mathbf{F}_{\Om,H}(|\phi_i|)\le\mathbf{F}_{\Om,H}(\phi_i)$, we may assume $\phi_i\ge0$.
Analog to the proof of Theorem \ref{wExist},
there is a BV function $w\le\tilde{u}$ on $\Si$ with $w=0$ outside $\overline{\Om}$ such that the subgraph $W$ of $w$ satisfies
\begin{equation}\aligned
\int_{\overline{\Om}\times\R}\left|\overline{\na}\chi_{_W}\right|-\int_\Om wH=\lim_{i\to\infty}\mathbf{F}_{\Om,H}(\phi_i).
\endaligned\end{equation}
From Lemma \ref{wWFOm**}, $W$ minimizes the following functional (w.r.t. the prescribed mean curvature $H$)
\begin{equation}\aligned
\int_{\overline{\Om}\times\R}\left|\overline{\na}\chi_{_F}\right|-\int_{\Si\times\R^+} H\chi_{_F}
\endaligned\end{equation}
for every Caccioppoli set $F$ with $\Om'\times(-\infty,0)\subset F\subset\tilde{U}$ and $F=W$ outside $\Om'\times\R$.
From the construction of $\tilde{U}$, we conclude that $\p W$ is a smooth hypersurface with mean curvature $H$ outside a singular set of dimension $\le n-7$ by Federer's dimension reduction argument and the maximum principle (see \cite{DS} for instance). 

Let $S=\p W\cap(\Om\times\R)$, and $S_*$ be the regular part of $S$. 
Recall (see \cite{Sp} for instance)
\begin{equation}\aligned\label{DeSEn1nwH}
\De_S\lan E_{n+1},\n_S\ran=&\lan E_{n+1},\na^S H\ran-\left(|A_S|^2+\overline{Ric}(\n_S,\n_S)\right)\lan E_{n+1},\n_S\ran\\
=&-\lan E_{n+1},\n_S\ran\lan \n_S,DH\ran-\left(|A_S|^2+\overline{Ric}(\n_S,\n_S)\right)\lan E_{n+1},\n_S\ran,
\endaligned
\end{equation}
where $\De_S$ denotes the Laplacian on $S_*$, $A_S$ denotes the second fundamental form of $S_*$, and $\n_S=(1+|Dw|^2)^{-1/2}(-Dw+E_{n+1})$ denotes the unit normal vector field to $S_*$. 
Now we can follow the proof of Theorem 14.13 in \cite{Gi} with solvability of Dirichlet problem of prescribed mean curvature equation and gradient estimate both on a small ball replaced by Theorem 1.5 and Theorem 1.1 in \cite{Sp}. This infers the Lipschitz continuity of $w$, and hence $w$ is smooth. $H\not\equiv0$ implies the non-flatness of $\p W$.
This completes the proof.
\end{proof}

If a complete manifold has two $M$-nonparabolic ends at least, we can deduce the existence of entire minimal graphs over the manifold as follows.
\begin{theorem}
Let $\Si$ be a complete Riemannian manifold, which has two $M$-nonparabolic ends at least.
Then for any $t>0$ there is a non-flat entire minimal graph $M\subset\Si\times[0,t]$ over $\Si$.
\end{theorem}
\begin{proof}
Let $E_+$ and $E_-$ be two $M$-nonparabolic ends of $\Si$, both of which are disconnected open sets with smooth boundaries. Fix a point $p\in\Si$.
Set $K_i=E_-\setminus B_i(p)$, and $\Om_i=B_i(p)\cup E_-$ for each $i$ with $\p E_-\subset B_i(p)$.
From Theorem \ref{wExist}, let $u_i$ be a BV solution on $\Si$ associated with $\mathrm{cap}_{t}(K_i,\Om_i)$. Let $M_i$ denote the boundary of the subgraph of $u_i$ in $\Si\times\R$. Up to choosing the subsequence, we may assume that $M_i$ converges to a closed set $M$ in $\Si\times[0,t]$. From Lemma \ref{MiniSmooth}, $M$ is an entire minimal graph over $\Si$.

Now, let us prove the non-flatness of $M$. 
Suppose $u_i$ converges to a constant $t_*\in[0,t]$. We choose $R>0$ so that $\p E_+\cup\p E_-\subset B_R(p)$. For $t_*<t$, let us deduce a contradiction.
Let $t_i=\sup_{\p B_R(p)}|t_*-u_i|$.
Let $w_i$ be a BV solution associated with $\mathrm{cap}_{t-t_*}(\overline{B}_R(p),\overline{B}_R(p)\cup (E_-\cap B_i(p)))$.
By Proposition \ref{KK'OmOm'UNIQ}, it follows that
\begin{equation}\aligned\label{wiwi+1uiE-}
w_i-t_i\le t-u_i\le w_i+t_i\qquad\ \ \mathrm{on}\ E_-\setminus B_R(p).
\endaligned
\end{equation}
From Theorem \ref{Uniu<t} and the $M$-nonparabolic $E_-$, the limit of $w_i$ is not constant on $E_-\setminus B_R(p)$.
Then \eqref{wiwi+1uiE-} contradicts to $t_i\to t-t_*$ and that $u_i$ converges to a constant $t_*\in[0,t]$. So we get $t_*=t$. 

On the other hand, $t_*=t$ is impossible by a similar argument as above from Theorem \ref{Uniu<t} again and the non-parabolic $E_+$. 
This completes the proof.
\end{proof}

Now let us study the connection of $M$-parabolicity and the rigidity of minimal graphs over $\Si$ outside compact sets.
\begin{lemma}\label{Rig-M-Parabolic}
Let $\Si$ be a complete Riemannian manifold, and $K\subset\Si$ be a compact set with $\p\overline{\Si\setminus K}=\p K$ and $\mathcal{H}^n(K)>0$. The following two statements are equivalent:
\begin{itemize}
\item[i)] $\Si$ is M-parabolic.
\item[ii)] For any BV function $u:\ \Si\to[0,t]$ with $u\equiv t$ on $K\setminus\p K$, if $\p\{(x,s)\in\Si\times\R|\, s< u(x)\}$ is area-minimizing in $\overline{\Si\setminus K}\times\R$, then $u\equiv t$ on $\Si$.
\end{itemize}
\end{lemma}
\begin{proof}
Let us show $i)\Rightarrow ii)$.
Suppose $\Si$ is $M$-parabolic. Let $u:\ \Si\to[0,t]$ be a bounded BV function on $\Si$ with $u\equiv t$ on $K\setminus\p K$, and $U=\{(x,s)\in\Si\times\R|\, s< u(x)\}$ be the subgraph of $u$ so that $M:=\p U$ is area-minimizing in $\overline{\Si\setminus K}\times\R$. From Lemma \ref{MiniSmooth}, $u$ is smooth on $\Si\setminus K$ with locally bounded gradient.
For any integer $i\ge 1$ with $K\subset B_i(p)$, from Theorem \ref{wExist} let $u_i$ be a BV solution associated with $\mathrm{cap}_{t}(K,B_i(p))$. Let $M_i$ denote the boundary of the subgraph of $u_i$ in $\Si\times\R$.

Suppose $M_i\setminus\overline{U}\neq\emptyset$. Let $W_i$ be an open set enclosed by $M_i$ and $M$ in $(\Si\times\R)\setminus\overline{U}$. 
Let $T$ be an $n$-current defined by $T=[|M|]+\p[|W_i|]$. Then $T$ has multiplicity one with $\p T=0$.  For a point $z=(x,s)\in \p(M_i\setminus\overline{U})$ with $x\in\Si\setminus K$ and $s\in\R$, we can adopt the argument in the proof of Proposition \ref{KK'OmOm'UNIQ} to get a contradition.
Hence, we conclude that $M_i\subset\overline{U}$ for each $i$ with $K\subset B_i(p)$, which implies
$u\ge u_i$ on $\Si$.
From Lemma \ref{ConMR1}, $u_i\rightarrow t$ locally on $\Si$ in the $C^0$-sense. Letting $i\rightarrow\infty$ implies $u\ge t$ on $\Si$. By the definition, we get $u\equiv t$ on $\Si$.
Combining Theorem \ref{Uniu<t}, we complete the proof.
\end{proof}

Recalling that a complete manifold $\Si$ has \emph{nondegenerate boundary at infinity} if there are a constant $\ep>0$ and a closed set $K_\ep\subset\Si$ so that $\mathcal{H}^{n-1}(\p U)\ge\ep$ for any open set $U\supset K_\ep$ with non-empty $\Si\setminus U$.  
\begin{theorem}\label{EntiresolM-para}
Let $\Si$ be a $M$-parabolic manifold with nondegenerate boundary at infinity. For any $p\in\Si$, there is a constant $\de_p>0$ so that for any $\de\in(0,\de_p]$ and $\tau>0$ there is a nonnegative  solution $u\in C^\infty(\Si\setminus B_\de(p))$ to \eqref{u} on $\Si\setminus\overline{B}_\de(p)$ with $u=0$ on $\p B_\de(p)$, $\sup_{B_{2\de}(p)}u=\tau$ and $\limsup_{x\rightarrow\infty}u(x)=\infty$.
\end{theorem}
\begin{proof}
We fix a point $p\in\Si$. 
There is a constant $\de_p\in(0,1/4]$ so that 
\begin{equation}\aligned\label{pOm2dep}
\mathcal{H}^{n-1}(\p\Om)\ge\f45\mathcal{H}^{n-1}(\p B_{4\de_p}(p))\ge \f32\mathcal{H}^{n-1}(\p B_{2\de}(p))
\endaligned
\end{equation}
for every $0<\de\le\de_p$ and every subset $\Om\supset B_{4\de_p}(p)$ with non-empty $\Si\setminus\Om$. Without loss of generality, we assume 
\begin{equation}\aligned\label{2depdep}
\mathcal{H}^{n-1}(\p B_{2\de}(p))\ge \f32\mathcal{H}^{n-1}(\p B_{\de}(p))\qquad \mathrm{for\ each}\ 0<\de\le\de_p.
\endaligned
\end{equation}
From Theorem \ref{wExist}, there is a (BV) solution $u_{i,t}$ on $\Si$ associated with $\mathrm{cap}_t(\overline{B}_{\de}(p),B_i(p))$ for every integer $i\ge1$ and $t>0$. Noting Proposition \ref{UNIu}, we can assume that $u_{i,t}$ is the minimizing solution, which is unique in this sense. This infers $\lim_{t'\to t}u_{i,t'}=u_{i,t}$. In particular, $u_{i,t}$ is not constant.

Let $w_{i,t}=t-u_{i,t}$. We claim
\begin{equation}\aligned\label{CLAwitinf}
\lim_{t\to\infty}\sup_{B_{2\de}(p)}w_{i,t}=\infty.
\endaligned
\end{equation}
If \eqref{CLAwitinf} fails, we assume there is a constant $\La>0$ so that $\sup_{B_{2\de}(p)}w_{i,t}\le\La$ for all $t>0$.
With co-area formula and \eqref{pOm2dep}, for all $t>\La$ we have
\begin{equation}\aligned
\int_{B_i(p)}\sqrt{1+|Dw_{i,t}|^2}\ge&\int_{B_i(p)}|Dw_{i,t}|
\ge\int_\La^t\mathcal{H}^{n-1}\left(\{x\in B_i(p)\setminus B_{2\de}(p)|\,u_{i,t}(x)=s\}\right)ds\\
\ge\int_\La^t \f32\mathcal{H}^{n-1}&(\p B_{2\de}(p)) ds=\f32\mathcal{H}^{n-1}(\p B_{2\de}(p))(t-\La).
\endaligned
\end{equation}
However,
\begin{equation}\aligned\label{captB2deippp}
\int_{B_i(p)}\left(\sqrt{1+|Du_{i,t}|^2}-1\right)\le\mathrm{cap}_t(\overline{B}_{2\de}(p),B_i(p))\le t\mathcal{H}^{n-1}(\p B_{2\de}(p)),
\endaligned
\end{equation}
where the second inequality in \eqref{captB2deippp} can be obtained by taking Lipschitz functions approaching to $\chi_{_{B_{2\de}(p)}}$.
The above two inequalities imply
\begin{equation}\aligned
\f32\mathcal{H}^{n-1}(\p B_{2\de}(p))(t-\La)-\mathcal{H}^n(B_i(p)\setminus B_{2\de}(p))\le t\mathcal{H}^{n-1}(\p B_{2\de}(p)),
\endaligned
\end{equation}
which is impossible for the large $t$. Hence, the claim \eqref{CLAwitinf} is true.

Noting $\lim_{t\to0}w_{i,t}\equiv0$ and $\lim_{t'\to t}u_{i,t'}=u_{i,t}$. Given a constant $\tau>0$, we can choose a constant $t_i>0$ for each $i\ge1$ so that
\begin{equation}\aligned\label{wititau}
\sup_{B_{2\de}(p)}w_{i,t_i}=\tau.
\endaligned
\end{equation}
Let $W_i$ denote the subgraph of $w_{i,t_i}$ in $\Si\times\R$ defined by
$$W_i=\{(x,\tau)\in\Si\times\R|\, \tau< w_{i,t_i}(x),\, x\in\Si\},$$
and $M_i=\p W_i$.
By Federer-Fleming compactness theorem \cite{FF}, up to choosing the subsequence, $[|W_i|]$ converges to a current $[|W_\infty|]$ for some open set $W_\infty\subset\Si\times\R$
such that $\p[|W_\infty|]$ is minimizing in $(\Si\setminus B_\de(p))\times\R$. Moreover, there is a function $w_\infty:\Si\to[0,\infty]$ with $w_\infty\big|_{B_\de(p)}=0$ such that 
$$W_\infty=\{(x,\tau)\in\Si\times\R|\, \tau<w_\infty\}.$$
In fact, $w_\infty$ is said to be a \emph{quasi-solution} on $\Si\setminus B_\de(p)$ (see $\S$16 in \cite{Gi} for more details). Set
$$\mathrm{dom}(w_{\infty})=\{x\in\Si|\, w_\infty(x)<\infty\}.$$
Let $M_\infty=\p W_\infty$. From Lemma \ref{MiniSmooth}, $w_\infty$ is smooth on $\mathrm{dom}(w_{\infty})\setminus \overline{B}_\de(p)$.
Hence, \eqref{wititau} implies
\begin{equation}\aligned\label{wtautau}
\sup_{B_{2\de}(p)}w_\infty=\tau.
\endaligned
\end{equation}

Now we assume $\mathrm{dom}(w_{\infty})\neq\Si$, or else we have completed the proof. Let $S=\p(\mathrm{dom}(w_{\infty}))$. 
Combining \eqref{pOm2dep}\eqref{2depdep}, we fix a suitable large $R>0$ so that
\begin{equation}\aligned\label{Rpn-1VOL*}
\mathcal{H}^{n-1}(S\cap B_R(p))=&\mathcal{H}^{n-1}(\p(\mathrm{dom}(w_{\infty}))\cap B_R(p))\\
\ge&\mathcal{H}^{n-1}(\p B_{2\de}(p))\ge\f32\mathcal{H}^{n-1}(\p B_{\de}(p)).
\endaligned
\end{equation}
For each $T>\tau$ and each $i$ with $t_i>T$, let $\Om_i(T)$ be a connected component in $\{x\in\Si|\, w_{i,t_i}(x)<T\}$ with $\p B_{2\de}(p)\subset\Om_i$. Let $B_{R,\ep}(S)=B_\ep(S)\cap B_R(p)$ denote the $\ep$-tubular neighborhood of $S$ restricted in $B_R(p)$.
Since $M_i$ converges as $i\to\infty$ to $M_\infty$ and $M_{\infty}-tE_{n+1}$ converges as $t\to\infty$ to $S\times\R$ locally in the measure sense, we conclude that $S$ is an area-minimizing hypersurface in $\Si$.
Hence, there is a constant $\La_{R,\ep}$ depending on $R,\ep$ (independent of $i$) so that 
\begin{equation}\aligned\label{BRepSOmiTVOL}
\liminf_{i\to\infty}\int_{B_{R,\ep}(S)\cap\Om_i(T)}\sqrt{1+|Dw_{i,t_i}|^2}\ge(T-\La_{R,\ep})\mathcal{H}^{n-1}(S\cap B_R(p))
\endaligned
\end{equation}
for all large $T>0$. 
Together with \eqref{Rpn-1VOL*}\eqref{BRepSOmiTVOL}, we have
\begin{equation}\aligned\label{OmiTlow}
\liminf_{i\to\infty}&\int_{\Om_i(T)}\left(\sqrt{1+|Dw_{i,t_i}|^2}-1\right)\ge\liminf_{i\to\infty}\int_{B_{R,\ep}(S)\cap\Om_i(T)}\left(\sqrt{1+|Dw_{i,t_i}|^2}-1\right)\\
\ge&\f32(T-\La_{R,\ep})\mathcal{H}^{n-1}(\p B_{\de}(p))-\mathcal{H}^{n}(B_{R,\ep}(S)).
\endaligned
\end{equation}
By the definition of $w_{i,t_i}$, we get
\begin{equation}\aligned\label{OmiTup}
\int_{\Om_i(T)}\left(\sqrt{1+|Dw_{i,t_i}|^2}-1\right)\le T\mathcal{H}^{n-1}(\p B_{\de}(p)).
\endaligned
\end{equation}
However, \eqref{OmiTlow} and \eqref{OmiTup} can not hold simultaneously for the sufficiently large $T>0$.
Hence, we deduce $\mathrm{dom}(w_{\infty})=\Si$. From the $M$-parabolic $\Si$, $\limsup_{x\rightarrow\infty}w_\infty(x)=\infty$ clearly (see Lemma \ref{Rig-M-Parabolic}).
Combining Lemma \ref{smalluG} in the Appendix I, we complete the proof.
\end{proof}
The condition of nondegenerate boundary at infinity is necessary(see Proposition \ref{non-const BV solutions} in Appendix II). After a suitable modification of the above proof, we can study the case of an end admitting nondegenerate boundary at infinity.
As a corollary of Theorem \ref{EntiresolM-para}, we immediately have the following conclusion.
\begin{corollary}
Let $\Si$ be a complete $M$-parabolic Riemannian manifold and a point $p\in\Si$ such that $\mathcal{H}^{n-1}(\p U)\to\infty$ as $d(p,\p U)\to\infty$ for any open set $U\ni p$ with non-empty $\Si\setminus U$. Then for any compact set $K\subset\Si$, there is a BV function $u$ on $\Si$ with $u\equiv0$ on $K\setminus\p K$ and $\limsup_{x\rightarrow\infty}u(x)=\infty$ such that $u$ is a smooth solution to \eqref{u} on $\Si\setminus K$, and $\p\{(x,s)\in\Si\times\R|\, s< u(x)\}$ is area-minimizing in $\overline{\Si\setminus K}\times\R$.
\end{corollary}

\section{The relationship between $M$-parabolic and parabolic manifolds}

Let $\Si$ be an $n$-dimensional complete Riemannian manifold.
Suppose that $\Si$ has \emph{the local volume doubling property}, i.e., there exists a constant $C_D>1$ such that 
\begin{equation}\aligned\label{VD}
\mathcal{H}^n(B_{2r}(p))\le C_D\mathcal{H}^n(B_{r}(p))\qquad\qquad \ \ \mathrm{for\ any}\ p\in\Si\  \mathrm{and\ any}\ r\in(0,1].
\endaligned
\end{equation}
Suppose that $\Si$ satisfies \emph{a local Neumann-Poincar$\mathrm{\acute{e}}$ inequality}, i.e., there exists a constant $C_N\ge1$ such that 
\begin{equation}\aligned\label{NP}
\int_{B_r(p)}|f-\bar{f}_{p,r}|\le C_N r\int_{B_r(p)}|Df|\qquad \mathrm{for\ any}\ p\in\Si\  \mathrm{and\ any}\ r\in(0,1],
\endaligned
\end{equation}
whenever $f$ is a function in the Sobolev space $W^{1,1}(B_r(p))$ and $\bar{f}_{p,r}=\fint_{B_r(p)}f$.

\begin{lemma}\label{Mollify}
Let $\Si$ be a complete $n$-dimensional Riemannian manifold satisfying \eqref{VD}\eqref{NP}. Let $f$ be a function of bounded variation on $\Si$. For each $\la\in(0,2]$, we define a function $f_\la$ on $\Si$ by
$$f_\la(x)=\f1{\int_0^\la\mathcal{H}^n(B_\tau(x))d\tau}\int_0^\la\left(\int_{B_\tau(x)}f\right)d\tau\qquad \mathrm{for\ each}\ x\in\Si.$$
Then there is a constant $C_{D,N}>0$ depending only on $C_D,C_N$ such that 
\begin{equation}\aligned
|Df_\la(x)|\le& C_{D,N}\la^{-1}\sup_{B_\la(x)}|f|,\\
\sqrt{1+|Df_\la|^2(x)}-1\le& C_{D,N}\fint_{B_\la(x)}\left(\sqrt{1+|Df|^2}-1\right).
\endaligned
\end{equation}
\end{lemma}
\begin{proof}
Given a constant $\la\in(0,2]$, let $V_{\la}(x)=\int_0^\la\mathcal{H}^n(B_\tau(x))d\tau$ for any $x\in\Si$. 
Since
\begin{equation}\aligned
\mathcal{H}^{n}(B_\tau(y))\le \mathcal{H}^{n}(B_{\tau+d(x,y)}(x)),
\endaligned
\end{equation}
with \eqref{VD} we have
\begin{equation}\aligned\label{pVtx}
|DV_{\la}|(x)\le\int_0^\la\mathcal{H}^{n-1}(\p B_\tau(x))d\tau=\mathcal{H}^{n}(B_\la(x)),
\endaligned
\end{equation}
and
\begin{equation}\aligned\label{Vtxge}
V_{\la}(x)\ge\int_{\la/2}^\la\mathcal{H}^{n}(B_\tau(x))d\tau\ge\f{\la}{2C_D}\mathcal{H}^{n}(B_\la(x)).
\endaligned
\end{equation}
Set $\bar{f}_{x,\la}=\fint_{B_\la(x)}f$ for any $\la>0$ and $x\in\Si$. 
By the definition of $f_\la$, we have 
\begin{equation}\aligned
f_\la(x)-f_\la(y)=&\f1{V_{\la}(x)}\int_0^\la\left(\int_{B_\tau(x)}(f-\bar{f}_{x,\la})-\int_{B_\tau(y)}(f-\bar{f}_{x,\la})\right)d\tau\\
&+\left(\f1{V_{\la}(x)}-\f1{V_{\la}(y)}\right)\int_0^\la\left(\int_{B_\tau(y)}(f-\bar{f}_{x,\la})\right)d\tau.
\endaligned
\end{equation}
It infers that
\begin{equation}\aligned\label{flax-y}
|f_\la(x)-f_\la(y)|\le&\f1{V_{\la}(x)}\int_0^\la\left(\int_{B_{\tau+d(x,y)}(x)\setminus B_{\tau-d(x,y)}(x)}|f-\bar{f}_{x,\la}|\right)d\tau\\
&+\f{|V_{\la}(x)-V_{\la}(y)|}{V_{\la}(x)V_{\la}(y)}\int_0^\la\left(\int_{B_\tau(y)}|f-\bar{f}_{x,\la}|\right)d\tau.
\endaligned
\end{equation}
Combining \eqref{pVtx}\eqref{flax-y}, we get
\begin{equation}\aligned
|Df_\la|(x)\le&\f2{V_{\la}(x)}\int_0^\la\left(\int_{\p B_\tau(x)}|f-\bar{f}_{x,\la}|\right)d\tau+\f{|DV_{\la}|(x)}{V_{\la}^2(x)}\int_0^\la\left(\int_{B_\tau(x)}|f-\bar{f}_{x,\la}|\right)d\tau\\
\le&\left(\f2{V_{\la}(x)}+\f{\la|DV_{\la}|(x)}{V_{\la}^2(x)}\right)\int_{B_\la(x)}|f-\bar{f}_{x,\la}|.
\endaligned
\end{equation}
With \eqref{NP}\eqref{pVtx}\eqref{Vtxge}, there is a constant $C\ge C_N^4$ (depending only on $C_D,C_N$) such that
\begin{equation}\aligned\label{DflaCDN12fDf}
|Df_\la|(x)\le C^{1/4}\la^{-1}\fint_{B_\la(x)}|f-\bar{f}_{x,\la}|\le C^{1/2}\fint_{B_\la(x)}|Df|.
\endaligned
\end{equation}
Let $\phi(s)=\sqrt{1+s^2}-1$ for each $s\in\R$. Since $\phi$ is convex on $\R$, there holds Jensen's inequality
$$\phi\left(\fint_{B_\la(x)}|Df|\right)\le\fint_{B_\la(x)}\phi(|Df|).$$
With \eqref{DflaCDN12fDf}, we have
\begin{equation}\aligned\label{CDNflafDD}
\sqrt{1+C^{-1}|Df_\la|^2(x)}-1=&\phi\left(C^{-1/2}|Df_\la|(x)\right)\le\phi\left(\fint_{B_\la(x)}|Df|\right)\\
\le&\fint_{B_\la(x)}\phi(|Df|)=\fint_{B_\la(x)}\left(\sqrt{1+|Df|^2}-1\right).
\endaligned
\end{equation}
The convexity of $\phi$ implies that
$$\phi(st)=\int_0^{st}\f{\tau}{\sqrt{1+\tau^2}}d\tau=t^2\int_0^{s}\f{r}{\sqrt{1+t^2r^2}}dr\ge t^2\int_0^{s}\phi'(r)dr=t^2\phi(s)$$ for any $s>0$, $t\in(0,1]$. From \eqref{CDNflafDD} and $C^{-1}\le1$, it follows that
\begin{equation}\aligned
C^{-1}\left(\sqrt{1+|Df_\la|^2(x)}-1\right)\le\sqrt{1+C^{-1}|Df_\la|^2(x)}-1\le \fint_{B_\la(x)}\left(\sqrt{1+|Df|^2}-1\right).
\endaligned
\end{equation}
With \eqref{DflaCDN12fDf}, we complete the proof.
\end{proof}

\begin{theorem}\label{SiVDNPpara}
Let $\Si$ be a complete Riemannian manifold satisfying \eqref{VD}\eqref{NP}. Then $\Si$ is parabolic if and only if $\Si$ is M-parabolic.
\end{theorem}
\begin{proof}
For each $i\ge2$, let $u_i$ be a BV solution on $\Si$ associated with $\mathrm{cap}_1(\overline{B}_1(p),B_{i}(p))$ from Theorem \ref{wExist}. In particular,
\begin{equation}\aligned\label{cap1Bii+1p}
\mathrm{cap}_1(\overline{B}_1(p),B_{i}(p))=\int_{B_{i}(p)\setminus \overline{B}_1(p)}\left(\sqrt{1+|Du_i|^2}-1\right)+\int_{\p B_{1}(p)}(1-u_i)+\int_{\p B_{i}(p)}u_i.
\endaligned
\end{equation}
Let
$$u_i^*(x)=\f1{\int_0^{1/2}\mathcal{H}^n(B_t(x))dt}\int_0^{1/2}\left(\int_{B_t(x)}u_i\right)dt,$$
then spt$\,u^*_i\subset B_{i+1}(p)$ and $u^*_i\equiv1$ on $B_{1/2}(p)$. From Lemma \ref{Mollify}, there is a constant $C_{D,N}\ge1$ depending only on $C_D,C_N$ such that $|Du^*_i|\le 2C_{D,N}$ on $\Si$ and
\begin{equation}\aligned
\sqrt{1+|Du^*_i|^2}(x)-1\le C_{D,N}\fint_{B_{1/2}(x)}\left(\sqrt{1+|Du_i|^2}-1\right)\qquad \mathrm{for\ any}\ x\in\Si.
\endaligned
\end{equation}
With Fubini's theorem, we have
\begin{equation}\aligned
\int_\Si&\left(\sqrt{1+|Du^*_i|^2}-1\right)\le C_{D,N}\int_{x\in\Si}\fint_{B_{1/2}(x)}\left(\sqrt{1+|Du_i|^2}-1\right)\\
=&C_{D,N}\int_{y\in\Si}\int_{x\in B_{1/2}(y)}\f1{\mathcal{H}^n(B_{1/2}(x))}\left(\sqrt{1+|Du_i(y)|^2}-1\right).
\endaligned
\end{equation}
From \eqref{VD}, $C_D\mathcal{H}^n(B_{1/2}(x))\ge \mathcal{H}^n(B_1(x))\ge\mathcal{H}^n(B_{1/2}(y))$. Then 
\begin{equation}\aligned\label{ui*ui}
\int_\Si\left(\sqrt{1+|Du^*_i|^2}-1\right)\le C_DC_{D,N}\int_{\Si}\left(\sqrt{1+|Du_i|^2}-1\right).
\endaligned
\end{equation}

Suppose that $\Si$ is $M$-parabolic. Then $\mathrm{cap}_1(B_1(p),B_{i}(p))\to0$ as $i\to\infty$, and $u_i\rightarrow 1$ locally on $\Si$ in the $C^0$-sense from Lemma \ref{ConMR1}.
Combining \eqref{cap1Bii+1p} and \eqref{ui*ui}, we get
\begin{equation}\aligned\label{Du*i2to0}
\lim_{i\to\infty}\int_\Si\left(\sqrt{1+|Du^*_i|^2}-1\right)=0.
\endaligned
\end{equation}
Since $|Du^*_i|\le 2C_{D,N}$ on $\Si$, it follows that
\begin{equation*}\aligned
\sqrt{1+|Du^*_i|^2}-1\ge\sqrt{1+\f1{\sqrt{2}}C_{D,N}^{-1}|Du^*_i|^2+\f18C_{D,N}^{-2}|Du^*_i|^4}-1=\f1{2\sqrt{2}}C_{D,N}^{-1}|Du^*_i|^2.
\endaligned
\end{equation*}
With \eqref{Du*i2to0}, we get
\begin{equation}\aligned
\lim_{i\to\infty}\int_\Si|Du^*_i|^2=0,
\endaligned
\end{equation}
which implies the parabolicity of $\Si$. Combining \eqref{capt*KOm}, we complete the proof.
\end{proof}

It's well-known that a complete manifold of Ricci curvature uniformly bounded below satisfies \eqref{VD} and \eqref{NP} up to constants $C_D,C_N$ (see Buser \cite{B} or Cheeger-Colding \cite{CC}\cite{CC1}).
As a corollary, we immediately have the following equivalence.
\begin{corollary}\label{RicMparpar}
Let $\Si$ be a complete Riemannian manifold of Ricci curvature uniformly bounded below. Then $\Si$ is parabolic if and only if $\Si$ is M-parabolic.
\end{corollary}

\section{Estimates on capacities and minimal graphic functions}

Let $\Si$ be a complete $n$-dimensional Riemannian manifold, and $p\in\Si$ be a fixed point. 
Assume there are constants $C_D,C_N$ so that $\Si$ satisfies volume doubling property
\begin{equation}\aligned\label{VDg}
\mathcal{H}^n(B_{2r}(x))\le C_D\mathcal{H}^n(B_{r}(x))
\endaligned
\end{equation}
and (1,1)-Poincar\'e inequality
\begin{equation}\aligned\label{NPBrxg}
\int_{B_r(x)}|f-\bar{f}_{x,r}|\le C_N r\int_{B_r(x)}|Df|
\endaligned
\end{equation}
for each $x\in\Si$, $r>0$ and $f\in W^{1,1}(\Si)$.

\begin{proposition}\label{Compcap}
There is a constant $\be\ge2$ depending only on $C_D,C_N$ such that for any $R>r\ge t>0$ and $p\in\Si$ there holds
\begin{equation}\aligned\label{captBrmintr}
\f{t^2}{\be}\mathrm{cap}(\overline{B}_{r/2}(p),B_{R+r/2}(p))\le\mathrm{cap}_t(\overline{B}_r(p),B_R(p))
\le \f{t^2}2\mathrm{cap}(\overline{B}_{r}(p),B_R(p)).
\endaligned\end{equation}
\end{proposition}
\begin{proof}
Since \eqref{capt*KOm}, we only need to prove the first inequality in \eqref{captBrmintr}.
For any BV solution $f$ on $\Si$ with $0\le f\le t$ on $\Si$, we define the function $f_\la$ as in Lemma \ref{Mollify} with $\la=r/2$. Noting that the restriction of $\la\in(0,2]$ in Lemma \ref{Mollify} has been removed here since \eqref{VDg}\eqref{NPBrxg} holds for all $r>0$.
Then $|Df_\la|\le\f C\la\sup_\Si|f|\le Ct/\la$, and
\begin{equation}\aligned
\sqrt{1+|Df_\la|^2}(x)-1\le C\fint_{B_\la(x)}\left(\sqrt{1+|Df|^2}-1\right)
\endaligned
\end{equation}
from Lemma \ref{Mollify}, where $C\ge1$ is a constant depending only on $C_D,C_N$.
From \eqref{VDg}, there holds 
$$C_D\mathcal{H}^n(B_{\la}(x))\ge \mathcal{H}^n(B_r(x))\ge\mathcal{H}^n(B_{\la}(y))$$ 
for any $y\in B_\la(x)$. 
Analog to the proof of \eqref{ui*ui}, we get
\begin{equation}\aligned
\int_\Si\left(\sqrt{1+|Df_\la|^2}-1\right)\le CC_D\int_{\Si}\left(\sqrt{1+|Df|^2}-1\right).
\endaligned
\end{equation}
From $C\ge1$ and $|Df_\la|\le Ct/\la\le2C$ as $\la=r/2\ge t/2$, we get
\begin{equation}\aligned
\sqrt{1+|Df_\la|^2}-1\ge\sqrt{1+\f1{2C}|Df_\la|^2+\f1{16C^2}|Df_\la|^4}-1=\f1{4C}|Df_\la|^2.
\endaligned
\end{equation}
Combining the definition of capacities, we complete the proof.
\end{proof}

Let $G(x,p)$ be the minimal positive Green function on $\Si$ with $\De_\Si G(\cdot,p)=-\de_p$, where $\de_p$ is the delta function. Recalling that (see Proposition 2 in \cite{G0} for instance)
\begin{equation}\aligned\label{GxpBrpsi}
\inf_{x\in\p B_r(p)}G(x,p)\le \f1{\mathrm{cap}(\overline{B}_r(p))}\le\sup_{x\in\p B_r(p)}G(x,p).
\endaligned\end{equation}
Combining Harnack's inequality for $G$ and \eqref{Green} on $\Si$ by Holopainen \cite{H}, we have the following estimate from Proposition \ref{Compcap}.
\begin{corollary}\label{captthEST}
There is a constant $\vartheta\ge1$ depending only on $C_D,C_N$ such that for any ball $B_r(p)\subset\Si$ and $r\ge t>0$ there holds
\begin{equation}\aligned
\vartheta^{-1}\int_r^\infty\f{s ds}{\mathcal{H}^n(B_s(p))}\le \f{t^2}{\mathrm{cap}_t(\overline{B}_r(p))}\le\vartheta\int_r^\infty\f{s ds}{\mathcal{H}^n(B_s(p))}.
\endaligned\end{equation}
\end{corollary}

We further assume that $\Si$ is nonparabolic. 
Given a constant $t>0$ and a compact set $K\subset\Si$ with $\p\overline{\Si\setminus K}=K$, let $u$ be a BV solution on $\Si$ associated with $\mathrm{cap}_t(K)$. 
Let $M$ denote the boundary of the subgraph of $u$ in $\overline{\Si\setminus K}\times\R$,
and $\mu_{_{\p M}}$ be the co-normal measure to $\p M$ defined in Lemma \ref{ThtKomiThtK}.
Now we give a global version of \eqref{captKOmu} and \eqref{captKOmuDu}.
\begin{proposition}\label{utKprop}
$\lim_{r\to\infty}\mathrm{cap}_1(\overline{B}_r(p))=\infty$, $\lim_{x\to\infty}u(x)=0$ and 
\begin{equation}\aligned\label{capumupM**}
\mathrm{cap}_t(K)=&\int_{\Si\setminus K}\left(\sqrt{1+|Du|^2}-1\right)+\int_{\p K}(t-u)\\
=&\int_{\Si\setminus K}\left(\f1{\sqrt{1+|Du|^2}}-1\right)+t||\mu_{_{\p M}}||.
\endaligned\end{equation}
\end{proposition}
\begin{proof}
Let us prove $\lim_{r\to\infty}\mathrm{cap}_1(\overline{B}_r(p))=\infty$ by contradiction. If $\lim_{r\to\infty}\mathrm{cap}_1(\overline{B}_r(p))<\infty$, then we have $\lim_{r\to\infty}\mathrm{cap}(\overline{B}_r(p))<\infty$ by Proposition \ref{Compcap}. By the linearity of $'\mathrm{cap}'$, it follows that
\begin{equation*}\aligned
\mathrm{cap}(\overline{B}_{r}(p))=&\inf_{\mathcal{L}_1(\overline{B}_{r}(p),\Si)}\int_{\Si}|D\phi|^2\le\inf_{\mathcal{L}_{\f12}(\overline{B}_{r}(p),B_R(p))}\int_{\Si}|D\phi|^2+\inf_{\mathcal{L}_{\f12}(\overline{B}_{R}(p),\Si)}\int_{\Si}|D\phi|^2\\
=&\f14\mathrm{cap}(\overline{B}_{r}(p),B_R(p))+\f14\mathrm{cap}(\overline{B}_{R}(p)).
\endaligned\end{equation*}
From $\lim_{R\to\infty}\mathrm{cap}(\overline{B}_R(p))<\infty$ and the above inequality for the large $R>0$, it follows that $\lim_{r\to\infty}\mathrm{cap}(\overline{B}_{r}(p))=0$. It contradicts to the assumption that $\Si$ is nonparabolic. Hence, we have proven $\lim_{r\to\infty}\mathrm{cap}_1(\overline{B}_r(p))=\infty$ for any point $p\in\Si$. 

From Harnack's inequality in Theorem 4.3 of \cite{D}, we clearly have $\lim_{x\to\infty}u(x)=\inf_\Si u$. Suppose $\g:=\inf_\Si u>0$. Let $u_i$ be a BV solution on $\Si$ associated with $\mathrm{cap}_t(K,B_i(p))$ for each $i$ with $K\subset B_i(p)$. Let $K_i=\{x\in B_i(p)|\, u_i(x)\ge\g/2\}$. Then $u_i\to u$ locally uniformly infers $\lim_{i\to\infty} d(p,K_i)=\infty$ and $K\subset K_i$ for large $i$. Hence, given any $r>0$ for large $i$ we have
\begin{equation}\aligned
\mathrm{cap}_t(K,B_i(p))\ge&\int_{B_i(p)\setminus K_i}\left(\sqrt{1+|Du_i|^2}-1\right)+\int_{\p B_i(p)}u_i\\
=&\mathrm{cap}_{\g/2}(K_i,B_i(p))\ge\mathrm{cap}_{\g/2}(\overline{B}_r(p)).
\endaligned\end{equation}
However, this violates $\lim_{r\to\infty}\mathrm{cap}_{1}(\overline{B}_r(p))=\infty$ and \eqref{captTKOm} for the sufficiently large $r>0$.

Now let us prove \eqref{capumupM**}. Let $M_i$ denote the boundary of the subgraph of $u_i$ in $\overline{B_i(p)\setminus K}\times\R$.
Denote $\Om_\tau=\{x\in\Si| u(x)>\tau\}$ and $\Om_{\tau,i}=\{x\in\Si| u_i(x)>\tau\}$ for each $\tau\in[0,t]$.
Then $\Om_{\tau,i}\subset\Om_\tau$ and $\lim_{i\to\infty}\Om_{\tau,i}=\Om_\tau$ from $u_i\le u$ and $u_i\to u$.
Combining \eqref{MiconvMU}, we obtain
\begin{equation}\aligned\label{MOmtauMiOmtaui}
\mathcal{H}^n\left(M\cap(\Om_\tau\times\R)\right)=\lim_{i\to\infty}\mathcal{H}^n\left(M_i\cap(\Om_{\tau}\times\R)\right)=\lim_{i\to\infty}\mathcal{H}^n\left(M_i\cap(\Om_{\tau,i}\times\R)\right).
\endaligned\end{equation}
With $\lim_{x\to\infty}u(x)=0$, the above limit infers
\begin{equation}\aligned\label{inequnauMiui2}
&\int_{\Si\setminus K}\left(\sqrt{1+|Du|^2}-1\right)+\int_{\p K}(t-u)=\lim_{\tau\to0}\left(\mathcal{H}^n\left(M\cap(\Om_\tau\times\R)\right)-\mathcal{H}^n(\Om_\tau)\right)\\
=&\lim_{\tau\to0}\lim_{i\to\infty}\left(\mathcal{H}^n\left(M_i\cap(\Om_{\tau,i}\times\R)\right)-\mathcal{H}^n(\Om_{\tau,i})\right)
\le\lim_{i\to\infty}\mathrm{cap}_{t}(K,B_i(p))=\mathrm{cap}_{t}(K).
\endaligned\end{equation}
By the maximum principle, $\inf_{\p K}u>0$. 
For any $\ep>0$, from \eqref{inequnauMiui2} there is a positive constant $\de<\inf_{\p K}u$ so that
\begin{equation}\aligned
\int_{\{0<u<\de\}}\left(\sqrt{1+|Du|^2}-1\right)<\ep.
\endaligned\end{equation}
We fix a positive constant $\de'<\de/2$. 
Combining \eqref{captTKOm} and $\Om_{\de,i}\subset\Om_\de$, for the sufficiently large $i$ we have
\begin{equation}\aligned\label{BriOmdeiuide'ep}
\int_{B_{r_i}(p)\setminus\Om_{\de,i}}&\left(\sqrt{1+|Du_i|^2}-1\right)+\int_{\p B_{r_i}(p)}u_i\le\mathrm{cap}_{\de}(\overline{\Om_{\de,i}},B_{r_i}(p))\\
\le&\f{\de}{\de-\de'}\mathrm{cap}_{\de-\de'}(\overline{\Om_{\de,i}},B_{r_i}(p))\le\f{\de}{\de-\de'}\mathrm{cap}_{\de-\de'}(\overline{\Om_\de},\Om_{\de'})\\
=&\f{\de}{\de-\de'}\int_{\{\de'<u<\de\}}\left(\sqrt{1+|Du|^2}-1\right)<\f{\de\ep}{\de-\de'}<2\ep.
\endaligned\end{equation}
Combining \eqref{MOmtauMiOmtaui}\eqref{inequnauMiui2} and \eqref{BriOmdeiuide'ep}, we get
\begin{equation}\aligned
\mathrm{cap}_{t}(K)=\int_{\Si\setminus K}\left(\sqrt{1+|Du|^2}-1\right)+\int_{\p K}(t-u).
\endaligned\end{equation}
From \eqref{Caychy1+Du2}, it follows that
\begin{equation}\aligned
\int_{B_{r_i}(p)\setminus\Om_{\de,i}}&\left(1-\f1{\sqrt{1+|Du_i|^2}}\right)\le\int_{B_{r_i}(p)\setminus\Om_{\de,i}}&\left(\sqrt{1+|Du_i|^2}-1\right)<2\ep.
\endaligned\end{equation}
Combining the above inequality, \eqref{captKOmuDu} and Lemma \ref{ThtKomiThtK}, we can get 
\begin{equation}\aligned
\mathrm{cap}_t(K)=\int_{\Si\setminus K}\left(\f1{\sqrt{1+|Du|^2}}-1\right)+t||\mu_{_{\p M}}||.
\endaligned\end{equation}
This completes the proof.
\end{proof}

Let $\na$ be the Levi-Civita connection on $M$ w.r.t. its induced metric.
Recalling that $\mathbf{h}$ denotes the height function defined in \eqref{height}.
\begin{lemma}
\begin{equation}\aligned\label{lanaumupM}
\int_{M\cap\{\mathbf{h}<\la\}}|\na u|^2=\la||\mu_{_{\p M}}||\qquad \mathrm{for\ each}\ \la\in(0,t].
\endaligned\end{equation}
\end{lemma}
\begin{proof}
Let $M_i$ denote the boundary of the subgraph of $u_i$ in $\overline{B_i(p)\setminus K}\times\R$,
and $\na^{M_i}$ be the Levi-Civita connection on $M_i$ w.r.t. its induced metric.
Denote $\Om_{\la,i}=\{x\in\Si| u_i(x)>\la\}$ for each $\la\in[0,t]$.
The strong maximum principle for $u_i$ and $\p\overline{\Si\setminus K}=\p K$ infer that
$\p(\Si\setminus\Om_{\la,i})=\p\Om_{\la,i}$.
By the slicing theorem for the current $[|M_i|]$ (see \cite{S} for instance), $\{x\in\Si|\, u_i(x)=s\}$ is countably $(n-1)$-rectifiable for a.e. $s\in\R$. Since $\mathcal{H}^n(M_i\cap(\Si\times(0,t)))<\infty$, by the co-area formula there is a sequence $\la_i\to\la$ with $\la_i<t$ so that $\p\Om_{\la_i,i}=\left\{x\in\Si|\, u_i(x)=\la_i\right\}$
is countably $(n-1)$-rectifiable with $\mathcal{H}^{n-1}(\p\Om_{\la_i,i})<\infty$ for each $i$.

Denote $M_{*i}=M_i\cap\{\mathbf{h}<\la_i\}$ and $M^*_{i}=M_i\cap\{\la_i<\mathbf{h}\}$ for all $0<\tau\le t$.
If we choose $\phi=1$ in \eqref{Intphibyparts} on $M^*_{i}$, then
\begin{equation}\aligned\label{muMi******}
\mu_{_{\p M^*_{i}}}(\p K\times\{t\})+\mu_{_{\p M^*_{i}}}(\p \Om_{\la_i,i}\times\{\la_i\})=0.
\endaligned\end{equation}
Here, $\mu_{_{\p M_i^*}}$ denotes the co-normal measure to $\p M_i^*$. From Theorem \ref{RemfiniteHn-1pK}, we have $\mu_{_{\p M_i^*}}=\lan\nu_{i}^*,E_{n+1}\ran\mathcal{H}^{n-1}$ on $\p M_i^*$ for the co-normal vector $\nu_{i}^*$ to $\p M_{i}^*$ (pointing out $M_{i}^*$) .
Let
$\mu_{_{\p M_{*i}}}$ be the co-normal measure to $\p M_{*i}$, then $\mu_{_{\p M_{*i}}}=-\lan\nu_{i}^*,E_{n+1}\ran\mathcal{H}^{n-1}$ on $\p \Om_{\la_i,i}\times\{\la_i\}$. 
Moreover, we choose $\phi=\mathbf{h}$ in \eqref{Intphibyparts} for $M_{*i}$, and get
\begin{equation}\aligned\label{M*inaMibfh**}
\int_{M_{*i}}|\na^{M_i} \mathbf{h}|^2=\int_{\p M_{*i}}\mathbf{h}d\mu_{_{\p M_{*i}}}=\la_i\mu_{_{\p M_{*i}}}(\p \Om_{\la_i,i}\times\{\la_i\})=-\la_i\mu_{_{\p M^*_{i}}}(\p \Om_{\la_i,i}\times\{\la_i\}).
\endaligned\end{equation}
Let $\mu_{_{\p M_i}}$ be the co-normal measure to $\p M_i$.
Combining \eqref{muMi******}\eqref{M*inaMibfh**}, we obtain
\begin{equation}\aligned\label{M*iMih2}
\int_{M_{*i}}|\na^{M_i} \mathbf{h}|^2=\la_i\mu_{_{\p M^*_{i}}}(\p K\times\{t\})=\la_i\mu_{_{\p M_{i}}}(\p K\times\{t\}).
\endaligned\end{equation}

Since $u_i\le u$ and $u_i\to u$ on $\p K$, by Lebesgue's dominated convergence theorem, it follows that
\begin{equation}\aligned
\mathcal{H}^n(M\cap(\p K\times\R)\cap\{\mathbf{h}<\la\})=\lim_{i\to\infty}\mathcal{H}^n(M_i\cap(\p K\times\R)\cap\{\mathbf{h}<\la_i\}).
\endaligned\end{equation}
Noting $|\na^{M_i} \mathbf{h}|=1$ a.e. on $M_i\cap(\p K\times\R)$ and $|\na^{M} \mathbf{h}|=1$ a.e. on $M\cap(\p K\times\R)$. Hence,
\begin{equation}\aligned\label{MpKhlai***}
\int_{M\cap(\p K\times\R)\cap\{\mathbf{h}<\la\}}|\na \mathbf{h}|^2=\lim_{i\to\infty}\int_{M_i\cap(\p K\times\R)\cap\{\mathbf{h}<\la_i\}}|\na^{M_i} \mathbf{h}|^2.
\endaligned\end{equation}
For any $\de>0$, let $M^\de=M\setminus(B_\de(K)\times\R)$ and $M^\de_i=M_i\setminus(B_\de(K)\times\R)$ for each $i$.
Let $\tau_0$ be a postive constant and $\tau_0<\min\{\la,\inf_{\p K}u\}$. 
From the estimates of nodal sets for solutions of elliptic equations (see Hardt-Simon \cite{HS1}), for any $\tau\in(0,\tau_0]$ and small $\de>0$ there is a constant $c_{\tau,\de}>0$ so that
\begin{equation}\aligned
\mathcal{H}^{n-1}(M^\de\cap\{\mathbf{h}=\tau\})+\mathcal{H}^{n-1}(M^\de_i\cap\{\mathbf{h}=\tau\})\le c_{\tau,\de}.
\endaligned\end{equation}
Since $u_i\to u$ smoothly on $B_R(p)\setminus B_\de(K)$ for any fixed $R>0$, we conclude that
\begin{equation}\aligned\label{Mdehlai***}
\int_{M^\de\cap\{\tau<\mathbf{h}<\la\}}|\na \mathbf{h}|^2=\lim_{i\to\infty}\int_{M_i^\de\cap\{\tau<\mathbf{h}<\la_i\}}|\na^{M_i} \mathbf{h}|^2.
\endaligned\end{equation} 
The fact that $M_i\to M$ locally in the measure sense (see \eqref{MiconvMU}) implies
\begin{equation}\aligned
\lim_{i\to\infty}\mathcal{H}^n(M_i\cap(B_\de(K)\times\R)\cap\{\mathbf{h}<\la_i\})=\mathcal{H}^n(M\cap(B_\de(K)\times\R)\cap\{\mathbf{h}<\la\}).
\endaligned\end{equation} 
Noting $\mathcal{H}^n(M\cap((B_\de(K)\setminus K)\times\R))\to 0$ as $\de\to0$. With \eqref{MpKhlai***}\eqref{Mdehlai***}, we get
\begin{equation}\aligned\label{Mtauhlanahlai}
\int_{M\cap\{\tau<\mathbf{h}<\la\}}|\na \mathbf{h}|^2=\lim_{i\to\infty}\int_{M_i\cap\{\tau<\mathbf{h}<\la_i\}}|\na^{M_i} \mathbf{h}|^2.
\endaligned\end{equation} 
From $|\na u|^2=|Du|^2(1+|Du|^2)^{-1}$, it follows that
\begin{equation}\aligned\label{Dnaucomp}
\sqrt{1+|Du|^2}-1\le|\na u|^2\sqrt{1+|Du|^2}\le2\left(\sqrt{1+|Du|^2}-1\right).
\endaligned\end{equation}
With \eqref{BriOmdeiuide'ep}, letting $\tau\to0$ in \eqref{Mtauhlanahlai} infers
\begin{equation}\aligned\label{Mtauhlanahlai*}
\int_{M\cap\{\mathbf{h}<\la\}}|\na \mathbf{h}|^2=\lim_{i\to\infty}\int_{M_i\cap\{\mathbf{h}<\la_i\}}|\na^{M_i} \mathbf{h}|^2.
\endaligned\end{equation} 
Combining Lemma \ref{ThtKomiThtK}, \eqref{M*iMih2} and \eqref{Mtauhlanahlai*}, we can deduce \eqref{lanaumupM}.
\end{proof}

Let us prove the asymptotic estimates for solutions associated with $\mathrm{cap}_t(K)$, which are sharp up to constants as follows.
\begin{theorem}\label{BdUpLowercapu*}
Given a compact set $K\subset\Si$ and $t>0$, let $u$ be a BV solution on $\Si$ associated with $\mathrm{cap}_t(K)$.
There is a constant $\Th\ge1$ depending only on $C_D,C_N$ such that for any $r\ge\max\{t,2\mathrm{diam}(K)\}$, $p\in K$ and $x\in \p B_r(p)$ there holds
\begin{equation}\aligned\label{ESTucaptK}
\f{\mathrm{cap}_t(K)}{\Th t}\int_r^\infty\f{s ds}{\mathcal{H}^n(B_s(p))}\le u(x)\le\f{\Th\mathrm{cap}_t(K)}{t}\int_r^\infty\f{s ds}{\mathcal{H}^n(B_s(p))}.
\endaligned\end{equation}
\end{theorem}
\begin{proof}
Given $r\ge\max\{t,2\mathrm{diam}(K)\}$ and $p\in K$. Noting $\inf_\Si u=0$ from Proposition \ref{utKprop}. From \cite{D}, there is a constant $C\ge1$ depending only on $C_D,C_N$ so that 
\begin{equation}\aligned\label{supinfuinfu}
\sup_{\p B_r(p)}u\le C\inf_{\p B_r(p)}u.
\endaligned\end{equation}
Let 
$\la=\inf_{\p B_r(p)}u$, and $\la'=\sup_{\p B_r(p)}u$.
Let $F,F'$ are two closed sets defined by $F=\{x\in\Si|\, u(x)\ge\la\}$ and $F'=\overline{\{x\in\Si|\, u(x)\ge\la'\}}$. 
Then $F'\subset\overline{B}_r(p)\subset F$ clearly.
Combining Proposition \ref{utKprop} and \eqref{Dnaucomp}, we have
\begin{equation}\aligned
\mathrm{cap}_{\la}(\overline{B}_r(p))\le\mathrm{cap}_{\la}(F)=\int_{\Si\setminus F}\left(\sqrt{1+|Du|^2}-1\right)
\le\int_{M\cap\{\mathbf{h}<\la\}}|\na u|^2
\endaligned\end{equation}
and
\begin{equation}\aligned
\int_{M\cap\{\mathbf{h}<\la'\}}|\na u|^2\le&2\int_{\Si\setminus F'}\left(\sqrt{1+|Du|^2}-1\right)+\int_{\p K}\max\{0,\la'-u\}\\
\le&2\mathrm{cap}_{\la'}(F')\le 2\mathrm{cap}_{\la'}(\overline{B}_r(p)).
\endaligned\end{equation}
Since $r\ge t$, combining Corollary \ref{captthEST} and \eqref{lanaumupM} we get
\begin{equation}\aligned
\la\le||\mu_{_{\p M}}||{\vartheta}\int_r^\infty\f{s ds}{\mathcal{H}^n(B_s(p))},
\endaligned\end{equation}
and
\begin{equation}\aligned
||\mu_{_{\p M}}||\int_r^\infty\f{s ds}{\mathcal{H}^n(B_s(p))}\le{\vartheta}\la',
\endaligned\end{equation}
where $\vartheta\ge1$ is the constant depending only on $C_D,C_N$.
Combining \eqref{capKulu} and \eqref{supinfuinfu}, we can get \eqref{ESTucaptK} from the above two inequalities.
\end{proof}

For each $r>0$, we define a function $$\Phi(r)=\int_r^\infty\f{s ds}{\mathcal{H}^n(B_s(p))}.$$
\begin{corollary}
Given a ball $B_r(p)\subset\Si$, let $f$ be a smooth solution to the minimal hypersurface equation \eqref{u} on $\Si\setminus B_r(p)$ with $0<f\le r$ on $\Si\setminus B_r(p)$ and $\lim_{x\to\infty}f=0$. Let $\La=\inf_{\p B_{2r}(p)}f$. Then there is a constant $c\ge1$ depending only on $C_D,C_N$ such that for any $x\in \p B_R(p)$ with $R\ge8r$ there holds
\begin{equation}\aligned\label{EstfPhiRr}
c^{-1}\La\f{\Phi(R)}{\Phi(r)}\le f(x)\le c\La\f{\Phi(R)}{\Phi(r)}.
\endaligned\end{equation}
\end{corollary}
\begin{proof}
We consider the BV solution $u$ on $\Si$ associated with $\mathrm{cap}_\La(\overline{B}_{2r}(p))$.
From Proposition \ref{utKprop}, $\lim_{x\to\infty}u(x)=0$ .
From the maximum principle, $f+\ep>u$ on $\Si\setminus B_{2r}(p)$ for any $\ep>0$, which infers $f\ge u$ on $\Si\setminus B_{2r}(p)$ by letting $\ep\to0$. Combining Corollary \ref{captthEST} and Theorem \ref{BdUpLowercapu*}, we get
\begin{equation}\aligned
f(x)\ge \f{\mathrm{cap}_\La(\overline{B}_{2r}(p))}{\La\Th}\Phi(R)\ge\f{\La\Phi(R)}{\Th\vartheta\Phi(r)}
\endaligned\end{equation}
for any $x\in \p B_R(p)$ with $R\ge8r$.

On the other hand, from \eqref{supinfuinfu} there are a constant $c_*\ge1$ depending only on $C_D,C_N$, and a closed set $K'\subset B_{2r}(p)$ with $\p K'\in C^\infty$ such that $\La':=\sup_{\p K'}f\le c_*\La$. We consider the BV solution $u'$ on $\Si$ associated with $\mathrm{cap}_{\La'}(K')$. Let $M'$ be the boundary of the subgraph of $u'$ in $\overline{\Si\setminus K'}\times\R$, which is smooth from Lemma \ref{C1a-reg}. Denote $M'_t=\{(x,s+t)\in\Si\times\R|\, (x,s)\in M'\}$. By moving $M'_{\La'}$ to $M'_0=M'$ downward along the vector $-E_{n+1}$, we get $u'\ge f$ on $\Si\setminus B_{2r}(p)$ from the maximum principle.
Combining Corollary \ref{captthEST} and Theorem \ref{BdUpLowercapu*}, we can get the right hand of \eqref{EstfPhiRr}. This completes the proof.
\end{proof}

\section{Appendix I: some estimates}

Let $\Si$ be an $n$-dimensional complete noncompact Riemannian manifold.

\begin{lemma}\label{smalluG}
Let $U$ be a bounded open set with smooth boundary $\p U$ and $\G$ be a connected component of $\p U$ so that there is an open set $U_\G$ with $\p U_\G=\G$ and $U_\G\cap U=\emptyset$. For any $\ep>0$, there is a constant $\La_\ep>0$ so that
if $u$ is a BV function on $U\cup\overline{U_\G}$ with $u=0$ on $U_\G$ and $0\le u\le\La_\ep$ on $U$, and $\p\{(x,s)|\, s<u(x),\, x\in U\cup\overline{U_\G}\}$ is area-minimizing in $\overline{U}\times\R$, then $u\in C^\infty(U\cup\G)$ and $\sup_\G|Du|\le\ep$.
\end{lemma}
\begin{proof}
Let $H_\G$ denote the mean curvature of $\G$ w.r.t. the normal vector pointing into $U$. Denote $\La=2\max\{1,-\inf_\G H_\G\}$. Let $\r_\G$ denote the distance function to $\G$ and $B_s^+(\G):=B_s(\G)\cap U$ with $B_s(\G)=\{x\in\Si|\, \r_\G(x)<s\}$ for any $s>0$.
Then there is a constant $\tau>0$ so that $\r_\G$ is smooth on $B_\tau^+(\G)$ and 
\begin{equation}\aligned\label{DeSirGLa}
\De_\Si\r_\G<\La\qquad \mathrm{on}\ \ B_\tau^+(\G).
\endaligned
\end{equation}
Let $\phi$ be a function on $[0,\infty)$ defined by
\begin{equation}\aligned
\phi(t)=\int_0^t\left((1+\ep^{-2})e^{2\La s}-1\right)^{-\f12}ds.
\endaligned
\end{equation}
By a direct calculation, $\phi'(t)=\left((1+\ep^{-2})e^{2\La t}-1\right)^{-\f12}$ and 
\begin{equation}\aligned\label{Laphi'phi''}
\La\phi'+\f{\phi''}{1+(\phi')^2}=0\qquad \mathrm{on}\ (0,\infty).
\endaligned
\end{equation}
Set $\La_\ep=(1+\ep^{-2})^{-\f12}\La^{-1}(1-e^{-\La\tau})$. We have
\begin{equation}\aligned\label{phitaugede}
\phi(\tau)=\int_0^\tau\phi'(s)ds\ge\int_0^\tau(1+\ep^{-2})^{-\f12}e^{-\La s}ds=\La_\ep.
\endaligned
\end{equation}

Set $\phi_\G=\phi\circ\r_\G$. 
Let us prove $0\le u\le\phi_\G$ on $B_\tau^+(\G)$ by contradiction.
If $\sup_{B_\tau^+(\G)}(u-\phi_\G)>0$, then we consider a point $y\in\overline{B_\tau^+(\G)}$ so that 
$$(u-\phi_\G)(y)=\sup_{B_\tau^+(\G)}(u-\phi_\G)>0.$$ 
From \eqref{phitaugede}, $y\in\G\cup B_\tau^+(\G)$. 
Since $\p\{(x,s)|\, s<u(x),\, x\in U\cup\overline{U_\G}\}$ is area-minimizing in $\overline{U}\times\R$, $u$ satisfies the minimal hypersurface equation \eqref{u} on $U$ from Lemma \ref{MiniSmooth}.
Combining Lemma \ref{C1a-reg}  and the smooth $\G$, we conclude that $\p\{(x,s)|\, s<u(x),\, x\in U\cup\overline{U_\G}\}$ is smooth in $\overline{B_\tau^+(\G)}\times\R$. Hence, we must have $y\in B_\tau^+(\G)$. 
From \eqref{DeSirGLa} and \eqref{Laphi'phi''}, we get
\begin{equation}\aligned
\mathrm{div}\left(\f{D\phi_\G}{\sqrt{1+|D\phi_\G|}}\right)=&\mathrm{div}\left(\f{\phi'D\r_\G}{\sqrt{1+(\phi')^2}}\right)=\f{\phi'\De_\Si\r_\G}{\sqrt{1+(\phi')^2}}+\f{\phi''}{(1+(\phi')^2)^{3/2}}\\
<&\La\phi'+\f{\phi''}{1+(\phi')^2}=0.
\endaligned
\end{equation}
Since $Du(y)=D\phi_\G(y)$ and $u$ satisfies \eqref{u}, we deduce a contradiction by the maximum principle. Hence, $0\le u\le\phi_\G$ on $B_\tau^+(\G)$.
With $\phi'(0)=\ep$ and the comparison principle, we get $|Du|\le\ep$ on $\G$.
This completes the proof.
\end{proof}

Let $\Om\subset\Si$ be an open bounded set.
Let $H$ be a smooth nonnegative function on $\Si$ with compact support in $\Om$. We extend the function $H$ to $\Si\times\R$ by letting $H(x,\tau)=H(x)$ for any $(x,\tau)\in\Si\times\R$.
We define a functional $\mathbf{F}_{\Om,H}$ on $BV(\Om)$ by letting
\begin{equation}\aligned
\mathbf{F}_{\Om,H}(\phi)=\int_\Om\sqrt{1+|D\phi|^2}-\int_\Om\phi H\qquad\qquad \mathrm{for\ each}\ \phi\in BV(\Om).
\endaligned\end{equation}
Let $W$ be a measurable set in $\Om\times\R$ so that $\p W\cap(\Om\times\R)$ is $n$-rectifiable with finite measure in $\Om\times[0,T]$ for some $T>0$. 
We define a function $w$ on $\Si$ by
$$w(x)=\int_0^\infty\chi_{_W}(x,\tau)d\tau\qquad\qquad \mathrm{for\ each}\ x\in\Si.$$
Analog to Lemma 14.7 in \cite{Gi}, we have the following version of prescribed mean curvature.
\begin{lemma}\label{wWFOm**}
The function $w\in BV(\Om)$, and
\begin{equation}\aligned
\mathbf{F}_{\Om,H}(w)\le\int_{\Om\times\R}\left|\overline{\na}\chi_{_W}\right|-\int_{\Om\times\R^+} H\chi_{_W}.
\endaligned\end{equation}
\end{lemma}
\begin{proof}
From the proof of Lemma 14.7 in \cite{Gi}, we get
\begin{equation}\aligned\label{AppwW}
\int_\Om\sqrt{1+|Dw|^2}\le \int_{\Om\times\R}\left|\overline{\na}\chi_{_W}\right|.
\endaligned\end{equation}
In particular, the above inequality implies $w\in BV(\Om)$.
On the other hand,
\begin{equation}\aligned\label{AppHwW}
\int_{\Om\times\R^+} H\chi_{_W}=\int_{\Om}\int_0^\infty H(x)\chi_{_W}(x,\tau)d\tau d\mathcal{H}^n(x)=\int_\Om H(x)w(x)d\mathcal{H}^n(x).
\endaligned\end{equation}
Combining \eqref{AppwW}\eqref{AppHwW}, we complete the proof.
\end{proof}

\section{Appendix II:  minimal graphs over rotationally symmetric manifolds}

let $\Si$ be an $n$-dimensional complete noncompact manifold with Riemannian metric $\si_\Si=dr^2+e^{-2f(r)}ds^2$ in the polar coordinate, where $ds^2$ is the Riemannian metric on the standard sphere $\mathbb{S}^{n-1}$. The 'complete' condition of $\Si$ is equivalent to that
\begin{equation}\aligned\label{Appcompleteness}
\int_1^\infty e^{-f(r)}dr=\infty.
\endaligned
\end{equation}
Let $o$ denote the origin of $\Si$.
Let $M$ be a rotationally symmetric minimal graph over $B_{r_2}(o)\setminus B_{r_1}(o)$ with the graphic function $u(r,\omega)$  such that $u(r,\omega)=u(r,\omega')$ for all $\omega,\omega'\in\mathbb{S}^{n-1}$. Denote $u'=\f{\p}{\p r}u$, and $u''=\f{\p^2}{\p r^2}u$.
Then from
$$\De_\Si u=e^{(n-1)f}\f{\p}{\p r}\left(e^{-(n-1)f}\f{\p}{\p r}u\right)=u''-(n-1)f'u',$$
and the minimal hypersurface equation \eqref{u}, we get
\begin{equation}\aligned\label{rotSiu}
0=\De_\Si u-\f{D|Du|^2\cdot Du}{2(1+|Du|^2)}=u''-(n-1)f'u'-\f{(u')^2u''}{1+(u')^2}=\f{u''}{1+(u')^2}-(n-1)f'u'
\endaligned
\end{equation}
on $[r_1,r_2]$.
Since the graph of $w$ is area-minimizing in $(B_{r_2}(o)\setminus \overline{B}_{r_1}(o))\times\R$ (see Lemma 2.1 in \cite{DJX0} for instance), we may assume $u'\ge0$ (or $u'\le0$) on $B_{r_2}(o)\setminus \overline{B}_{r_1}(o)$.
We solve the ODE \eqref{rotSiu} on $(r_1,r_2)$ and get
\begin{equation}\aligned\label{u'fc**}
(n-1)(f-c)=\log\f{u'}{\sqrt{1+(u')^2}}\qquad \mathrm{for\ some\ constant}\ c\ge f(r_2).
\endaligned
\end{equation}
Namely,
\begin{equation}\aligned\label{u'fc-12}
u'=\left(e^{-2(n-1)(f-c)}-1\right)^{-\f12}\qquad \mathrm{for\ some\ constant}\ c\ge f(r_2).
\endaligned
\end{equation}

\begin{proposition}\label{non-const BV solutions}
We further assume $\lim_{R\to\infty}\f{f(R)}R=\infty$. Then $\Si$ is $M$-parabolic. Moreover, for any compact set $K\subset\Si$, there are no non-constant nonnegative BV solutions to \eqref{u} on $\Si\setminus K$ taking value 0 on $K$.
\end{proposition}
\begin{proof}
Let $\{R_i\}_{i\ge1}$ be a sequence with $\lim_{i\to\infty}\f{f(R_i)}{R_i}=\infty$. For each $i$, we can approach $\chi_{_{B_{R_i}(o)}}$ by locally Lipschitz functions, which infers that $\Si$ is $M$-parabolic.
We suppose that there is a non-constant nonnegative BV solution $w$ to \eqref{u} on $\Si\setminus K$ with $w=0$ on $K$.
By the equation \eqref{u}, $w$ is locally bounded. Let $r_0>0$ so that $K\subset B_{r_0}(o)$. Without loss of generality, we assume $R_1\ge r_0$ and $f(R_i)=\sup_{B_{R_i}(o)\setminus B_{r_0}(o)}f$.
Denote $\La=\sup_{\p B_{r_0}(o)}w>0$. Let $a_i>0$ be a sequence so that $\left(e^{a_i}-1\right)^{-\f12}>\sup_{\p B_{R_i}(o)}|Dw|$ for each $i\ge1$.
Let $\phi_i$ be a smooth functoin on $[r_0,R_i)$ defined by
\begin{equation}\aligned
\phi_i(r)=\La+\int_{r_0}^{r}\left(e^{-2(n-1)(f(t)-f(R_i))+a_i}-1\right)^{-\f12}dt\qquad \mathrm{for\ each}\ r\in[r_0,R_i).
\endaligned
\end{equation}
Let $\r$ denote the distance function to $o$. From \eqref{u'fc-12}, $\phi_i\circ\r$ satisfies \eqref{u} on $B_{R_i}(o)\setminus B_{r_0}(o)$. Since $\lim_{r\to R_i}\phi'_{i}(r)=\left(e^{a_i}-1\right)^{-\f12}>\sup_{\p B_{R_i}(o)}|Dw|$, by the maximum principle we get
\begin{equation}\aligned
w\le\phi_i\circ\r\qquad\qquad \mathrm{on}\  B_{R_i}(o)\setminus B_{r_0}(o).
\endaligned
\end{equation}
Letting $i\to\infty$ implies that $w\le\La$ on $\Si\setminus B_{r_0}(0)$. The $M$-parabolic $\Si$ infers that $w$ is constant (see Lemma \ref{Rig-M-Parabolic}). It's a contradiction. We have finished the proof.
\end{proof}

\begin{proposition}\label{non-bded BV solutions}
Suppose $f(r)=\f{\la}{n-1} r$ for each $r\in[r_1,r_2]$ and some constant $\la>0$.
Then there are no classic solutions to \eqref{u} on $B_{r_2}(o)\setminus B_{r_1}(o)$ with boundary data $\psi$ satisfying $\p_\omega\psi=0$ on $\p B_{r_1}\cup\p B_{r_2}$ and $\mathrm{osc}_{_{\p B_{r_1}\cup\p B_{r_2}}}\psi>\f{\pi}{2\la}$.
\end{proposition}
\begin{proof} 
Let $w$ be a classic solution to \eqref{u} on $B_{r_2}(o)\setminus B_{r_1}(o)$ with boundary data $\psi$ satisfying $\p_\omega\psi=0$ on $\p B_{r_1}\cup\p B_{r_2}$. Let $u$ be a rotationally symmetric function defined by
\begin{equation}\aligned
u(r,\omega)=\fint_{\p B_1(o)}w(r,\cdot)
\endaligned
\end{equation}
for any $r\in[r_1,r_2]$ and any $\omega\in\mathbb{S}^{n-1}$.
Let $\p_\r$ denote the radial vector field on $\Si$. Then
\begin{equation}\aligned
|u'|(r,\omega)=\left|\fint_{\p B_1(o)}\lan Dw(r,\cdot),\p_\r\ran\right|\le\fint_{\p B_1(o)}|Dw|(r,\cdot).
\endaligned
\end{equation}
With Jensen's inequality for the convex function $\sqrt{1+s^2}$, it follows that 
\begin{equation}\aligned
\int_{B_{r_2}(o)\setminus B_{r_1}(o)}\sqrt{1+(u')^2}\le&\int_{B_{r_2}(o)\setminus B_{r_1}(o)}\fint_{\p B_1(o)}\sqrt{1+|Dw|^2}\\
=&\int_{B_{r_2}(o)\setminus B_{r_1}(o)}\sqrt{1+|Dw|^2}.
\endaligned
\end{equation}
Since the graph of $w$ is area-minimizing in $(B_{r_2}(o)\setminus \overline{B}_{r_1}(o))\times\R$, the above inequality implies that $w$ is rotationally symmetric, i.e., $u(r,\omega)=w(r,\omega)=w(r,\omega')$ for all $\omega,\omega'\in\mathbb{S}^{n-1}$ and all $r\in[r_1,r_2]$.

Without loss of generality, we assume $u'(r_1)\ge0$. 
We substitute $f(r)=\f{\la}{n-1} r$ into \eqref{u'fc-12} and obtain
\begin{equation}\aligned
u'(r,\omega)=\f1{\sqrt{e^{2\la(T-r)}-1}}\qquad \mathrm{for\ some}\ T\ge r_2.
\endaligned
\end{equation}
Then it follows that
\begin{equation}\aligned
u(r_2,\omega)-u(r_1,\omega)=\int_{r_1}^{r_2}&\f1{\sqrt{e^{2\la(T-r)}-1}}dr=-\f1{\la}\int_{e^{\la(T-r_1)}}^{e^{\la(T-r_2)}}\f{dy}{y\sqrt{y^2-1}}\\
=\f1{\la}\int_{e^{-\la(T-r_1)}}^{e^{-\la(T-r_2)}}\f{ds}{\sqrt{1-s^2}}
=&\f1{\la}\left(\arccos e^{-\la(T-r_1)}-\arccos e^{-\la(T-r_2)}\right)\le\f{\pi}{2\la}
\endaligned
\end{equation}
for any $r\in(r_1,r_2)$.
In other words,
\begin{equation}\aligned\label{noclassicsol}
\mathrm{osc}_{_{B_{r_2}(o)\setminus B_{r_1}(o)}}u\le\f{\pi}{2\la}.
\endaligned
\end{equation}
This completes the proof.
\end{proof}

\begin{proposition}\label{captrCOV>0}
There is a smooth positive function $f$ on $\R^+$ so that $\lim_{r\to\infty}\mathrm{cap}_3(\overline{B}_r(o))$ is a positive finite constant, and the BV solution $u_r$ on $\Si$ associated with $\mathrm{cap}_t(\overline{B}_r(o))$ satisfies $\lim_{x\to\infty}u_r(x)>0$ for $r>>1$.
\end{proposition}
\begin{proof}
For each integer $j\ge1$, let $\mathbf{f}_{2j-1}=1+2^{-j}$ on $(2^j-2^{-j},2^j+2^{-j})$ and $\mathbf{f}_{2j-1}=0$ on others; let $\mathbf{f}_{2j}=2^{j^2}$ on $[2^j+2^{-j},2^{j+1}-2^{-j-1}]$ and $\mathbf{f}_{2j}=0$ on others. Set $\mathbf{f}=\sum_{j\ge1}\mathbf{f}_{j}$.
Denote $s_i=2^i-2^{-i}$ and $r_i=2^{i}+2^{-i}$. For each integer $k>i\ge1$ and $\a\in((1+2^{-k})^{-2},1]$, we define
\begin{equation}\aligned
I_{i,k}(\a)=\int_{s_i}^{r_k}\left(\a\mathbf{f}^{2}(s)-1\right)^{-1/2}ds\quad\mathrm{and}\quad
\mathcal{E}_{k}(\a)=\int_{s_k}^{r_k}\left(\a\mathbf{f}^{2}(s)-1\right)^{-1/2}ds.
\endaligned
\end{equation}
Denote $\a_i=(1+2^{-2i})(1+2^{-i})^{-2}$ for each integer $i\ge1$.
Let $S_{i,k}=I_{i,k}(\a_k)$, then
\begin{equation}\aligned
S_{i,k}=\sum_{j=i}^{k}2^{1-j}\left(\a_k(1+2^{-j})^{2}-1\right)^{-\f12}+\sum_{j=i}^{k-1}(2^{j}-2^{-j-1}-2^{-j})\left(\a_k 2^{2j^2}-1\right)^{-\f12}.
\endaligned
\end{equation}
Noting
\begin{equation}\aligned
\mathcal{E}_{k}(\a_k)=2^{1-k}\left(\a_k(1+2^{-k})^{2}-1\right)^{-\f12}=2,
\endaligned
\end{equation}
and $(1+x^2)(1+x)^{-2}(1+y)^2\ge1+y-x$ for $0<x<<1$ and $x<y<1$ by Taylor's expansion.
For large $i$, we have
\begin{equation}\aligned
0<S_{i,k}-2<&\sum_{j=i}^{k-1}2^{1-j}\left((1+2^{-2k})\left(\f{1+2^{-j}}{1+2^{-k}}\right)^{2}-1\right)^{-\f12}+\sum_{j=i}^{k-1}2^{j}\left(\a_k 2^{2j^2}-1\right)^{-\f12}\\
<&\sum_{j=i}^{k-1}2^{1-j}\left(2^{-j}-2^{-k}\right)^{-\f12}+2\sum_{j=i}^{k-1}2^{j} 2^{-j^2}\\
<&\sum_{j=i}^{\infty}2^{1-j}2^{\f{j+1}2}+\sum_{j=i}^{\infty} 2^{-j-1}=\f4{\sqrt{2}-1}2^{-i/2}+2^{-i}.
\endaligned
\end{equation}
Namely, $\lim_{i\to\infty}\lim_{k\to\infty}S_{i,k}=2$. 
From $\f{\p}{\p\a}I_{i,k}(\a)<0$ and $\lim_{\a\to(1+2^{-k})^{-2}}I_{i,k}(\a)=\infty$, for all $k>i>>1$ there is a constant $\a_{k,i}<\a_k<1$ with $\lim_{k\to\infty}\a_{k,i}=1$ so that $I_{i,k}(\a_{k,i})=3$.

Let $\e$ be a smooth symmetric function in $\R$ with support in $[-1,1]$ and $\int_{\R^n}\e(|x|)dx=1$. Denote $\e_\ep=\ep^{-n}\e(\ep^{-1}\cdot)$ for all $\ep>0$.
We can mollify each function $\mathbf{f}_j$ by convolution $\mathbf{f}_j*\e_{\ep_j}$ for the sufficiently small $\ep_j>0$. From the above argument, there are a smooth function $\mathbf{f}_*=\sum_{j\ge1}\mathbf{f}_j*\e_{\ep_j}$ and a positive constant $\a_{k,i}^*<\a_k<1$ with $\lim_{k\to\infty}\a_{k,i}^*=1$ so that
\begin{equation}\aligned\label{Appurri}
\int_{s_i}^{r_k}\left(\a_{k,i}^*\mathbf{f}_*^{2}(s)-1\right)^{-1/2}ds=3\quad\mathrm{and}\quad
\int_{s_k}^{r_k}\left(\a_{k,i}^*\mathbf{f}_*^{2}(s)-1\right)^{-1/2}ds>2.
\endaligned
\end{equation}
Let $f$ satisfy $\mathbf{f}_*=e^{-(n-1)f}$ on $\Si\setminus B_{s_1}(o)$. Clearly, \eqref{Appcompleteness} holds.
Let $u_{i,k}$ be a function on $B_{r_k}(o)\setminus B_{s_i}(o)$ given by
\begin{equation}\aligned\label{Appuik3sir}
u_{i,k}(r,\omega)=3-\int_{s_i}^r\left(\a_{k,i}^*\mathbf{f}_*^{2}(s)-1\right)^{-1/2}ds\qquad \mathrm{for\ each}\ (r,\omega)\in B_{r_k}(o)\setminus B_{s_i}(o).
\endaligned
\end{equation}
We extend $u_{i,k}$ to $\Si$ by letting $u_{i,k}=3$ on $B_{s_i}(o)$ and $u_{i,k}=0$ on $\Si\setminus B_{r_k}(o)$. Then $u_{i,k}$ is the BV solution on $\Si$ associated with $\mathrm{cap}_3(\overline{B}_{s_i}(o),B_{r_k}(o))$ from \eqref{u'fc-12} (up to a sign). Since $u_{i,k}'=\f{\p}{\p r}u_{i,k}\le0$, with \eqref{Appurri}\eqref{Appuik3sir} we have
\begin{equation}\aligned
u_{i,k}(r,\omega)\ge u_{i,k}(s_k,\omega)>2\qquad\qquad\mathrm{for\ each}\ (r,\omega)\in B_{s_k}(o)\setminus B_{s_i}(o).
\endaligned
\end{equation}
Hence, the BV solution $u_i$ on $\Si$ associated with $\mathrm{cap}_3(\overline{B}_{s_i}(o))$
satisfies $$u_i(r,\omega)=\lim_{k\to\infty}u_{i,k}(r,\omega)\ge2\qquad\qquad\mathrm{for\ each}\
(r,\omega)\in \Si\setminus B_{s_i}(o).$$

On the other hand, from
\begin{equation}\aligned
(1+|u'_{i,k}|^2)^{1/2}=(\a^*_{k,i})^{1/2}\mathbf{f}_*\left(\a^*_{k,i}\mathbf{f}_*^2-1\right)^{-1/2},
\endaligned
\end{equation}
and $\mathbf{f}_*=e^{-(n-1)f}$ on $\Si\setminus B_{s_1}(o)$, we have
\begin{equation}\aligned\label{Appcaprri}
&\mathrm{cap}_3(\overline{B}_{s_i}(o),B_{r_k}(o))=\int_{B_{r_k}(o)\setminus B_{s_i}(o)}\left((1+|D u_{i,k}|^2)^{1/2}-1\right)\\
=& n\omega_n\int_{s_i}^{r_k}\left((\a^*_{k,i})^{1/2}\mathbf{f}_*(s)\left(\a^*_{k,i}\mathbf{f}_*^2(s)-1\right)^{-1/2}-1\right)\mathbf{f}_*(s)ds\\
=& n\omega_n\int_{s_i}^{r_k}\left((\a^*_{k,i})^{1/2}\mathbf{f}_*(s)-\left(\a^*_{k,i}\mathbf{f}_*^{2}(s)-1\right)^{1/2}\right)\mathbf{f}_*(s)\left(\a^*_{k,i}\mathbf{f}_*^{2}(s)-1\right)^{-1/2}ds\\
=&n\omega_n\int_{s_i}^{r_k}\mathbf{f}_*(s)\left((\a^*_{k,i})^{1/2}\mathbf{f}_*(s)+\left(\a^*_{k,i}\mathbf{f}_*^{2}(s)-1\right)^{1/2}\right)^{-1}\left(\a^*_{k,i}\mathbf{f}_*^{2}(s)-1\right)^{-1/2}ds\\
\le&n\omega_n(\a^*_{k,i})^{-1/2}\int_{s_i}^{r_k}\left(\a^*_{k,i}\mathbf{f}_*^{2}(s)-1\right)^{-1/2}ds.
\endaligned
\end{equation}
Combining \eqref{Appurri} and $\lim_{k\to\infty}\a_{k,i}^*=1$, from \eqref{Appcaprri} we have
\begin{equation}\aligned
\mathrm{cap}_3(\overline{B}_{s_i}(o))\le\lim_{k\to\infty} n\omega_n(\a^*_{k,i})^{-1/2}\int_{s_i}^{r_k}\left(\a^*_{k,i}\mathbf{f}_*^{2}(s)-1\right)^{-1/2}ds=3n\omega_n.
\endaligned
\end{equation}
This implies $\lim_{r\to\infty}\mathrm{cap}_3(\overline{B}_r(o))\le3n\omega_n$ from the monotonicity of $\mathrm{cap}_3(\overline{B}_r(o))$ on $r$.
\end{proof}

\bibliographystyle{amsplain}

\end{document}